\documentclass[11pt]{amsart}
\pdfoutput=1

\usepackage{amssymb}
\usepackage{amsthm}
\usepackage{amsfonts}
\usepackage{amsmath}
\usepackage{pmboxdraw}
\usepackage{verbatim} 
\usepackage{graphicx}
\usepackage{color}
\usepackage[colorlinks=true, citecolor=blue, filecolor=black, linkcolor=black, urlcolor=black]{hyperref}
\usepackage{cite}
\usepackage[normalem]{ulem}
\usepackage{enumerate}
\usepackage{tikz}
\usepackage{booktabs}
\usepackage{subcaption}
\usepackage{bbm}
\usepackage{simplewick}
\usepackage{mathtools}
\usepackage{todonotes}
\mathtoolsset{showonlyrefs}

\newcommand{\RN}[1]{%
	\textup{\uppercase\expandafter{\romannumeral#1}}%
}

\def\Phiplus{\Phi^+}
\def\Phiminus{\Phi^-}

\def\bp{{\bar\partial}}

\def\bfs{\boldsymbol}
\def\ms{\medskip}

\def\pa{\partial}
\def\sm{\setminus}

\def\ti{\tilde}
\def\wh{\widehat}

\def\ve{\varepsilon}

\def\id{ \mathrm{id}}
\def\SLE{ \mathrm{SLE}}
\def\Re{ \mathrm{Re}}
\def\Im{ \mathrm{Im}}

\DeclareMathOperator{\Sing}{Sing}

\def\FF{\mathcal{F}}
\def\LL{\mathcal{L}}
\def\OO{\mathcal{O}}

\def\XX{\mathcal{X}}
\def\YY{\mathcal{Y}}

\def\CC{\mathcal{C}}

\def\C{\mathbb{C}}

\def\E{\mathbf{E}}
\def\H{\mathbb{H}}
\def\P{\mathbf{P}}
\def\R{\mathbb{R}}
\def\S{\mathbb{S}}
\def\A{\mathbb{A}}
\def\T{\mathbb{T}}
\def\q{\boldsymbol{q}}

\theoremstyle{plain}
\newtheorem*{thm*}{Theorem}
\newtheorem{thm}{Theorem}[section]
\newtheorem{lem}[thm]{Lemma}
\newtheorem{cor}[thm]{Corollary}
\newtheorem{prop}[thm]{Proposition}

\theoremstyle{definition}
\newtheorem*{eg*}{Example}
\newtheorem*{egs*}{Examples}
\newtheorem*{def*}{Definition}

\theoremstyle{remark}
\newtheorem*{rmk*}{Remark}
\newtheorem*{rmks*}{Remarks}
\newcounter{tmp}

\numberwithin{equation}{section}

\begin{document}
\title[Conformal field theory for annulus SLE]{Conformal field theory for annulus SLE: \\ partition functions and martingale-observables}

\author{Sung-Soo Byun}
\author{Nam-Gyu Kang}
\author{Hee-Joon Tak}

 \address{\noindent Department of Mathematical Sciences, Seoul National University,\newline Seoul, 151-747, Republic of Korea}

\email{sungsoobyun@snu.ac.kr}

\address{School of Mathematics, Korea Institute for Advanced Study, \newline 
	Seoul, 02455, Republic of Korea }

\email{namgyu@kias.re.kr}

\address{Department of Mathematical Sciences, Seoul National University,\newline Seoul, 151-747, Republic of Korea}

\email{tdd502@snu.ac.kr}

\begin{abstract}
We implement a version of conformal field theory in a doubly connected domain to connect it to the theory of annulus SLE of various types, including the standard annulus SLE, the reversible annulus SLE, and the annulus SLE with several force points. 
This implementation considers the statistical fields generated under the OPE multiplication by the Gaussian free field and its central/background charge modifications with a weighted combination of Dirichlet and excursion-reflected boundary conditions. 
We derive the Eguchi-Ooguri version of Ward's equations and Belavin–Polyakov–Zamolodchikov equations for those statistical fields and use them to show that the correlations of fields in the OPE family under the insertion of the one-leg operators are martingale-observables for variants of annulus SLEs.
We find Coulomb gas (Dotsenko-Fateev integral) solutions to the parabolic partial differential equations for partition functions of conformal field theory for the reversible annulus SLE. 

\end{abstract}

\thanks{
The authors were partially supported by Samsung Science and Technology Foundation (SSTF-BA1401-51). Furthermore, Sung-Soo Byun was partially supported by the DFG-NRF through the IRTG 2235. Nam-Gyu Kang was partially supported by a KIAS Individual Grant (MG058103) at Korea Institute for Advanced Study.
}

\maketitle
\tableofcontents

\section{Introduction}

\indent Since Schramm introduced Schramm-Loewner evolution (SLE) in a simply connected domain \cite{MR1776084}, SLE theory has successfully produced rigorous proofs of several remarkable conjectures in statistical physics over the past two decades. 
Examples include the proofs of conformal invariance of scaling limits of some important critical lattice models by employing SLE. 
Such results have been obtained by the fundamental fact that certain discrete observables converge to conformally covariant martingale processes. 
For instance, in \cite{MR1851632}, Smirnov used Cardy's observables to prove conformal invariance of scaling limit of critical site percolation in a hexagonal lattice. 
Schramm and Sheffield made use of a certain bosonic observable to prove that the level lines of discrete Gaussian free field (GFF) with specific height gaps converge to SLE(4), see \cite{MR2486487}. 
Chelkak, Duminil-Copin, Hongler, Kemppaninen, and Smirnov proved the existence of scaling limit of FK Ising model interfaces and its conformal invariance utilizing certain parafermionic observables, see \cite{MR3151886} and references therein. 
In addition, some geometric properties of SLE curves have been obtained through \emph{SLE martingale-observables}. 
Examples include the study of Hausdorff dimensions (Beffara's observables) \cite{MR2435854}, and left passage probabilities (Schramm's observables) \cite{MR1871700} of SLE.

Beyond the SLE theory in a simply connected domain, there have been several developments of the theory in a multiply connected domain \cite{lawler2011defining} that contains richer geometric structures. 
In a doubly connected domain, Zhan has made significant results on basic properties of SLE, see e.g., \cite{zhan2004random,zhan2004stochastic,MR3334276}. 
In particular, in \cite{MR3334276}, he considered the annulus $\SLE(\kappa,\Lambda)$ with two marked points and introduced specific parabolic PDE (known as the \emph{null-vector equation} in the physics literature) of $Z$ that provides a necessary condition for which the associated SLE trace is reversible, cf. \cite{MR3548530,MR2435856}. 
Here $\Lambda := \kappa (\log Z)'$ is the drift function of the driving process of SLE, where the function $Z$ is called \emph{annulus SLE partition function}. 
We remark that such an equation for some special cases with $\kappa=3,6$ also appeared in the literature \cite{dubedat2004critical,MR3602850} in connection with the study of discrete lattice models. 

In the physics literature, in particular, in the context of conformal field theory (CFT) \cite{MR1424041}, it is well known that under the insertion of the ratio of the \emph{one-leg operator} $\Psi$ to its correlation function $\E\Psi$, the correlators in a particular class of fields are martingale-observables for SLEs. 
In a doubly connected domain, Hagendorf, Bernard, and Bauer used a certain bosonic observable to propose the partition functions for annulus SLE with $\kappa=4$ associated with GFFs, see \cite{MR2651436}. 
Hagendorf presented physical arguments on producing martingale-observables for the Dirichlet boundary case in \cite{1742-5468-2009-02-P02033}.

In this paper, we use the framework built in \cite{MR3052311,KM17,AKM1,KM} to construct a version of CFT in a doubly connected domain with a plan to serve as a solid cornerstone to develop a theory in a general multiply connected domain. 
The scope of this paper includes Coulomb gas formalism and derivations of some important equations in CFT, such as the Eguchi-Ooguri version of Ward's equations and Belavin-Polyakov-Zamolodchikov (BPZ) equations. 

The simplest fields we consider in this paper form a one-parameter family of GFFs with a \emph{weighted combination} of Dirichlet and excursion reflected (ER) boundary conditions, see Figure~\ref{Fig_DGFF}. 
In the analysis literature, the Green's function with ER boundary condition is often called the \emph{modified Green's function} \cite{koebe1916abhandlungen}.
In probability theory, the associated stochastic process, ER Brownian motion, was introduced by Lawler and Drenning \cite{MR2992562,lawler2006laplacian}. 
As a characteristic feature of the ER Green's function, its normal derivatives on the inner boundary component have a vanishing mean, see e.g., \cite{MR2241034,MR2334725}. 
By combining this property with zero Dirichlet boundary condition under the (rational) weight, we introduce a one-parameter family of the Green's functions interpolating Dirichlet Green's function with ER Green's function and consider the associated GFFs. 
	\begin{figure}[h!]
	\begin{center}
		\begin{subfigure}[h]{0.48\textwidth}
			\centering
			\includegraphics[width=0.8\textwidth,height=0.8\textwidth]{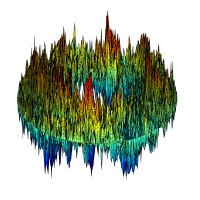}
			\caption{ER}
		\end{subfigure} 
		\begin{subfigure}[h]{0.48\textwidth}
		 	\centering
		 \includegraphics[width=0.8\textwidth,height=0.8\textwidth]{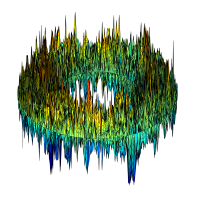}
			\caption{Dirichlet}
		\end{subfigure}
	\end{center}
	\caption{Discrete GFFs with ER and Dirichlet boundary conditions}\label{Fig_DGFF}
\end{figure}

One of the primary motivations for introducing such general boundary conditions is that this interpolation gives rise to infinitely many linearly independent solutions to BPZ equation. 
Our observation affirms the prediction that contrary to a simply connected domain where BPZ equation has a finite-dimensional solution space \cite{MR3294954,MR3294955,MR3296159,MR3296160}, the solution space is no longer finite-dimensional in a doubly connected domain. 
We also refer to \cite{MR4046010,MR4071342,kytola2016pure} for the analysis of solutions to various equations arising from CFT with a concrete understanding of the underlying algebraic structure.

The statistical fields satisfying BPZ equations, thus producing SLE martingale-observables form the OPE family $\FF_{\bfs \beta} \equiv \mathcal{F}_{(b),\bfs \beta}$ generated by central/background charge modifications $\Phi_{\bfs \beta} \equiv \Phi_{(b),\bfs \beta}$ of the GFFs under OPE multiplication. 
Here, $b$ is a real parameter related to the central charge $c$ as $c=1-12b^2$ and the SLE parameter $\kappa$ as $b = \sqrt{\kappa/8}-\sqrt{2/\kappa}$. 
We require that for a divisor $\bfs\beta = \sum \beta_j\cdot q_j$ called a background charge, the total sum of the charges should vanish, i.e., $\sum \beta_j=0$. 
On a complex torus (of genus one), this requirement is known as the \emph{neutrality condition} ($\mathrm{NC}_0$) in the physics literature, see e.g., \cite{MR1424041}. 
To derive BPZ equations, we need two main ingredients, the first is \emph{Ward's equation} and the second is the construction of the one-leg operator $\Psi$ satisfying \emph{level two degeneracy equation}.

Whereas Ward's equations in a simply connected domain describe the Virasoro fields in terms of Lie derivative operators within correlations, the statement of those in a doubly connected domain contains the derivative of correlation functions with respect to the modular (Teichm\"uller) parameter as well. 
These forms of Ward's equations are well known in the physics literature, e.g., Ward's equations in a doubly connected domain \cite{MR2651436} and Eguchi-Ooguri's version of Ward's equations on a complex torus of genus one \cite{MR873031}. 
We derive them not from the path integral formalism but the functional equations of classical special functions such as pseudo-addition theorem of Weierstrass zeta function. 
See \cite{KM17} for a mathematical derivation of Eguchi-Ooguri's version of Ward's equations.

For given another divisor $\bfs{\tau}=\sum_j \tau_j \cdot \xi_j$ satisfying $a+\int \bfs \tau=0$, we construct the one-leg operator $\Psi_{\bfs\beta}\equiv\Psi_{\bfs\beta}(p,\bfs \xi)$ from the GFF as the modified multi-vertex field (or the OPE exponential) with the charge $a=\sqrt{2/\kappa}$ at given marked point $p$. 
With this specific choice of the charge at $p$, the one-leg operator $\Psi_{\bfs\beta}$ satisfies the desired level two degeneracy equation. 
The correlation function $\E \Psi_{\bfs\beta}$ of the one-leg operator (boundary condition changing operator) is called the \emph{partition function} in CFT (up to a multiplicative constant depending only on a domain), see e.g., \cite[Section 11.3]{MR1424041}. The partition functions have different descriptions in other contexts, such as random matrix theory, statistical mechanics, and SLE theory. For example, Dub\'{e}dat introduced the partition functions of SLE and the partition functions of the free field with associated boundary data and then showed that these two definitions coincide, see \cite[Theorem~5.3]{MR2525778}. Unless stated otherwise, the partition function means the correlation function of the one-leg operator throughout the paper. 
In some typical examples, we relate the partition function in CFT with the partition of the associated lattice model.

We first study the annulus $\SLE(\kappa,\Lambda)$ with $\Lambda \equiv \Lambda_{\bfs\beta}(p,\bfs \xi) = \kappa \,\pa_\xi|_{\xi=p} \log \E\,\Psi_{\bfs\beta}(\xi,\bfs \xi)$, see \cite{KM} for this type of SLE in a simply connected domain from a viewpoint of Radon-Nikodym derivative (or coordinate change \cite{MR2188260}) between the laws of SLEs. 
Here, the marked boundary points $p, \bfs{q}, \bfs{\xi}$ play different roles: the point $p$ corresponds to the starting point of $\SLE(\kappa,\Lambda)$, whereas $\bfs \xi$ (resp., $\bfs q$) are force points having non-trivial (resp., trivial) conformal dimensions. 
We consider a collection of all correlation functions of fields in the OPE family $\FF_{\bfs \beta}$ under the insertion of the properly normalized one-leg operator. 
Then employing BPZ-Cardy equations, we show that this collection forms martingale-observables for $\SLE(\kappa,\Lambda)$. 
In the case of Dirichlet boundary condition, this statement generalizes a result of Izyurov and Kyt\"{o}l\"{a} in \cite{MR3010393}, where they showed that 1-and 2-point functions of GFF with piecewise Dirichlet boundary condition are martingale-observables for associated $\SLE(4,\Lambda)$. 
Here, the neutrality condition coincides with the requirement that jumps of the 1-point function on the boundary should add up to zero. 

However, the correlation function $Z_{\bfs\beta}:=\E\,\Psi_{\bfs\beta}$ of this type of $\SLE(\kappa,\Lambda)$ with one force point does not satisfies the null-vector equation (see \eqref{Lawler Zhan null} in Subsection~\ref{Main result sub}) for the reversible annulus SLE partition function (introduced in \cite{MR3334276}) unless $\kappa=4$. 
To find explicit solutions to the null-vector equation of this type for generic $\kappa>0$ and construct a class of $\SLE(\kappa,\Lambda)$ martingale-observables associated with these solutions, we apply the \emph{method of screening} (also known as the \emph{Coulomb gas} or \emph{Dotsenko-Fateev} integrals). 
If $\kappa > 4$, we can deform a Pochhammer contour in this methodology so that the solution can be represented in terms of the Euler type integral. 
In the case $\kappa \le 4$, we perform proper analysis on an analytic continuation (with respect to $\kappa$) of the Pochhammer contour integral to obtain a non-trivial real-valued solution. 

Beyond the method of screening, we present further constructions of SLE martingale-observables, which include the implementations of non-atomic background charges and the use of periodization (or chiral bosonization) of conformal fields.
These methods allow us to construct martingale-observables for continuum interface curves in statistical physics such as discrete GFF, loop-erased random walk, and critical Ising interfaces. 
Moreover, we derive some geometric properties such as hitting/left passage probabilities of certain annulus SLEs.

\section{Main results}

\begingroup
\setcounter{tmp}{\value{thm}}
\setcounter{thm}{0} 
\renewcommand\thethm{\Alph{thm}}

\subsection{Notation}

Throughout this paper, we use the following symbols and notations: let $\T= \{ z \in \C : |z|=1 \}$; for $r>0$ let $\mathbb A_r = \{z \in \C: e^{-r} < |z| <1 \}$, $\T_r= \{ z \in \C : |z|=e^{-r} \}$, $\S_r = \{z \in \C: 0 <\Im\,z < r \}$, and $\R_r= \{ z\in \C: \Im\,z=r \}$. 
We denote by $\mathcal C_r$ the cylinder $\S_r / \langle z \mapsto z+2 \pi \rangle$, and by $D_r$ a doubly connected domain with modulus $r$. 
A subset $K$ of a doubly connected domain $D$ is called a hull in $D$ if $D\sm K$ is a doubly connected domain and $K$ has positive distance from one boundary component of $D$. 
For any hull $K$ of $D$, let us define $\text{cap}_D(K) := \text{mod}(D)- \text{mod} (D \sm K)$ as the capacity of $K$ in $D$, where $\text{mod}(\cdot)$ denotes the modulus of a doubly connected domain. 
Following \cite{zhan2004stochastic,MR3334276}, we write $\Theta(r,z):=\theta (\frac{z}{2\pi},\frac{ir}{\pi} )$, where $\theta$ is a Jacobi theta function, see e.g., \cite[p.63]{MR808396}. 
In Subsection~\ref{Special_sub} we recall the definition of $\Theta$ and its basic properties. 
It is convenient to introduce 
$$
\Theta_{\chi}(r,z):=\Theta(r,z)\,\exp\Big( \frac{z^2}{4(r+\chi)} \Big) , \qquad \chi \in [0,\infty]
$$ 
to describe a one-parameter family of GFFs with a \emph{weighted combination} of Dirichlet and excursion reflected (ER) boundary conditions.
We write $\zeta_r$ for the Weierstrass zeta function with basic periods $(2\pi,2ir)$. 
In the sequel we sometimes omit the modular parameter $r$ and write for instance $\Theta_\chi(\cdot)\equiv\Theta_\chi(r,\cdot)$.

\subsection{Basic setup}
Let $\Phi_{(0)}$ be the GFF in $D_r$ with a weighted combination of Dirichlet and ER boundary conditions, see Subsection~\ref{Cor ftns_Sub} for details. Its 2-point function is given by
$$
\E\,\Phi_{(0)}(\zeta)\Phi_{(0)}(z) = 2\log\Big| \dfrac{\Theta_\chi(r,\zeta-\bar{z})}{\Theta_\chi(r,\zeta-z)}\Big|
$$
in the $\mathcal{C}_r$-uniformization.
For a real parameter $b$, we define the \emph{central charge modification} $\Phi_{(b)}$ of GFF by adding a non-random pre-pre-Schwarzian (PPS) form as follows:
\begin{equation*}
	\Phi_{(b)}:=\Phi_{(0)}-2b\arg w',
\end{equation*}
where $w$ is a conformal map from $D_r$ onto $\mathcal C_r$. 
See Subsection~\ref{Conformal Subsec} for the definition of PPS forms. 
We remark that $\arg w^{\prime}$ does not depend on the choice of a conformal map.

Given marked points $\q=\{ q_k\}$, we call a divisor $\bfs{\beta}= \sum_k \beta_k \cdot q_k$ a \emph{background charge} if it satisfies the neutrality condition $\int \bfs{\beta}=0$. 
We define the \emph{background charge modification} $\Phi_{ \bfs{\beta} } \equiv \Phi_{ \bfs{\beta},(b) }$ associated with $\bfs\beta$ as 
\begin{equation*}
\Phi_{ \bfs{\beta} }:= \Phi_{(b)}+\sum \beta_k \arg \Big\{ \Theta_\chi(w_k-w_z) \Theta_\chi(\bar{w}_k-w_z) \Big\},
\end{equation*}
where $w_z=w(z)$ and $w_k=w(q_k)$. 
Let us denote by $\FF_{\bfs\beta} \equiv \FF_{\bfs\beta,(b)}$ the OPE family of the $\Phi_{\bfs\beta}$, the algebra (over $\C$) spanned by the generators $1$, mixed derivatives of $\Phi_{\bfs\beta}$, those of OPE exponentials $e^{\ast \alpha \Phi_{\bfs\beta}}$ ($\alpha \in \C$), and their holomorphic/anti-holomorphic/mixed parts (with certain neutrality conditions) under the OPE multiplication $\ast$. 
We present a more precise definition for OPE exponentials in the following subsection.

We now recall the (covering) annulus SLE with $N$-force points $q_1,\cdots q_N$. 
Like the SLE in a simply connected domain, the annulus $\SLE(\kappa,\Lambda)$ can be described by the associated Loewner's differential equation \cite{zhan2004stochastic}.
Let $\xi$ be a (real-valued) continuous function defined on $[0,T)$ for some $0< T <r$. 
By definition, the Loewner kernel $S(r,z)$ on $\A_r=\{z: e^{-r}<|z|<1\}$ is given by
$$
S(r,z):=\lim_{N \to \infty}\sum_{n=-N}^{N} \frac{e^{2nr}+z}{e^{2nr}-z}.
$$
For each $z \in \A_r$, let $g_t(z)$ be the solution (which exists up to the first time $\tau_z\in(0,T]$ that $g_t(z)$ hits $\xi_t$) of the annulus Loewner equation 
\begin{equation*}
	\displaystyle	\pa_t g_t(z) =g_t(z) S(r-t, g_t(z)e^{-i\xi_t}), \qquad g_0(z)=z.
\end{equation*} 
For each time $t \in [0,T)$, let $K_t := \{ z \in \A_r : \tau_z\leq t\}$. 
Then $K_t$ is a hull in $\A_r$ with capacity $t$. 
Moreover, $g_t$ is a conformal map from $(\A_r \sm K_t,\T_r)$ onto $(\A_{r-t},\T_{r-t})$. 

It is convenient to describe the annulus SLE in the covering space $\S_r$, where the Loewner kernel $H$ is expressed in terms of $\Theta$ as 
$$
H(r,z):=-i \, S(r,e^{iz})=2 \, \pa_z \log \Theta(r,z).
$$
For each $z\in \S_r$, we denote by $\ti{g}_t(z)$ the solution of the equation
\begin{align*}
	\pa_t \ti{g}_t (z)=H(r-t,\ti{g}_t(z)-\xi_t), \qquad \ti{g}_0(z)=z,
\end{align*}
and set $\ti{K}_t := \{ z \in \S_r : \tau_z\leq t\}$. 
Then $\ti{g}_t$ is a conformal map from $(\S_r \sm \ti{K}_t; \R_r)$ onto $(\S_{r-t};\R_{r-t})$. 
Moreover for each $z \in \S_r \sm \ti{K}_t$, $e^{i\ti{g}_t(z)}=g_t(e^{iz})$, and $K_t=\{ e^{iz}\in \A_r : z\in \ti{K}_t \}$. 
In particular, $\tilde{K}_t$ is $2\pi$-periodic and for each $n \in \mathbb{Z}$, we have $\tilde{g}_t(z+2n\pi)=\tilde{g}_t(z)+2n\pi$. 
We call that $\xi$ generates a trace $\gamma$ if 
$\gamma_t=\gamma(t):=\lim_{z \to \xi_t} \ti{g}_t^{-1}(z)$. 
It was proved by Zhan in \cite{zhan2004stochastic} that $\xi_t:=\sqrt{\kappa}\,B_t$ ($t<r$) a.s. generates continuous trace, where $B_t$ is a standard one-dimensional Brownian motion. 

Suppose $\Lambda$ is a real-valued $C^{0,1}$ function on $(0,\infty)\times (\C \sm \{ z \in 2\pi\mathbb{Z} +2ir\mathbb{Z} \})^N$. Let $\xi_t$ be the solution, (which exists up to the first time $T$ such that $\xi_t-\tilde{g}_t(q_j) \in 2\pi \mathbb{Z}$,) of the SDE
\begin{equation*}
	d\xi_t =\sqrt{\kappa}\,dB_t +\Lambda (r-t,\xi_t-\ti{g}_t(q_1),\cdots,\xi_t- \ti{g}_t(q_N) )	\, dt, \qquad \xi_0=p.
\end{equation*}
Applying the Girsanov theorem, it can be shown that up to time $T$, $\xi_t$ a.s. generates continuous trace, see \cite[Section 3.2]{MR3334276} for further details. 
We call the trace generated by $\xi_t$ the (covering) annulus $\SLE(\kappa,\Lambda)$ trace started from $p$ with force points $q_1.\cdots,q_N$. 

Now we recall the notion of SLE martingale-observables. 
A non-random (conformal) field $f$ in a planar domain $D$ (or on a Riemann surface in general) is an assignment of smooth function
$ ( f \|\phi ) : \phi \, U \to \C$
for each local chart $\phi$. 
By definition, a non-random field $M$ is a \emph{martingale-observable} for annulus $\text{SLE}(\kappa,\Lambda)$ if for any $z_1, \cdots, z_n$, the process
$$
M_t(z_1, \cdots, z_n):=(M \| \ti{g}_t^{-1} )(z_1, \cdots, z_n), \qquad (t<T)
$$
(stopped when any $z_j$ or any $q_k$ is swallowed by the Loewner hull $\tilde{K}_t$) is a local martingale on the annulus SLE probability space.

\subsection{Statement of main results} \label{Main result sub}
For a smooth vector field $v$, let us denote by $\mathcal L^+_v$ (resp., $\mathcal L^-_v$) the $\C$-linear (resp., anti $\C$-linear) part of Lie derivative operator $\mathcal L_v$ with respect to $v$, i.e., $\mathcal L_v^\pm := (\mathcal L_v \mp i \mathcal L_{iv})/2$. 
We define a \emph{stress tensor} $A_{\bfs\beta} \equiv A_{\bfs\beta,(b)}$ in terms of the current fields $J_{(0)}:=\pa \Phi_{(0)}$ and $J_{\bfs\beta}=\pa \Phi_{\bfs\beta}$ by
$$
A_{\bfs\beta}:= -\frac{1}{2}J_{(0)}\odot J_{(0)}+\Big(ib\, \pa-\E J_{\bfs\beta} \Big)J_{(0)}, 
$$
where $\odot$ denotes the Wick's product. 
The Virasoro field $T_{\bfs\beta} \equiv T_{\bfs\beta,(b)}$ is given by
$$
T_{ \bfs{\beta} } :=-\frac{1}{2}J_{ \bfs{\beta} } \ast J_{ \bfs{\beta} }+ib\,\pa J_{ \bfs{\beta} }=A_{\bfs\beta}+\E T_{ \bfs{\beta} }. 
$$
See Subsection~~\ref{Subsec b.c.m.} for more details. 

We derive the following form of Ward's equation.
\begin{thm} \label{Thm_Ward}
	For any string $\XX$ of fields in the OPE family $\FF_{\bfs\beta}$, we have
	$$ 
	2\,\E A_{\bfs\beta}(\zeta) \XX =\Big(\LL_{v_\zeta}^+ +\LL_{v_{\bar{\zeta}}}^- \Big)\E \XX +\pa_r \E \XX,
	$$ 
	where all fields are evaluated in the identity chart of $\mathcal C_r$ and the Loewner vector field $v_{\zeta}$ is given by
	$$
	(v_\zeta \| \id_{\bar{\mathcal C}_r})(z):=H(r,\zeta-z).
	$$
\end{thm}

\begin{eg*}[Hadamard's variational formula] As a simplest non-trivial example of Theorem~\ref{Thm_Ward}, let us consider the case $\bfs \beta =\bfs 0$ and $\XX=\Phi(z_1)\Phi(z_2)$. Then Ward's equation is equivalent to the following functional relation of the Green's function $G$:
\begin{equation}
4 \, \pa_{\zeta} G(\zeta,z_1) \pa_\zeta G(\zeta,z_2)+\Big( v_\zeta(z_1)\pa_{z_1}+ v_\zeta(z_2)\pa_{z_2}+ v_{\zeta}(\bar{z}_1) \bp_{z_1} + v_{\zeta}(\bar{z}_2) \bp_{z_2}+\pa_r \Big) G(z_1,z_2)=0.
\end{equation}
Letting $\zeta \to p \in \R$, this equation gives rise to Hadamard's variational formula for annulus Loewner chains:
\begin{equation*}
\frac{d}{dt}G_{\Omega_t}( z_1,z_2 )\Big|_{t=0}=-P(z_1) P(z_2), \qquad P(z):=\frac{\pa}{\pa n}\Big|_{\zeta=p} G(\zeta,z),
\end{equation*} 
where $\Omega_t:=\CC_r \backslash \tilde{K}_t$. 
Such a variational formula was studied by Izyurov and Kyt\"{o}l\"{a} \cite{MR3010393}. 
We also refer to \cite{MR3201917,MR2169966} for variants of Hadamard's formula in a planar domain. 
\end{eg*}

We define the one-leg operator $\Psi$ in terms of a modified multi-vertex field (OPE exponential) below. For this purpose, it is convenient to introduce the formal (1-point) bosonic fields $\Phi_{(0)}^{\pm}$ in $D_r$. 
Although they are not Fock space fields, they can be interpreted as the ``holomorphic part" and the ``anti-holomorphic part" of the GFF $\Phi_{(0)}$ in the following sense: 
$\Phi_{(0)}=\Phiplus_{(0)}+\Phiminus_{(0)}$, $\Phiminus_{(0)}=\overline{\Phiplus_{(0)}}$. 
Given divisors $\bfs{\sigma}=\sum_{j=1}^n \sigma_j \cdot z_j$, $\bfs{\sigma_*}=\sum_{j=1}^n \sigma_{*j} \cdot z_j$, 
we set
$$
\Phi_{(0)}[\bfs{{\sigma}},\bfs{{\sigma_*}}] := \sum_{j=1}^n \Big( \sigma_j\Phiplus_{(0)}(z_j)-\sigma_{*j}\Phiminus_{(0)}(z_j) \Big). 
$$
Then $\Phi_{(0)}[\bfs{{\sigma}},\bfs{{\sigma_*}}]$ is a well-defined Fock space field if and only if the following neutrality condition $(\mathrm{NC}_0)$ holds:
$\sum_{j=1}^n (\sigma_j+\sigma_{*j} )=0$. 

Given a double divisor $(\bfs{{\sigma}},\bfs{{\sigma_*}})$ for $z_j \in \CC_r $ satisfying the neutrality condition $(\mathrm{NC}_0)$, we define the \emph{modified multi-vertex field (OPE exponentials)} $\OO_{\bfs\beta}{[\bfs{\sigma},\bfs{\sigma_*}]}\equiv \OO_{\bfs{\beta},(b)}{[\bfs{\sigma},\bfs{\sigma_*}]}$ by
\begin{equation*} 
\OO_{\bfs{\beta}}{[\bfs{\sigma},\bfs{\sigma_*}]} := \frac{ C_{(b)}[\bfs{\sigma}+\bfs{\beta}/2,\bfs{\sigma_*}+\bfs{\beta}/2] }{ C_{(b)}[\bfs{\beta}/2,\bfs{\beta}/2] } 
\, e^{\odot i\Phi_{(0)}[\bfs{{\sigma}},\bfs{\sigma_*}]}. \end{equation*}
(The reason for the terminology ``OPE exponentials" will become clear due to Proposition~\ref{V=O}.)
Here the Coulomb gas correlation function $C_{(b)}{[\bfs{\sigma},\bfs{\sigma_*}]}$ is a differential of conformal dimensions $(\lambda_j,\lambda_{*j})= ( \tfrac12 \, \sigma_j^2 -b\, \sigma_j \, , \, \tfrac12\, \sigma_{*j}^2 -b \, \sigma_{*j}) $ at each $z_j$ and its evaluation in the identity chart of $\CC_r$ is given by 
\begin{align*}
C_{(b)}{[\bfs{\sigma},\bfs{\sigma_*}]}=	&\,\Theta'(0)^{\frac12\sum_j (\sigma_j^2+\sigma_{*j}^2)} \exp \Big( -\frac{ (\sum \sigma_j z_j+\sigma_{*j} \bar{z}_j )^2 }{4(r+\chi)} \Big) \prod_j \Theta(z_j-\bar{z}_j )^{\sigma_j\sigma_{*j}} 
\\
\times& \prod_{j<k} \Theta(z_j-z_k)^{\sigma_j\sigma_k}\Theta(\bar{z}_{j}-z_k)^{\sigma_{*j}\sigma_k} \Theta(z_j-\bar{z}_k)^{\sigma_j\sigma_{*k}}\Theta(\bar{z}_j-\bar{z}_k)^{\sigma_{*j}\sigma_{*k}}.
\end{align*}
Given a divisor $\bfs{\tau}= \sum \tau_j \cdot \xi_j$ satisfying $ a+\sum \tau_j =0$, we define the \emph{one-leg operator} $\Psi$ by 
$$
\Psi(z) \big(\equiv \Psi_{\bfs{\beta}}(z, \bfs{z})\equiv \Psi_{\bfs{\beta}}[ a \cdot z+ \bfs{\tau} ] \, \big)  :=\OO_{\bfs\beta} [ a\cdot z + \tfrac12 \bfs{\tau} \, , \, \tfrac12 \bfs{\tau} ]. 
$$
Then the following form of BPZ equation holds.

\begin{thm} \label{BPZ_ER}
	Suppose that $2a(a+b)=1$. Then for any $\XX \in \FF_{\bfs\beta}$, we have
\begin{equation}
\frac{1}{a^2}\pa^2_{z}\E \Psi(z) \XX = \Big( \LL_{v_{z}}^++\LL_{v_{\bar{z}}}^-+\pa_r+2h_{1,2} \frac{\zeta_r(\pi)}{\pi}+2\E T_{\bfs\beta}(z) \Big) \E \Psi(z) \XX,
\end{equation}
where all fields are evaluated in the identity chart of $\mathcal C_r$.
Here, Lie derivative operators do not apply to $z$ and $h_{1,2}=a^2/2-ab$. 
\end{thm}

To connect CFT with SLE theory, we choose real parameters $a$ and $b$ in terms of SLE parameter $\kappa$ as 
\begin{equation*} \label{a,b,kappa}
	a=\sqrt{2/\kappa}, \qquad b=\sqrt{\kappa/8}-\sqrt{2/\kappa}.
\end{equation*} 
Now we consider $\SLE(\kappa,\Lambda)$, where $\Lambda$ is defined by 
\begin{equation*} \label{Lambda_main}
	\Lambda(r,p,\bfs{\xi}):=\kappa \frac{ \pa_{\xi}|_{\xi=p} \E \Psi_{\bfs\beta}(\xi,\bfs{\xi}) }{Z_{\bfs\beta}(p,\bfs{\xi})}, \qquad Z_{\bfs\beta}(p,\bfs{\xi}):=\E \Psi_{\bfs\beta}(p,\bfs{\xi}).
\end{equation*} 

The insertion of the one-leg operator produces an operator that changes the values of Fock space fields. 
See \cite[Section~2.3]{MR3052311} for its probabilistic counterpart: insertion of Wick's exponential $e^{\odot \Phi(f)}$ of the GFF in application to a test function $f$ results in the change of the law of $\Phi(f)$ in terms of the Green potential. 
We now consider the following insertion procedure:
\begin{equation*} \label{modi cor_main}
	\wh{\E}[\mathcal X]:=\frac{\E \Psi_{ \bfs{\beta} }(p , \bfs{\xi})\XX }{ Z_{\bfs\beta}(p,\bfs{\xi}) }.
\end{equation*}
By the BPZ-Cardy type equations (Proposition~\ref{BPZ-C_ER}), we construct a family of SLE martingale-observables.

\begin{thm} \label{main}
	For any string $\XX$ of fields in the OPE family $\FF_{\bfs\beta}$, a non-random field 
	$M=\wh{\E}[\XX]$ is a martingale-observable for $\SLE(\kappa,\Lambda)$. 
\end{thm}

\begin{eg*} (\emph{Bosonic observables})
Let $\bfs\beta=0$ and $M(z):=\wh{\E} \Phi(z)$. Then $M$ is evaluated as 
\begin{equation}
M(z)=-2b\arg w'+2a \arg \Theta_\chi(r,w_z)+\sum \tau_k \arg \Big\{ \Theta_{\chi}(r,w_z-w_k) \Theta_{\chi}(w_z-\bar{w}_k) \Big\}. 
\end{equation}
In particular when $\chi=0$, the harmonic function $M$ satisfies piecewise Dirichlet boundary conditions having additional jump of $2\pi a$ at $p$ and $2\pi \tau_k$ at $\xi_k$'s. 

For $b=0$ (thus when $\kappa=4$) the associated bounded martingale plays an important role in the study of scaling limits of discrete GFF interface \cite{MR2486487} and GFF/SLE couplings \cite{MR3010393,MR2525778}. 
Further applications of such observables are presented in Propositions~\ref{Prop_LPP SLE4} and ~ \ref{Prob:hitting} in the context of left passage probability and hitting probability of SLE traces. 
\end{eg*}

\begin{eg*} (\emph{Vertex observables}) Let $\bfs{\tilde{\tau}}=\sum \tilde{\tau}_j\cdot \tilde{\xi}_j$ be a divisor satisfying $a+\int \bfs{\tilde{\tau}}=0$. 
Consider the vertex observable $M$ given by 
$$
M=\wh{\E} \OO_{\bfs\beta}[ \tilde{\bfs{\tau}}-\bfs{\tau} ]=\frac{ \E \Psi_{\bfs\beta}[a \cdot p+\bfs{\tilde{\tau}}] }{\E \Psi_{\bfs\beta}[a \cdot p+\bfs{\tau}]}.
$$	
This martingale-observable provides the Radon-Nikodym derivative between the law of two annulus SLEs associated with partition functions $\E \Psi_{\bfs\beta}[a \cdot p+\bfs{\tilde{\tau}}] $ and $\E \Psi_{\bfs\beta}[a \cdot p+\bfs{\tau}]$. 
Similar martingale-observables also appear in the study of restriction properties, see e.g., \cite{MR1992830,MR2118865}.
\end{eg*}

We emphasize that except for the ER case when $\chi=\infty$, the drift function $\Lambda$ is not $2\pi$-periodic with respect to space variables. 
To construct martingale-observables for annulus SLE with $2\pi$-periodic drift function, we use the weighted summation of the chiral fields, see \cite{MR896667} for a similar idea on compact Riemann surfaces. (cf. see also \cite{MR3334276} for implementation of this idea to the annulus SLE partition functions.)
To be more precise, for a weight function $\omega: \mathbb{Z}\to \R_{+}$, set 
$$
\Psi_{\bfs\beta}^\omega(p,\bfs{\xi}):=\sum_{ n \in \mathbb{Z} } \omega(n) \Psi_{\bfs\beta}(p+2n\pi,\bfs{\xi}).
$$ 
Here we assume that $\omega$ is suitably chosen so that the summation above converges within correlations. 
For instance, when $\chi<\infty$, we allow any weight function $\omega$ of polynomial type.
We then define 
$$
Z_{\bfs\beta}^\omega(p,\bfs{\xi}):=\E \Psi_{\bfs\beta}^\omega(p,\bfs{\xi}), \qquad \Lambda^\omega(r,p,\bfs{\xi}):=\kappa \, \pa_p \log Z_{\bfs\beta}^\omega(p,\bfs{\xi}). 
$$
Extending Theorem~\ref{main}, we obtain the following. 

\begin{thm} \label{main_period}
	For any string $\XX$ of fields in the OPE family $\FF_{\bfs\beta}$, a non-random field 
	\begin{equation}
	M=\wh{\E}[\XX]:=\frac{\E \Psi^\omega_{ \bfs{\beta} }(p , \bfs{\xi})\XX }{ Z_{\bfs\beta}^\omega(p,\bfs{\xi}) } 
	\end{equation}
	is a martingale-observable for $\SLE(\kappa,\Lambda^\omega)$. 
\end{thm}

\begin{eg*}
Let us consider the case that $\bfs\beta=\bfs 0$, $\bfs \tau=-a \cdot q$ for $q \in \R_r$ and $\omega(n)\equiv 1$. 
Then up to a multiplicative constant, the partition functions $Z \equiv Z_{\bfs\beta}(p,\bfs{\xi})$ and $Z^\omega \equiv Z_{\bfs\beta}^\omega(p,\bfs{\xi})$ are given by ($x=p-\Re \, q$) 
$$
Z(x)=\Theta_I(r,x)^{ \frac{2}{\kappa} } \exp\Big( \frac{x^2}{2\kappa(r+\chi)} \Big), \qquad 
Z^\omega(x)=\Theta_I(r,x)^{ \frac{2}{\kappa} }\Theta_I\Big(\frac{\kappa}{2}(r+\chi),x+\pi \Big), 
$$
where 
$$
\Theta_I(r,z)=i \, \Theta(r,z-ir) \exp\Big(-\frac{r}{4}-\frac{iz}{2} \Big).
$$
This leads to
$$
\Lambda(x)=H_I(r,x)+\frac{x}{r+\chi}, \qquad \Lambda^\omega(x)=H_I(r,x)+\frac{\kappa}{2} H_I\Big(\frac{\kappa}{2}(r+\chi),x+\pi \Big),
$$
where $H_I(r,z):=2 \, \pa_z \log \Theta_I(r,z)$.
Notice here that $\Lambda^\omega$ is $2\pi$-periodic, whereas $\Lambda$ is not. 

In particular when $\chi=0$, it can be shown that the associated SLE$(\kappa,\Lambda^\omega)$ trace ends at $\{ q+2n\pi : n \in \mathbb{Z} \}$. 
Then by virtue of vertex observables in the previous example, one can see that for each $m \in \mathbb{Z}$, 
$$
M_m(x):=\frac{Z(x-2m\pi)}{Z^\omega(x)}=\exp\Big( \frac{(x-2m\pi)^2}{2\kappa\,r} \Big)\Theta_I\Big(\frac{\kappa \, r}{2},x+\pi \Big)^{-1}
$$
gives rise to the probability that SLE$(\kappa,\Lambda^\omega)$ trace ends at $q+2m\pi$.
\end{eg*}

We now focus on the \emph{chordal type} annulus $\SLE(\kappa,\Lambda)$ with one force point $q$ starting from $p$, where both of the marked points lie on the same boundary component. 
It was shown in \cite{MR3334276} that in order for the SLE process to be reversible, the associated partition function $Z$ should satisfy the null-vector equation 
\begin{equation}
\label{Lawler Zhan null}
	\pa_r Z = \frac{\kappa}{2} Z'' + H\, Z' + h_{1,2}(\kappa)H'\, Z+C(r) Z, \qquad h_{1,2}(\kappa) := \frac{6-\kappa}{2\kappa},
\end{equation}
where $C(r)$ is a constant depending only on the modular parameter $r$. 

To construct CFT for such chordal type annulus SLE, we use the method of \emph{screening}. 
Due to the local commutations of annulus SLE$(\kappa,\Lambda)$, the partition function should have the same conformal dimensions at $p$ and $q$. 
In the case of chordal SLE in the upper-half plane, it can be achieved by the \emph{effective one-leg operator} with the total charge $2b-a$ (including the background charge $2b$) placed at $q$ and with the charge $a$ placed at $p$. 
For more details we refer the reader to \cite{KM}. 
However, this choice of partition function cannot be realized in the annulus because the neutrality conditions depend on the connectivity. 
See \cite{KM17} for the neutrality conditions for background charges on a compact Riemann surface of genus $g$. 
Instead we consider 
$$
\Psi_\beta(p,q):=C(\kappa) \oint_{\mathcal{P}(p,q)} \OO_{\bfs\beta}[a\cdot p +a\cdot q +s \cdot \zeta, \, \bfs{0}] \, d\zeta, \qquad (s = -2a),
$$
where $\mathcal{P}(p,q)$ is the Pochhammer contour entwining $p$ and $q$, see Figure~\ref{Pochh}. 
Here $C(\kappa)$ is a universal constant depending only on $\kappa$ and the background charge is given as $\bfs\beta:=\beta \cdot q_1-\beta \cdot q_2$ ($q_2-q_1=2\pi$) for some $\beta \in \R$. 

\begin{figure}[h!]
	\begin{center}
		\includegraphics[width=5.0in]{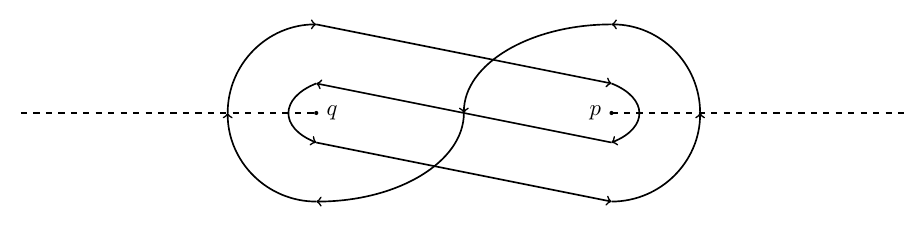}
	\end{center}
	\caption{Pochhammer contour} \label{Pochh}
\end{figure}

With this choice of $s =-2a$, the modified multi-vertex field $\OO_{\bfs\beta}[a\cdot p +a\cdot z +s \cdot \zeta, \, \bfs{0}]$ is a $1$-differential with respect to the ``screening" variable $\zeta$ and it satisfies the neutrality condition $(\mathrm{NC}_0)$. 
Clearly, $\Psi_\beta$ has the same conformal dimensions both at $p$ and $q$. 
For each $\kappa>0$, we have the following expression (up to a multiplicative constant) of $Z_\beta\equiv Z_\beta(r,p-q):=\E\Psi_\beta(p,q)$ in terms of Jacobi's theta function: 
\begin{equation}
Z_{\beta }=\Theta(p-q)^{\frac{2}{\kappa}} \oint_{\mathcal{P}(p,q)} \Theta(p-\zeta)^{-\frac{4}{\kappa} } \Theta(\zeta-q)^{-\frac{4}{\kappa} } \exp\Big( - \frac{ (\Sigma-2\beta \pi )^2 }{ 4(r+\chi) } \Big)\, d\zeta,
\end{equation}
where $\Sigma=a \, ( p+q-2\zeta )$. 
We remark that for $\kappa >4$, this Pochhammer contour integral simplifies to the integration over the interval between $p$ and $q$. 
It can be shown by the contour deformation method as we explain in Subsection~\ref{Screening_sub}. 
With this choice of the one-leg operator $\Psi_\beta(p,q)$, we construct a family of martingale-observables for chordal type $\SLE(\kappa,\Lambda_\beta)$, where $\Lambda_\beta=\kappa \, \pa_p \log Z_\beta$. 
\begin{thm} \label{main_SLE Z}
	For each $\kappa >0$, the partition function $Z_\beta:=\E\,\Psi_\beta(\cdot,q)$ is a non-trivial (real-valued) solution of the null-vector equation \eqref{Lawler Zhan null}. 
	Moreover, for any string $\XX$ of fields in the OPE family $\FF_{\bfs\beta}$, a non-random field 
	$$
	M=\wh{\E} \XX:= \frac{\E\Psi_\beta(p,q) \XX}{ \E\Psi_\beta(p,q) }
	$$
	is a martingale-observable for chordal type $\SLE(\kappa,\Lambda_\beta)$. 
\end{thm}

\subsection{Organization of the paper and further results}
The rest of this paper is organized as follows.

In Section~\ref{CFT with ER b.c.} we review some standard notions in CFT and describe the conformal Fock space fields generated by central charge modifications of GFF. 
For the convenience of the reader we borrow the relevant material from \cite{MR3052311} without proofs, thus making our exposition self-contained.

Section~\ref{WARD Section} is devoted to the study of Eguchi-Ooguri and Ward's equations (Theorem~\ref{Thm_Ward}). 

In Section~\ref{SLE MOs ER} we show the BPZ equations (Theorem~\ref{BPZ_ER}). After relating CFT to SLE theory, we then show the BPZ-Cardy equations (Proposition~\ref{BPZ-C_ER}) and complete the proof of Theorem~\ref{main}. 

In Section~\ref{Section_further generalizations}, we present several ways to extend the class of SLE martingale-observables and prove Theorems \ref{main_period} and ~\ref{main_SLE Z}. 
Moreover, the implementation of non-atomic background charge is explained in Subsection~\ref{Subsec_NVE msr} as well. 

In Section~\ref{Section_Examples}, based on the theories developed in the previous sections, further examples of SLE martingale-observables are indicated, which include those used in the study of continuum limits of discrete models. 

In Appendix~\ref{Appendix_RH}, we present representations of various Green's functions in terms of special functions.

\endgroup
\setcounter{thm}{\thetmp}

\section{Conformal Fock space fields} \label{CFT with ER b.c.}

We introduce a class of random fields in a doubly connected domain as correlation functional valued maps. 
All fields that we consider in this paper are constructed from the (modifications of) GFF and their derivatives by means of Wick's calculus and named after Fock space. 
We consider Dirichlet and ER boundary conditions for the GFF and their interpolations, see Subsection~\ref{Cor ftns_Sub} for more details. 
This subsection also recalls some basic concepts and properties concerning the Fock space fields and OPE products.
In Subsection~\ref{Conformal Subsec} we revise the definition of Fock space fields so that their values depend on local coordinates. 
In Subsection~\ref{subsec_CCM} we introduce the central charge modifications of GFF and present the associated Virasoro field. 
Based on the Coulomb gas formalism introduced in the following subsections, in Subsection~\ref{Subsec b.c.m.}, we present a more general theory of background charge modifications of GFF.

\subsection{Special functions} \label{Special_sub}

In this subsection we compile basic properties of some special functions used to represent the Loewner vector field and the correlation functions of Fock space fields. 

For each $r >0$, the Loewner kernel $H(r, \cdot)$ in $\S_r$ is given by 
\begin{equation*}
	H(r,z):= -i \, \textrm{P.V.} \lim_{n \to \infty} \sum_{n \in 2\mathbb Z} \frac{e^{nr}+e^{iz}}{e^{nr}-e^{iz}}.
\end{equation*}
We also write 
$$
H_I(r,z):=H(r,z+ir)+i.
$$
We refer to \cite{MR3334276} for fundamental properties of the Loewner kernel $H(r,\cdot)$. 
For reader's convenience, we list some of them as follows:
\begin{itemize}
	\item a meromorphic function $H(r, \cdot) $ on $\C$ has simple poles at $2\pi m +2ir n, (m,n \in \mathbb{Z})$ with residue 2;
	\item $H(r,z+2\pi)=H(r,z)$ and $H(r,z+2ir)=H(r,z)-2i$;
	\item $H(r,\cdot)$ is an odd function; 
	\item $H(r,\cdot)$ takes real values on $\R\sm \{2\pi n: n\in \mathbb{Z}\}$ and $\Im\, H(r, \cdot)\equiv -1$ on $\R_r$;
	\item $H(r,\pi)=0$, $H(r,ir)=-i$, and $H(r,\pi+ir)=-i$. 
\end{itemize}

A holomorphic function
$$
\Theta(r,z):=\frac{1}{i}\sum_{n=-\infty}^{\infty}(-1)^ne^{-r (n+\frac{1}{2} )^2 }e^{ (n+\frac{1}{2} )iz}
$$ 
vanishes only on $2\pi\mathbb{Z} + 2ir \mathbb{Z}$ and 
$$\Theta_I(r,z):=\sum_{n=-\infty}^{\infty}(-1)^ne^{-rn^2 }e^{niz}$$
is related to $\Theta$ as 
\begin{equation} \label{Theta and I}
\Theta_I(r,z)=i \, \Theta(r,z-ir) \exp\Big(-\frac{r}{4}-\frac{iz}{2} \Big).
\end{equation}
Notice that $\Theta(r,\cdot)$ is odd, whereas $\Theta_I(r,\cdot)$ is even. 
The functions $H$ and $H_I$ are represented in terms of $\Theta$ and $\Theta_I$ as 	
\begin{equation} \label{H Theta}
H(r,z)=2\frac{\Theta'(r,z)}{\Theta(r,z)}, \qquad H_I(r,z)=2\frac{\Theta'_I(r,z)}{\Theta_I(r,z)},
\end{equation}
where derivatives are taken with respect to the $z$-variable.
We remark that $\Theta$ and $\Theta_I$ are related to the classical Jacobi theta functions 
$$
\theta(z,\tau)=\frac{1}{i} \sum_{n=-\infty}^\infty (-1)^n e^{\pi i \tau (n+\frac{1}{2})^2} e^{(2n+1)\pi i z}, \qquad 
\theta_2(z,\tau)= \sum_{n=-\infty}^\infty (-1)^n e^{\pi i \tau n^2} e^{2n\pi i z} 
$$
as 
$$\Theta(r,z)=\theta \Big(\frac{z}{2\pi},\frac{ir}{\pi} \Big),\qquad \Theta_I(r,z)=\theta_2 \Big(\frac{z}{2\pi},\frac{ir}{\pi} \Big).$$
It is well known that $\Theta$ and $\Theta_I$ solve the heat equations		
\begin{equation} \label{heat eq. Theta}
\pa_r\Theta=\Theta'',	\qquad	\pa_r\Theta_I=\Theta_I''.
\end{equation}
Moreover, they are represented in terms of fundamental solutions of the heat equation as
\begin{align}
\begin{split}
\label{Theta_I Gauss}
\Theta(r,z)&=\sqrt{\frac{\pi}{r}} \sum_{n \in \mathbb{Z}} (-1)^n \exp\Big( -\frac{(z-\pi+2n\pi)^2}{4r} \Big), 
\\
\Theta_I(r,z)&=\sqrt{\frac{\pi}{r}} \sum_{n \in \mathbb{Z}} \exp\Big( -\frac{(z-\pi+2n\pi)^2}{4r} \Big),
\end{split}
\end{align}
see e.g., \cite[Eq.(20.2.13),(20.13.4)]{olver2010nist}.

To describe correlation functions with weighted boundary conditions indexed by $\chi \in [0,\infty]$, it is convenient to introduce the following theta function
\begin{equation}
\Theta_\chi(r,z):= \Theta(r,z) \exp\Big( \frac{z^2}{4(r+\chi)} \Big).
\end{equation}
In the sequel, we also write 
\begin{equation}
H_\chi(r,z):=2 \log \Theta'_\chi(r,z)= H(r,z)+\frac{z}{r+\chi}.
\end{equation}
In particular, we denote $\ti{\Theta}:=\Theta_{0}$ and $\tilde{H}:=H_0$. 
It is worth pointing out that $\ti{\Theta}(r,z) \in i \R_+$ if $z \in \R_r$. 
Note also that the function $H_\chi(r,\cdot)$ is real-valued on $\R_r$ if and only if $\chi=0$. 
Throughout this paper, we will frequently use the following (quasi) periodicities of theta functions 
\begin{align}
&\Theta(r,z+2\pi)=-\Theta(r,z), \qquad 
\Theta(r,z+2ir)=-\Theta(r,z) \exp(r-iz)\label{theta per},
\\
&\Theta_I(r,z+2\pi)=\Theta_I(r,z), \qquad 
\Theta_I(r,z+2ir)=-\Theta_I(r,z) \exp(r-iz) \label{theta I per}, 
\\
&\tilde{\Theta}(r,z+2\pi)=-\tilde{\Theta}(r,z) \exp(\pi (z+\pi)/r), \qquad 
\tilde{\Theta}(r,z+2ir)= -\tilde{\Theta}(r,z). \label{ti theta per}
\end{align}

The Weierstrass zeta function $\zeta_r$ with basic periods $(2\pi,2ir)$
\begin{equation} \label{eq: zeta} 
	\zeta_r(z):= \frac{1}{z}+ {\sum_\eta}^* \Big(\frac{1} {z-\eta}+\frac{1}{\eta}+ \frac{z}{\eta^2} \Big)
\end{equation}
is holomorphic except for the points $z \in 2\pi\mathbb{Z} +2ir\mathbb{Z}$, where it has simple poles with residue 1.
Here, the sum $\sum_\eta^* $ in \eqref{eq: zeta} is taken over all $\eta$ of the form $2\pi m +2ir n$ for $(m,n) \in \mathbb{Z}^2\setminus\{(0,0)\}$. 
It is well known that $\zeta_r(\pi)/\pi$ is represented as 
\begin{equation}
\label{Mod inv}
\frac{\zeta_r (\pi)}{\pi}=-\frac{1}{3}\frac{\Theta'''(r,0)}{\Theta'(r,0)} =-\frac13 \pa_r \log \Theta'(r,0)
= \frac{1}{12}- \frac{1}{2}\sum_{k=1}^{\infty} \sinh^{-2}(kr). 
\end{equation}
This value will appear in the formula of the correlation function of the Virasoro field (Proposition~\ref{Virasoro2_ER}). 

The Loewner kernel $H$ is represented in terms of the Weierstrass zeta function: 
\begin{equation} \label{h-zeta}
	H(r,z)=2\, \zeta_r(z)-\frac{2}{\pi}\zeta_r(\pi)\, z.
\end{equation}
In particular, we have the following asymptotic behavior of $H$:		
\begin{equation}\label{h-zero}
	H(r,z)=\frac{2}{z}-\frac{2\zeta_r(\pi)}{\pi}z+O(z^3),	\qquad \textrm{as } z \to 0.
\end{equation}

The following pseudo-addition formula (see \cite[p.57]{MR808396}) for zeta function is a crucial ingredient to derive a version of Eguchi-Ooguri equations (Proposition~\ref{Arep_ER}).

\begin{prop} \label{addthm}
	If $x+y+z=0$, then
	\begin{equation} \label{functional equation for zeta}
		(\zeta_r(x)+\zeta_r(y)+\zeta_r(z))^2+\zeta_r'(x)+\zeta_r'(y)+\zeta_r'(z)=0.
	\end{equation}
\end{prop} 

By \eqref{H Theta} and \eqref{h-zeta}, one can rewrite \eqref{functional equation for zeta} as follows:
\begin{align}
	\begin{split} \label{add H}
		H(z-w) ( H(z)-H(w) ) &= \frac{1}{2}H^2(z-w)+\frac{1}{2}( H(z)-H(w))^2
		\\&+ H'(z -w)+H'(z)+H'(w)+\frac{6}{\pi}\zeta_r(\pi).
	\end{split}
\end{align}

\subsection{Correlation functions of Fock space fields} \label{Cor ftns_Sub}

In this subsection we introduce a one-parameter family of GFFs with a weighted combination of Dirichlet and ER boundary conditions in a doubly connected domain and recall some basic concepts in CFT such as Fock space fields and OPE products. 
As an expansion of the tensor product of two Fock space fields near diagonal, operator product expansion (OPE) gives rise to important binary operations, namely OPE multiplications or OPE products on Fock space fields. 

\subsubsection {One-parameter family of Gaussian free fields}

An excursion reflected Brownian motion (ERBM) $B_{D_r}^{ER}$ in a doubly connected domain $D_r$ is a strong Markov process stopped at the outer boundary of $D_r$ that acts like a Brownian motion when it is away from $\pa D_r$. 
Also it reflects on the inner boundary of $D_r$ at a random point chosen according to the harmonic measure from $\infty$ whenever it hits the inner boundary. 
It is well known that ERBM is conformally invariant and so is the Green's function for ERBM.
See \cite{lawler2011defining,MR2992562} for more details. 
We also remark that ERBM is a special case of a more general process called Brownian motion with darning \cite{chen2016chordal}, which can be constructed using Dirichlet form theory. 
See also \cite{MR3706737} for another generalization of ERBM, called obliquely reflected Brownian motion. 

The Green's function $G_{D_r}^{Diri}$ in a doubly connected domain $D_r$ with zero Dirichlet boundary condition (b.c.) vanishes on $\pa D_r$. 
In the identity chart of the annulus $\A_r$ we have 
\begin{equation} \label{A_r Diri Green}
	G_{\A_r}^{Diri}(\zeta,z)=-\frac{\log|\zeta| \log|z|}{r}-\log \Big|z \frac{\sum_{k=-\infty}^{\infty}(-1)^ke^{-(k^2-k)r}(\zeta/z)^k}{\sum_{k=-\infty}^{\infty}(-1)^ke^{-(k^2-k)r}(\zeta\bar{z})^k} \Big|,
\end{equation}
see e.g., \cite[p.262]{MR822470}. 
The Green's function $G_{\A_r}^{ER}$ for ERBM in the annulus $\A_r$ is characterized by 
$$
\mathbb{E}^\zeta \Big[ \int_0^{\tau_{\A_r}} \bfs{1}_E ( B_{\A_r}^{ER} ) (t)\, dt \Big] =\int_E G_{\A_r}^{ER}(\zeta,z)\, dz
$$
for any Borel subset $E \subset \A_r$, where $\tau_{\A_r}:=\inf \{ t: B_{\A_r}^{ER}(t) \in \mathbb{T} \}$.
Then by \cite[Proposition 5.2 and Lemma 5.5]{MR2992562}, we have 
$$
G_{\A_r}^{ER}(\zeta,z)=-\log \Big|z \, \frac{\sum_{k=-\infty}^{\infty}(-1)^ke^{-(k^2-k)r}(\zeta/z)^k}{\sum_{k=-\infty}^{\infty}(-1)^ke^{-(k^2-k)r}(\zeta\bar{z})^k} \Big|.
$$ 
In the cylinder $\CC_r$, such Green's functions $G_r \equiv G_{\CC_r}$ are represented in terms of Jacobi theta function as
\begin{equation} \label{Green ER/Diri}
G_r(\zeta,z)=\begin{cases}
\log \Big|	\dfrac{\Theta(r,\zeta-\bar{z})}{\Theta(r,\zeta-z)}\Big| &\text{for ER b.c.}
\vspace{0.5em}
\\
\displaystyle \log \Big|	\dfrac{\Theta(r,\zeta-\bar{z})}{\Theta(r,\zeta-z)}\Big|-\frac{\Im\, \zeta \, \Im\, z}{r} &\text{for Dirichlet b.c.}
\end{cases}
\end{equation} 

From the analytic point of view, the Green's function $G_r$ with ER b.c. satisfies the \emph{zero period condition} on the inner boundary component $\gamma$. Therefore we have that for $z\in\gamma$,
\begin{equation} \label{b.c. Diri ER}
\begin{cases}
\displaystyle \oint_{\gamma} \frac{\pa G_r }{\pa n_z} \,
ds_z=0& \text{for ER b.c.}
\vspace{0.5em}
\\
\hspace{1.2em}G_r(\zeta,z)=0& \text{for Dirichlet b.c.} 
\end{cases}
\end{equation} 
Here $\pa n_z$ denotes the normal derivative and $ds_z$ is the arc length measure. 
Extending \eqref{b.c. Diri ER}, we introduce the weighted combination of Dirichlet and ER boundary conditions: for $\chi \in [0,\infty]$, 
\begin{equation} \label{b.c. chi}
\frac{1}{1+\chi} G_r(\zeta,z)+\frac{\chi}{1+\chi} \Big(\frac{1}{2\pi} \oint_{\gamma} \frac{\pa G_r}{\pa n_z} \, ds_z\Big)=0,
\end{equation}
where $z\in\gamma$.
One may realize \eqref{b.c. chi} as an analog of the well-known Robin boundary condition, in which the Neumann condition is replaced by the ER one.
It is then straightforward to see that the Green's function 
\begin{align} \label{Green interpol}
\begin{split}
G_r^\chi(\zeta,z):=\log \Big|	\dfrac{\Theta_\chi(r,\zeta-\bar{z})}{\Theta_\chi(r,\zeta-z)}\Big|
=\log \Big|	\dfrac{\Theta(r,\zeta-\bar{z})}{\Theta(r,\zeta-z)}\Big|-\frac{\Im\, \zeta \, \Im\, z}{r+\chi}
\end{split}
\end{align} 
satisfies the boundary condition \eqref{b.c. chi}. 

We denote by $\Phi$ the GFF with weighted boundary condition \eqref{b.c. chi}, see Figure~\ref{Fig_DGFF}. By definition, the $2$-point correlation function of GFF $\Phi$ in $D_r$ is given by
\begin{equation}\label{2pt cor_ER}
	\E[\Phi(\zeta)\Phi(z)]:=2\,G^\chi_{D_r}(\zeta,z)=2\log \Big|	\frac{\Theta_\chi(r,w(\zeta)-\overline{w(z)})}{\Theta_\chi(r,w(\zeta)-w(z))}\Big|,
\end{equation} 
where $w$ is a conformal map from $D_r$ onto $\mathcal C_r$. In general, the $n$-point correlation function of $\Phi$ is given by 
\begin{equation}\label{n-pt cor}	
	\E[\Phi(z_1)\cdots \Phi(z_n)]=\sum \prod_k 2G^\chi_{D_r}(z_{i_k},z_{j_k}),
\end{equation}
where the sum is taken over all partitions of the set $\{ 1,\cdots ,n \}$ into disjoint pairs $\{ i_k ,j_k\}$. 
The \emph{current field} $J=\pa \Phi$, its conjugate $\bar J=\bp \Phi$, and higher order derivatives of $\Phi$ are distributional fields (random generalized functions) and their correlators are computed by differentiating those of GFF.
For example, for $\zeta\neq z $, in the identity chart of $\CC_r$, we have
\begin{equation}\label{2pt cor CG_ER}
	\E[J(\zeta)\Phi(z)]=\dfrac{1}{2} (H_\chi (r,\zeta-\bar{z})-H_\chi(r,\zeta-z) ), \qquad 	\E[J(\zeta)J(z)] =\frac12 \, H^{\prime}_\chi (r,\zeta-z ).
\end{equation}

\subsubsection{Fock space fields} 

We treat the GFF as a Gaussian (correlation) functional-valued function and view its correlation function as the kernel representing the expectation of the value of GFF on test functions. 
We also treat the derivatives of GFF as centered (mean zero) Gaussian generalized functions and include Wick's product of derivatives of GFF in a collection of fields we consider. 

By definition, a basic Fock space correlation functional is a (formal form of) Wick's product $\mathcal X=X_1 (z_1)\odot \cdots \odot X_n (z_n)$ of derivatives $X_j$ of the GFF, and it has zero correlation, i.e., $\E \mathcal X=0$. 
Here, points $z_j\in D_r$ are not necessarily distinct and called the \emph{nodes} of $\mathcal X$. 
We write $S_{\mathcal X} \equiv S(\mathcal X)$ for the set of nodes of $\mathcal X$. 

For basic functionals $\XX_j$ of the form $\XX_j=X_{j1} (z_{j1})\odot \cdots \odot X_{jn_j} (z_{jn_j})$ with pairwise disjoint $S(\mathcal X_j)$ we define the \emph{tensor product} $\XX=\XX_1 \cdots \XX_m$ by Wick's formula
\begin{equation}\label{Wick's formula}
	\XX_1 \cdots \XX_m =\sum \prod_{\{	v,v'\}	}\E [X_v (z_v)X_{v' }(z_{v'})]\bigodot_{v''}X_{v''} (z_{v''}).
\end{equation} 
The sum in the right-hand side of \eqref{Wick's formula} is taken over all graphs (Feynman diagrams) with vertices $v$ labeled by functionals $X_{jk}$ such that there are no edges (Wick's contractions) of vertices with the same $j$, and the Wick's products are over unpaired vertices $v''$. 
For example, the Feynman diagram with edges $\{1,4\},\{3,6\}$ and unpaired vertices $2,5$ corresponds to
$$\contraction{}{(\Phi(z}{_1) \odot \Phi(z_2) \odot \Phi(z_3))(}{\Phi(z} \contraction[2ex]{(\Phi(z_1) \odot \Phi(z_2) \odot }{\Phi(z}{_3))(\Phi(z_4) \odot}{\Phi(z} (\Phi(z_1) \odot \Phi(z_2) \odot \Phi(z_3))(\Phi(z_4) \odot \Phi(z_5) \odot \Phi(z_6)) := \E[\Phi(z_1)\Phi(z_4)] \E[\Phi(z_3)\Phi(z_5)] \Phi(z_2)\odot \Phi(z_6).$$
The tensor product of correlation functionals is commutative and associative, see \cite[Proposition 1.1]{MR3052311}. 

We identify two functionals $\mathcal X_1$ and $\mathcal X_2$ if $\E[\XX_1 \YY]=\E[\XX_2 \YY]$ holds for all functionals $\YY$ with $S_{\mathcal Y} \cap (S_{\XX_1} \cup S_{\XX_2}) = \emptyset$. 
The complex conjugation $\overline{\XX}$ of $\XX$ is the unique correlation functional (modulo trivial functionals) such that 
$ \E [\overline{\XX} \YY ]=\overline{\E [\XX\YY ]}$
for all $\YY$'s of the form $\Phi(z_1)\odot \cdots \odot \Phi(z_n)$. 

\subsubsection{OPE calculations}
The \emph{operator product expansion} (OPE) of two Fock space fields $X$ and $Y$ is an asymptotic expansion of $X(\zeta)Y(z)$ near diagonal. 
In particular, if a field $X$ is holomorphic (i.e., $\bp X = 0$ within correlation), the OPE is defined as a formal Laurent series expansion
$$
X(\zeta)Y(z)=\sum C_n(z) (\zeta -z )^n, \qquad \text{as } \zeta \to z.
$$
More precisely, it means that $\E X(\zeta)Y(z) \XX =\E \sum C_n(z)(\zeta -z )^n \XX$ for each Fock space correlation functional $\XX$ for $\zeta,z \notin S_{\XX}$.
In this case, $*_n$ product is defined as $X *_n Y:=C_n$, the $n$-th OPE coefficient. 
In particular, we write $*$ for $*_0$ and call $X*Y$ the \emph{OPE multiplication}, or the \emph{OPE product} of $X$ and $Y$. 
We remark that the $\ast_n$ product is neither associative nor commutative. On the other hand, it satisfies Leibniz's rule, $\pa ( X *_n Y )= \pa X *_n Y + X *_n \pa Y$.

We now present some examples of OPEs. As $\zeta \to z, (\zeta \neq z)$, 
\begin{equation}\label{OPE_G}
	\Phi(\zeta)\Phi(z)=\log \frac{1}{|\zeta -z|^2}+2c(z)+\Phi^{\odot 2}(z) +o(1),
\end{equation}
where $c(z):=u(z,z)$ and $u(\zeta,z):=G_{D_r}^\chi(\zeta,z)+\log|\zeta-z|$. 
By definition, for $z \in D_r$, $\Phi* \Phi(z) = 2c(z)+ \Phi^{\odot 2}(z)$, where $c(z)$ is evaluated as
\begin{equation} \label{c(z)}
c(z)=\log \Big| \frac{ \Theta_\chi( w(z)-\overline{w(z)} ) }{ w'(z) \Theta'_\chi(0) } \Big|.
\end{equation}
We remark that the function $c(z)$ is called the \emph{conformal radius} for the ER b.c., whereas it is called the \emph{domain constant} for the Dirichlet b.c., see e.g., \cite[Section 4]{MR2230350}. 

Differentiating \eqref{OPE_G}, we have
\begin{align}
	\label{OPE CG_ER}
J(\zeta)\Phi(z)
&= -\dfrac{1}{\zeta-z}+(J\odot \Phi)(z)+\dfrac{w^{\prime}(z)}{2}H_\chi (w(z)-\overline{w(z)} )-\frac{w^{\prime \prime}(z)}{2w^{\prime}(z)}+o(1),
\\
\label{OPE CC_ER}
J(\zeta)J(z) &= -\dfrac{1}{(\zeta-z)^2}+(J\odot J)(z)
	-\frac{1}{6}S_w(z)-w^{\prime}(z)^2 \Big( \frac{\zeta_r(\pi)}{\pi}-\frac{1}{2(r+\chi)} \Big)+o(1).
\end{align}
Here, $S_h$ is the Schwarzian derivative $S_h$ of a conformal map $h$, i.e., 
$$
S_h:=N'_h-\frac{1}{2}N_h^2, \qquad N_h=\frac{h''}{h'}.
$$
We sometimes use the notation $\sim$ for the singular part of the OPE. For instance, we have 	
\begin{equation} \label{sing GG GC}
J(\zeta)\Phi(z) \sim -\frac{1}{\zeta -z},\qquad J(\zeta)J(z)\sim -\frac{1}{(\zeta -z)^2}.
\end{equation}

\subsection{Conformal Fock space fields} \label{Conformal Subsec}

We treat a stress tensor as a Lie derivative operator to state Ward's identities. 
For this purpose, it is needed to consider Fock space fields in a planar domain (e.g., a doubly connected domain) as defined on a Riemann surface. 
The correlation functions of conformal Fock space fields depend on the choice of local charts at their nodes. 

By definition, a non-random conformal field $f$ on a Riemann surface is a smooth function $ ( f \|\phi ) : \phi U \to \C$ for each local chart $\phi$. 
We define a general conformal Fock space field as $X=\sum_{\alpha}f_\alpha X_\alpha$, where $f_\alpha$'s are non-random conformal fields and basic fields $X_\alpha$ are Wick's products of derivatives of GFF. 

We now list some basic non-random fields $f$ with specific transformation laws between $f=(f \|\phi)$ and $\ti{f}=(f \| \ti{\phi})$ for any two overlapping charts $\phi$ and $\ti{\phi}$. 
A non-random field $f$ is called a \emph{differential} of conformal dimension $ (\lambda, \lambda_*)$ if 
$$
f=(h')^\lambda (\overline{h'})^{\lambda_*} \ti{f} \circ h,
$$
where $h$ is the transition map between $\phi$ and $\ti{\phi}$, i.e.,
$h= \ti{\phi} \circ \phi^{-1} : \phi (U \cap \ti{U}) \to \ti{\phi} (U \cap \ti{U})$. 
By definition, pre-pre-Schwarzian forms (PPS-forms), pre-Schwarzian forms (PS-forms) and Schwazian forms (S-forms) of order $\mu$ are conformal Fock space fields with transformation laws
$$
f=\ti{f} \circ h+\mu \log h', \qquad f=h' \ti{f} \circ h+\mu N_h, \qquad f=(h')^2 \ti{f} \circ h+\mu S_h
$$
respectively. 
In general, PPS-forms of order $(\mu,\nu)$ satisfy 
$$
f=\ti{f} \circ h+\mu \log h'+\nu \log \overline{h'}. 
$$

\subsubsection{Lie derivative of conformal Fock space field}

Suppose a non-random smooth vector field $v$ is holomorphic in some open set $U\subset M$. 
For a conformal Fock space field $X$, the Lie derivative $\LL_v X$ in $U$ is defined as
$$
( \LL_v X \| \phi ) = \frac{d}{dt} \Big|_{t=0} (X \| \phi \circ \psi_{-t} ),
$$
where $\psi_t$ is a local flow of $v$, and $\phi$ is a given chart.
For example, we have the followings: 
\begin{itemize}
	\item $\LL_v X = ( v\pa+\bar{v}\bp +\lambda v'+ \lambda_* \overline{v'} ) X$ for a $(\lambda,\lambda_*)$-differential $X$;
	\smallskip 
	\item $ \LL_v X = ( v\pa+v' ) X +\mu v''$ for a PS-form $X$ of order $\mu$;
	\smallskip 
	\item $\LL_v X = ( v\pa+2v' ) X +\mu v'''$ for an S-form $X$ of order $\mu$, 
\end{itemize}
see	\cite[Proposition 4.1]{MR3052311}.
For reader's convenience, we list some basic properties of the Lie derivatives:
\begin{itemize}
	\item $\E[\LL_v X] =\LL_v \E[X]$;
	\smallskip 
	\item $\LL_v (\bar{X}) =\overline{( \LL_v X )}$;
	\smallskip 
	\item $\LL_v (\pa X) = \pa (\LL_v X)$;
	\smallskip 
	\item Leibniz's rule holds for tensor, Wick's, and OPE products.
\end{itemize}

We define the $\C$-linear part $\LL_v^+$ and the anti $\C$-linear part $\LL_v^-$ of $\LL_v$ by		
$$
\LL_v^+ :=\frac{\LL_v-i\LL_{iv}}{2},\qquad \LL_v^- :=\frac{\LL_v + i\LL_{iv}}{2}
$$
so that $\LL_v$ is decomposed as $\LL_v = \LL_v^+ + \LL_v^-$. 
Note that while $\LL_v$ depends $\R$-linearly on $v$, $\LL_v^+$ (resp., $\LL_v^-$) depends (resp., anti-) $\C$-linearly on $v$.

\subsubsection{Stress energy tensor}

Here we recall the definition of a \emph{stress tensor} and review its basic properties. See \cite[Sections 5.2~5.3, Lecture 7, and Appendix 11]{MR3052311} for more details. 

A stress tensor for a family of conformal Fock space fields is defined as a pair of holomorphic and anti-holomorphic quadratic differentials which represent the Lie derivative operators within correlations of fields in this family. 
In the physics literature, special forms of formulas are known as (conformal) Ward's identities. 
The existence of a stress tensor is a remarkable property of certain families of conformal Fock space fields. 
For instance, the OPE family of GFFs has a stress tensor in common. 

A Fock space field $X$ in $D_r$ is said to have a stress tensor $(A,\bar{A})$ (or $X \in \FF(A,\bar{A}) )$ if $A$ is a holomorphic quadratic differential and the residue form of Ward's identity 
\begin{equation}
\LL_v^+ X(z)=\frac{1}{2\pi i}\oint_{(z)}vA X(z), \qquad \LL_v^- X(z)=-\frac{1}{2\pi i}\oint_{(z)} \bar{v}\bar{A} X(z)
\end{equation}
holds on $D_{\textrm{hol}}(v)\cap U$ for all smooth vector field $v$, where $D_{\textrm{hol}}(v)$ is the the maximal open set where $v$ is holomorphic.
By \cite[Proposition 5.8]{MR3052311}, $\FF(A,\bar{A})$ is closed under $\ast_n$ product, in particular, under differentiations. 
If $X$ is a differential or a form, the residue form of Ward's identity holds for $X$ if and only if Ward's OPE for $X$ holds in some or every chart. 
For example, a $(\lambda ,\lambda_*)$-differential $X$ is in $\FF(A,\bar{A})$ if and only if the following singular OPEs hold in every/some chart:
\begin{equation} \label{eq: Ward4Diff}
	A(\zeta) X(z) \sim \frac{\lambda X(z)}{(\zeta -z )^2}+\frac{\pa X (z)}{\zeta -z}; \qquad A(\zeta) \bar{X} (z) \sim \frac{\overline{\lambda}_* \bar{X}(z)}{(\zeta - z )^2}+\frac{\pa \bar{X}(z)}{\zeta -z }.
\end{equation}
For any form $X$ of order $\mu$ belongs to $\FF(A,\bar{A})$ if and only if the following singular OPE holds in every/some chart:
\begin{align} \label{eq: Ward4PPS}
	A(\zeta)X(z)&\sim \dfrac{\mu}{(\zeta - z)^2}+\dfrac{\pa X(z)}{\zeta - z} &\textrm{for a PPS-form }X;\\
	A(\zeta)X(z)&\sim \dfrac{2\mu}{(\zeta - z)^3}+\dfrac{X(z)}{(\zeta - z )^2}+\dfrac{\pa X(z)}{\zeta - z} &\textrm{for a PS-form }X;\\
	A(\zeta)X(z)&\sim \dfrac{6\mu}{(\zeta - z)^4}+\dfrac{2X(z)}{(\zeta - z )^2}+\dfrac{\pa X(z)}{\zeta - z} &\textrm{for an S-form }X.
\end{align} 

\subsection{Central charge modification} \label{subsec_CCM}

In this subsection we introduce the central charge modification $\Phi_{(b)}$ of GFF, where $b$ is the real parameter related to the central charge $c$ as $c = 1-12b^2$. 
In the chordal/radial CFT in a simply connected domain with a marked boundary/interior point $q$, two neutrality conditions are imposed in the construction of OPE exponentials. 
They have Wick's part and the correlation part; Wick's part is Wick's exponential of the linear combination of bosonic fields, and the correlation part can be expressed in terms of Coulomb gas correlation functions. 
We require that this linear combination of bosonic fields should be a well-defined Fock space field. 
This requirement is called a neutrality condition on the coefficients or charges of linear combination. 
At the same time, conformal invariance of Coulomb gas correlation functions in a simply connected domain gives rise to the other neutrality condition. 
Because the neutrality conditions are different, we need to place the background charge at $q$ to reconcile these two neutrality conditions. 
In a doubly connected domain (or on a compact Riemann surface of genus one), no background charge is placed for $\Phi_{(b)}$. It is consistent with the neutrality condition ($\mathrm{NC}_0$) that the sum of background charges should vanish due to a version of Gauss-Bonnet theorem. In Subsection~\ref{Subsec b.c.m.} we introduce the background charge modification $\Phi_{\bfs\beta}$ of GFF for a background charge $\bfs\beta$ with $\mathrm{NC}_0$.

We express a stress tensor and the Virasoro field of $\Phi_{(b)}$ in terms of the current field $J_{(b)} := \pa \Phi_{(b)}$. 
From now on, the subscript $(0)$ is added to the notation of fields in the OPE family of the GFF, the algebra over $\C$ spanned by $1$ and the derivatives of GFF under OPE multiplication. 
For instance, $\Phi_{(0)}$ is the new notation for the GFF and $J_{(0)}=\pa \Phi_{(0)}$. 
First, we recall the definition of the Virasoro field. 
\begin{def*}
	A Fock space field $T$ is said to be \emph{Virasoro field} for $\FF(A,\bar{A})$ if
	\begin{itemize}
		\item $T \in \FF(A,\bar{A})$, and
		\item $T-A$ is a non-random meromorphic Schwarzian form.
	\end{itemize}
	The order of Schwarzian form $T-A$ is denoted by $\frac{1}{12}c$, where $c$ is called the \emph{central charge}.	
\end{def*}
For each real parameter $b$, we define
\begin{equation}\label{ccm_ER}
	\Phi \equiv \Phi_{(b)}:=\Phi_{(0)}-2b\arg w', \qquad
	J \equiv J_{(b)}:=J_{(0)}+ib\dfrac{w^{\prime \prime}}{w^{\prime}},
\end{equation}
where $w$ is a conformal map from $D_r$ onto $\mathcal C_r$. 

\begin{prop} \label{Virasoro2_ER}
	The field $\Phi_{(b)}$ has a stress tensor 
	$$ A_{(b)}:= -\frac{1}{2}J_{(0)}\odot J_{(0)}+(ib\pa-\jmath_{(b)})J_{(0)}, \qquad \jmath_{(b)}:=\E[J_{(b)}]=ib\dfrac{w^{\prime \prime}}{w^{\prime}}, 
	$$
	and its Virasoro field is given by
	\begin{align*}
		T_{(b)} &=-\dfrac{1}{2}J_{(b)} \ast J_{(b)}+ib\pa J_{(b)} 
		= 	A_{(b)}+\dfrac{c}{12}S_w+(w^{\prime})^2\Big( \dfrac{\zeta_r(\pi)}{2\pi}-\frac{1}{4(r+\chi)} \Big),
	\end{align*}	
	where $c = 1-12b^2$.
\end{prop}
\begin{proof}
	First, we prove the case $b=0$. 
	It is obvious that $A_{(0)}:=-\frac{1}{2}J_{(0)}\odot J_{(0)}$ is holomorphic quadratic differential. 
	We claim that $W=(A_{(0)},\bar{A}_{(0)})$ is a stress tensor for $\Phi_{(0)}$. 
	Since $\Phi_{(0)}$ is a scalar field or a $(0,0)$-differential, by \eqref{eq: Ward4Diff}, all we need to verify is the following (Ward's) OPE in the identity chart of cylinder $\mathcal C_r:$
	\begin{equation}\label{S0 WOPE}
		A_{(0)}(\zeta) \Phi_{(0)}(z) \sim \dfrac{J_{(0)}(z)}{\zeta-z}.
	\end{equation}
	This is clear from \eqref{OPE CG_ER} and Wick's calculus.
	 
	We now define the Virasoro field $T_{(0)}$ of $\Phi_{(0)}$ by $T_{(0)}:=-\frac{1}{2}J_{(0)}*J_{(0)}$, equivalently by the following OPE:
	\begin{equation}\label{T0 OPE}
		J_{(0)}(\zeta)J_{(0)}(z)=-\dfrac{1}{(\zeta-z)^2}-2\,T_{(0)}(z)+o(1). 
	\end{equation}
	Then it follows from \eqref{OPE CC_ER} that 
	$$
	T_{(0)} =A_{(0)}+\dfrac{1}{12}S_w+(w^{\prime})^2\Big( \dfrac{\zeta_r(\pi)}{2\pi}-\frac{1}{4(r+\chi)} \Big).
	$$
	
	Next, let us prove the general case $b \neq 0$. It is easy to see that $ib\pa J_{(0)}$ and $\jmath_{(b)} J_{(0)}$ satisfy the following transformation rules:	
	\begin{align*}
		ib \pa J_{(0)}= ibh'' \ti{J}_{(0)}\circ h+ib(h')^2 \pa \ti{J}_{(0)}\circ h; \qquad 
		\jmath_{(b)} J_{(0)}= ib\Big(\frac{h''}{h'} \Big)h' \ti{J}_{(0)} \circ h+(h')^2 (\ti{\jmath}_{(b)} \ti{J}_{(0)} )\circ h.
	\end{align*}
	Thus we readily see that $A_{(b)}$ is a holomorphic quadratic differential. 
	
	 Again, we claim that $W=(A_{(b)},\bar{A}_{(b)})$ is a stress tensor for $\Phi_{(b)}$. Since $\Phi_{(b)}$ is the real part of a PPS-form, by \eqref{eq: Ward4PPS}, it suffices to show the following (Ward's) OPE in the identity chart of cylinder $\mathcal C_r:$
	$$
	A_{(b)}(\zeta) \Phi_{(b)}(z)= \Big(A_{(0)}(\zeta)+ (ib\pa-\jmath_{(b)}) J_{(0)}(\zeta) \Big) \Phi_{(b)}(z) \sim \dfrac{ib}{(\zeta-z)^2}+\dfrac{J_{(b)}(z)}{\zeta-z}.
	$$
	The above OPE follows from \eqref{OPE CG_ER}, \eqref{OPE CC_ER} and \eqref{S0 WOPE}. Let 
	$
	T_{(b)} =-\frac{1}{2}J_{(b)} \ast J_{(b)}+ib\pa J_{(b)}.
	$
	Then by \eqref{T0 OPE}, we obtain
	\begin{align*}
		T_{(b)} &= A_{(b)}+\dfrac{1}{12}S_w+(w^{\prime})^2 \Big( \dfrac{\zeta_r(\pi)}{2\pi}-\frac{1}{4(r+\chi)} \Big) -\dfrac{1}{2}\jmath_{(b)}^2+ib \pa \jmath_{(b)}
		\\
		&= A_{(b)}+\dfrac{1-12b^2}{12}S_w+(w^{\prime})^2 \Big( \dfrac{\zeta_r(\pi)}{2\pi}-\frac{1}{4(r+\chi)} \Big),
	\end{align*}
	which completes the proof. 
\end{proof}

\subsection{Formal bosonic fields and neutrality condition} \label{formal bosonic}

In this subsection we introduce formal fields and their formal correlation functions. 
We use them to describe modified multi-vertex fields (OPE exponentials) in the next subsection. 
We first define the \emph{bi-variant} field $\Phi^+ \equiv \Phi^+_{(b)}$ by 
$$
\Phi^+_{(b)}(z,z_0):=\Big \{ \Phi_{(b)}^{+}(\gamma)= \int_{\gamma}J_{(b)} (\zeta) \, d\zeta \Big \} =\Phi^+_{(0)}(z,z_0)+ib\log \dfrac{w'(z)}{w'(z_0)},
$$
where $\gamma$ is a curve from $z_0$ to $z$ and $w : D_r \to \CC_r$ is a conformal map. 
Note that the values of $\Phi^+$ are multivalued functionals. 

Define the complex Green's function in $D_r$ as 
$$
G^{\chi+}_{D_r}(z,z_1):=\dfrac{1}{2}\log\Big(\dfrac{\Theta_\chi(r,w(z)-\overline{w(z_1)})}{\Theta_\chi(r,w(z)-w(z_1))} \Big).
$$
We remark that for the ER case when $\chi=\infty$, the complex Green's function plays an important role in the theory of conformal mappings between canonical multiply connected domains, see e.g., \cite{MR2241034}. 
Integrating \eqref{2pt cor CG_ER} over a curve from $z_0$ to $z$, we obtain 
$$
\E [\Phi^+_{(0)}(z,z_0)\Phi_{(0)}(z_1) ]=2 (G^{\chi+}_{D_r}(z,z_1)-G^{\chi+}_{D_r}(z_0,z_1) ).
$$

Now we introduce {\it{formal (1-point) fields}} $\Phi_{(0)}^{\pm}$ in $D_r$ as centered Gaussian formal fields with (formal) 2-point correlations:
\begin{align}
	\label{formal ++ cor}	
	\E [\Phiplus_{(0)}(z_1)\Phiplus_{(0)}(z_2) ]&=	-\log \Theta_\chi (r,w(z_1)-w(z_2) );
	\\
	\label{formal +- cor}
	\E [\Phiplus_{(0)}(z_1)\Phiminus_{(0)}(z_2) ]&=+\log \Theta_\chi (r,w(z_1)-\overline{w(z_2)} ).
\end{align}
Formal fields $\Phi_{(0)}^{\pm}$ satisfy $\Phi_{(0)}=\Phi_{(0)}^++\Phi_{(0)}^-$ within correlations, e.g.,
$$
\E \,\Phi_{(0)}(z_1)\Phi_{(0)}(z_2) = \E (\Phi_{(0)}^+(z_1)+\Phi_{(0)}^-(z_1) ) (\Phi_{(0)}^+(z_2)+\Phi_{(0)}^-(z_2) ),
$$ 
where $\E$ in the left-hand side stands for correlation while we use formal correlations in the right-hand side. We define $\Phi_{(b)}^{\pm}$ by 
\begin{equation} \label{formal ccm}
	\Phiplus_{(b)}=\Phiplus_{(0)}+ib\log w^{\prime}, \qquad \Phiminus_{(b)}=\Phiminus_{(0)}-ib\log \overline{w'}.
\end{equation}
By definition, the following formal rule holds: 
$$
\Phiplus_{(0)}(z,z_0)=\Phiplus_{(0)}(z)-\Phiplus_{(0)}(z_0).
$$
The formal \emph{dual boson} $\ti{\Phi}_{(0)}$ is defined by 
\begin{equation} \label{dual boson}
\ti{\Phi}_{(0)}=-i (\Phiplus_{(0)}-\Phiminus_{(0)} ), \qquad \ti{\Phi}_{(0)}(z,z_0)=\ti{\Phi}_{(0)}(z)-\ti{\Phi}_{(0)}(z_0).
\end{equation}
Note that we have the following relations
$$ 
2\, \Phiplus_{(0)}=\Phi_{(0)}+i \ti{\Phi}_{(0)}, \qquad 2\, \Phiminus_{(0)}=\Phi_{(0)}-i \ti{\Phi}_{(0)}.
$$

For a finite set of distinct points $\{z_j	\}_{j=1}^n$ in $\bar{D}_r$, let us denote $\bfs{\sigma} =\sum \sigma_j \cdot z_j$, where $\sigma_j$ is a ``charge" at each $z_j$. 
We also consider $\bfs{\sigma}$ as a divisor, i.e., a function $\bfs{\sigma}: \bar{D}_r \to \R$ which takes the value $0$ at all points except for finitely many points, and $\bfs{\sigma}(z_j)=\sigma_j$. 
Sometimes it is convenient to consider $\bfs{\sigma}$ as an atomic measure $\bfs{\sigma}=\sum \sigma_j \cdot \delta_{z_j}$. 
For example, $\int \bfs{\sigma}=\sum \sigma_j$. 

For a double divisor $(\bfs{\sigma},\bfs{\sigma_*})$, we define the formal bosonic field $\Phi\,[\bfs{{\sigma}},\bfs{{\sigma_*}}]$ by
\begin{equation} \label{lin comb GFF}
	\Phi\,[\bfs{{\sigma}},\bfs{{\sigma_*}}] := \sum \sigma_j\Phiplus(z_j)-\sigma_{*j}\Phiminus(z_j), \qquad \Phi^{\pm} \equiv \Phi^{\pm}_{(0)}. 
\end{equation}
In general, it is not a well-defined Fock space field. 
The proposition below shows that the neutrality condition
\begin{equation}\label{NC0_ER} 
	\int \bfs{\sigma}+\bfs{\sigma_*}=0
\end{equation}
guarantees well-definedness of $\Phi[\bfs{{\sigma}},\bfs{{\sigma_*}}] $ as a Fock space field. 
For example, bi-variant fields satisfy the neutrality condition.

\begin{prop}
	If a double divisor $(\bfs{\sigma},\bfs{\sigma_*})$ satisfies (\ref{NC0_ER}), then the formal bosonic field $\Phi[\bfs{{\sigma}},\bfs{{\sigma_*}}]$ can be represented as a linear combination of well-defined Fock space fields.
\end{prop}
\begin{proof}
	Let us choose any point $z_0 \in D_r$. Then
	$$
	\Phi[\bfs{{\sigma}},\bfs{{\sigma_*}}]=\Phiplus(z_0) \int \bfs{\sigma} -\Phiminus(z_0) \int \bfs{\sigma_*}+\sum\sigma_j\Phiplus(z_j,z_0)-\sigma_{*j}\Phiminus(z_j,z_0).
	$$
	Under the neutrality condition, the first two terms on the right-hand side become the Fock space correlation functional:
	$$
	\Phiplus(z_0) \int \bfs{\sigma} -\Phiminus(z_0) \int \bfs{\sigma_*}=\Phi(z_0)\int\bfs{\sigma}.
	$$
\end{proof}

\subsection{Coulomb gas correlation function} \label{Vertex fields}

In this subsection we construct modified multi-vertex fields and compute their correlation functions by the \emph{Coulomb gas formalism}.

Given divisors $\bfs{\sigma}=\sum \sigma_j \cdot z_j$, $\bfs{{\tau}}=\sum \tau_k \cdot \zeta_k$ for $z_j, \zeta_k \in D_r$, set 
$$
\Phi[\bfs{\sigma}]:=\sum \sigma_j \, \Phi(z_j), \qquad \Phi[\bfs{\tau}]:=\sum \tau_k \, \Phi(\zeta_k),
$$
where $\Phi \equiv \Phi_{(0)}$.
Define 
$$
\E \Phi[\bfs{\sigma}] *\Phi[\bfs{\tau}]:=\sum_{z_j \neq \zeta_k} \sigma_j \tau_k \,\E \Phi(z_j)\Phi(\zeta_k)+\sum_{z_j = \zeta_k} \sigma_j \tau_k \,\E \Phi*\Phi(z_j).
$$
In particular, set 
\begin{equation*} 
	c[ \bfs{\sigma} ]:= \frac{1}{2}\, \E \Phi[\bfs{\sigma}] * \Phi[\bfs{\sigma}]
\end{equation*}
and define the (non-chiral) \emph{Coulomb gas correlation function} $C_{(b)}$ as 
\begin{equation*} 
	C_{(b)}[ \bfs{\sigma} ]:= e^{-c[ \bfs{\sigma} ]} \prod (w'_j)^{-b\sigma_{j}} (\overline{w'_j})^{b\sigma_{j}}, 
\end{equation*}
where $w_j:=w(z_j)$. By \eqref{2pt cor_ER} and \eqref{c(z)}, if all $z_j$'s are in the bulk, $	C_{(b)}[ \bfs{\sigma}]$ is $( \tfrac12 \sigma_j^2-b\sigma_{j} , \tfrac 12 \sigma_j^2+b\sigma_{j} )$-differential at $z_j$, and its evaluation in the identity chart of cylinder $\CC_r$ is given as 
\begin{align*}
	C_{(b)}[ \bfs{\sigma} ]&=\Theta'_\chi(0)^{\sum \sigma_j^2} \prod \Theta_\chi(z_j-\bar{z}_j)^{ -\sigma_j^2 } \prod_{j < k} \Big|\frac{\Theta_\chi(z_j-z_k)} {\Theta_\chi(z_j-\bar{z}_k) } \Big|^{2\sigma_j \sigma_k} .
\end{align*}
We remark that the terminology ``Coulomb gas'' comes from the resemblance of the 2-point function $\E \Phi(z_1)\Phi(z_2)$ with the electric potential energy between two unit charges in a planar domain (or on a compact Riemann surface). 
For the genus one case we refer to \cite{MR4029833,forrester2006particles} and references therein. 

By definition, a non-chiral vertex field $\mathcal{V}_{(0)}[\bfs{\sigma}]$ is given by the OPE exponential 
\begin{equation*}
	\mathcal{V}_{(0)}[\bfs{\sigma}]:=e^{*i\Phi[\bfs{\sigma}]}\equiv \sum_{n=0}^{\infty} \frac{i^n \Phi^{\ast n}[\bfs{\sigma}]}{n!}.
\end{equation*}
Then it is easy to show that 
\begin{equation*}
	\mathcal{V}_{(0)}[\bfs{\sigma}]=
	C_{(0)}[\bfs{\sigma}]\mathcal{V}^{\odot}[\bfs{\sigma}], \qquad \mathcal{V}^{\odot}[\bfs{\sigma}]= e^{\odot i\Phi[\bfs{\sigma}]}.
\end{equation*}

Let us denote by $(\bfs{{\sigma}},\bfs{{\sigma_*}})=(\sum \sigma_j \cdot z_j ,\sum \sigma_{*j} \cdot z_j )$ a double divisor satisfying the neutrality condition $(\mathrm{NC}_0)$. We define (chiral) Coulomb gas correlation function $C_{(b)}{[\bfs{\sigma},\bfs{\sigma_*}]}$ as 
\begin{align}
	\begin{split}
		\label{Coulomb ftn}
	C_{(b)}{[\bfs{\sigma},\bfs{\sigma_*}]}:&=\Theta'_\chi(0)^{\sum \frac{1}{2}( \sigma_j^2+\sigma_{*j}^2 ) } \prod \Theta_\chi(w_{j}-\bar{w}_{j})^{\sigma_{j} \sigma_{*j}} \prod (w'_j)^{\lambda_j} ( \overline{w'_j} ) ^{\lambda_{*j}} 
		\\
		&\times\prod_{j < k} \Theta_\chi(w_j-w_k)^{\sigma_j\sigma_k}\Theta_\chi(\bar{w}_{j}-w_k)^{\sigma_{*j}\sigma_k} \Theta_\chi(w_j-\bar{w}_k)^{\sigma_j\sigma_{*k}}\Theta_\chi(\bar{w}_j-\bar{w}_k)^{\sigma_{*j}\sigma_{*k}},
	\end{split}
\end{align}
where the conformal dimensions are given as 
\begin{equation} \label{Cou dimension}
	(\lambda_j,\lambda_{*j}) = ( \lambda(\sigma_j),\lambda(\sigma_{*j}) ), \qquad \lambda(x):=\tfrac12 x^2-b\, x.
\end{equation}
Notice here that the non-chiral Coulomb gas correlation function is related to chiral one as 
$ C_{(b)}[ \bfs{\sigma} ]=C_{(b)}{[\bfs{\sigma},-\bfs{\sigma}]}$. 

By definition, the \emph{multi-vertex field} $\OO{[\bfs{\sigma},\bfs{\sigma_*}]}\equiv \OO_{(b)}{[\bfs{\sigma},\bfs{\sigma_*}]}$ is given by
\begin{equation} \label{mvf_def}
	\OO_{(b)}{[\bfs{\sigma},\bfs{\sigma_*}]} := C_{(b)}{[\bfs{\sigma},\bfs{\sigma_*}]} V^\odot[\bfs\sigma,\bfs\sigma_*], \qquad V^\odot[\bfs\sigma,\bfs\sigma_*]=e^{\odot i\Phi_{(0)}[\bfs{{\sigma}},\bfs{{\sigma_*}}]}.
\end{equation}	
Then the field $\OO{[\bfs{\sigma},\bfs{\sigma_*}]}$ is a differential with conformal dimension
$(h_j,h_{*j})=(\lambda(\sigma_j),\lambda(\sigma_{*j}))$ at $z_j$. 
In particular, we have 
\begin{equation} \label{OPE exponential}
\OO_{(b)}{[\bfs{\sigma}]}:=\OO_{(b)}{[\bfs{\sigma},-\bfs{\sigma}]}=C_{(b)}[\bfs{\sigma}] e^{ \odot i \Phi_{(b)}[\bfs{\sigma}] }=e^{ \ast i \Phi_{(b)}[\bfs{\sigma}]}.
\end{equation}

Now we show that $\OO{[\bfs{\sigma},\bfs{\sigma_*}]}$ satisfies Ward's OPE. 

\begin{prop} \label{WOPE:V_ER}
	Under the neutrality condition, $\OO(\bfs{{z}})\equiv \OO{[\bfs{\sigma},\bfs{\sigma_*}]}(\bfs{{z}})$ satisfies Ward's OPE: as $\zeta \to z_j \in D_r$,
	$$
	T(\zeta)\OO(\bfs{{z}}) \sim h_j\frac{\OO(\bfs{z})}{(\zeta-z_j)^2}+\frac{\pa_j \OO(\bfs{z})}{\zeta-z_j}, \qquad
	T(\zeta)\overline{\OO}(\bfs{z}) \sim \bar{h}_{*j}\frac{\overline{\OO}(\bfs{z})}{(\zeta-z_j)^2}+\frac{\pa_j \overline{\OO}(\bfs{z})}{\zeta-z_j}.
	$$ 
\end{prop}
\begin{proof}
	Since vertex fields are differentials, it is enough to prove the proposition in the identity chart of $\CC_r$. Recall that in the identity chart of $\CC_r$, 
	$$
	J_{(b)}=J_{(0)}, \qquad T_{(b)}=-\dfrac{1}{2}J_{(0)}\odot J_{(0)} +ib\,\pa J_{(0)} +\Big( \dfrac{\zeta_r(\pi)}{2\pi}-\frac{1}{4(r+\chi)} \Big).
	$$ 
 By Wick's calculus, we have 
 $$
 \pa_j \OO=
 \dfrac{\pa_j C_{(b)}[\bfs{\sigma},\bfs{\sigma_*}] }{C_{(b)}[\bfs{\sigma},\bfs{\sigma_*}]} \OO + i \sigma_{j} J_{(b)}(z_j) \odot \OO.
 $$
	
	First, let us show Ward's OPE in the case $b=0$. Let us denote
	\begin{align*}
	 F(\zeta,\bfs{z}):&=\E J_{(0)}(\zeta) \sum \Big(	i\sigma_j\Phiplus_{(0)}(z_j)-i\sigma_{*j}\Phiminus_{(0)}(z_j)	\Big) 
	=-\dfrac{i}{2} \sum \Big(\sigma_j H_\chi(r,\zeta-z_j)+\sigma_{*j} H_\chi(r,\zeta-\bar{z}_j) \Big).
	\end{align*}
	Then by \eqref{h-zero},\eqref{Coulomb ftn} and \eqref{Cou dimension} as $\zeta \to z_j$, we have 
	$$
	F(\zeta,\bfs{z}) \sim 
	 -\dfrac{i\sigma_j}{\zeta-z_j}, \qquad 	 F(\zeta,\bfs{z})^2 \sim -\frac{2 h_j}{(\zeta-z_j)^2}- \frac{\pa_j C_{(b)}[\bfs{\sigma},\bfs{\sigma_*}] }{C_{(b)}[\bfs{\sigma},\bfs{\sigma_*}]} \frac{2}{\zeta-z_j}.
	$$
	On the other hand, by Wick's formula, we have 
	$$ T_{(0)}(\zeta)\OO(\bfs{z})=T_{(0)}(\zeta)\odot\OO(\bfs{z})- F(\zeta,\bfs{z}) J_{(0)}(\zeta)\odot \OO(\bfs{z})-\tfrac{1}{2} F(\zeta,\bfs{z})^2 \OO(\bfs{z}).
	$$
 Combining all the above equations, our assertion in the case $b=0$ follows. For general $b \in \R$, we only need to check 
	$$
	ib\pa J_{(0)}(\zeta) \OO(\bfs{z}) \sim -\sigma_j b \frac{\OO (\bfs{z})}{(\zeta-z_j)^2}, \qquad \textrm{as } \zeta \to z_j.	
	$$
	Again, this easily follows from Wick's calculus. We leave it to the reader to verify the proposition for $\overline{\mathcal O}$. 
\end{proof}

\subsection{Background charge modifications} \label{Subsec b.c.m.}
Given marked points $q_k$, we call a divisor $\bfs{\beta}= \sum_k \beta_k \cdot q_k$ satisfying the neutrality condition $\int \bfs{\beta}=0$ a \emph{background charge}. 
By definition, a \emph{background charge modification} $\Phi_{ \bfs{\beta} } \equiv \Phi_{ \bfs{\beta},(b) }$ of the GFF $\Phi \equiv \Phi_{(b)}$ is given as 
\begin{equation} \label{one pt GFF beta} 
\Phi_{ \bfs{\beta},(b) }:= \Phi_{(0)}+ \varphi_{ \bfs{\beta} }, \qquad \varphi_{ \bfs{\beta} }(z):=-2b\arg w'_z+\sum \beta_k \arg \Big\{ \Theta_\chi(w_k-w_z) \Theta_\chi(\bar{w}_k-w_z) \Big\}.
\end{equation}
Here $w_z=w(z)$ and $w_k=w(q_k)$. By Wick's calculus, the non-random harmonic function $\varphi_{ \bfs{\beta} }$ can also be recognized as 
\begin{equation} \label{insertion bcm}
\varphi_{ \bfs{\beta} } (z)= \E \Phi_{ \bfs{\beta} }(z)= \E [ e^{ \odot i \Phi[ \bfs{\beta}/2, \bfs{\beta}/2 ]} \Phi (z)]. 
\end{equation}

The current field $J_{ \bfs{\beta} }$ is defined in a natural way as 
\begin{equation} 
J_{ \bfs{\beta},(b) } \equiv J_{ \bfs{\beta} }:= \pa \Phi_{ \bfs{\beta} }=J_{(0)}+\jmath_{ \bfs{\beta} }, 
\end{equation}
where 
\begin{equation}\label{jmath beta} 
\jmath_{ \bfs{\beta} }:=\pa \varphi_{ \bfs{\beta} }=ib\dfrac{w''_z}{w'_z}+\dfrac{iw'_z}{4}\sum \beta_k \Big\{ H_\chi(w_k-w_z)+H_\chi(\bar{w}_k-w_z) \Big\}.
\end{equation}
Notice that for the trivial divisor $\bfs{\beta}=\bfs{0}$, we have
$\Phi_{ \bfs{0} }=\Phi_{(b)}$ and $J_{ \bfs{0} }=J_{(b)}$. 

We now generalize Proposition~\ref{Virasoro2_ER}.
\begin{prop} \label{AT beta}
	The field $\Phi_{ \bfs{\beta} }$ has a stress tensor 
	$$ 
	A_{ \bfs{\beta} }:= A_{ (0) }+\Big(ib\pa-\jmath_{ \bfs{\beta} } \Big)J_{(0)}, 
	$$
	and its Virasoro field is given as
	\begin{align*}
	T_{ \bfs{\beta} } &:=-\dfrac{1}{2}J_{ \bfs{\beta} } \ast J_{ \bfs{\beta} }+ib\,\pa J_{ \bfs{\beta} }. 
	\end{align*}
	In particular, $\E T_{ \bfs{\beta} }(z)$ is evaluated in the identity chart of $\CC_r$ as 
	\begin{align*}
	\begin{split}
	\E T_{ \bfs{\beta} }(z)
	&=\dfrac{1}{32}\Big[ \sum \beta_k \Big\{ H_\chi(r,q_k-z)+H_\chi(r,\bar{q}_k-z) \Big\} \Big]^2
	\\
	&+\dfrac{b}{4} \, \sum \beta_k \Big\{ H'_\chi(r,q_k-z)+H'_\chi(r,\bar{q}_k-z) \Big\}+ \dfrac{\zeta_r(\pi)}{2\pi}-\frac{1}{4(r+\chi)}. 
	\end{split}
	\end{align*}
\end{prop}
\begin{proof}
	Throughout the proof, we denote by $\phi$ the local chart and $h= \ti{\phi} \circ \phi^{-1}$ the transition map between $\phi$ and a chart $\tilde{\phi}$ overlapping with $\phi$. We also use the tilde notation for conformal fields evaluated to the chart $\tilde{\phi}$, e.g., $\tilde{J}_{(0)}=( J_{(0)}\| \tilde{\phi} )$.
	
	We claim that $A_{ \bfs{\beta} }$ is a $(2,0)$-differential and Ward's OPE
	\begin{equation} \label{Ward OPE beta}
	A_{\bfs{\beta} }(\zeta) \Phi_{ \bfs{\beta} }(z) \sim \frac{ib}{(\zeta-z)^2}+\frac{J_{ \bfs{\beta} }(z) }{\zeta-z}
	\end{equation}
	holds. Then the first part of the proposition immediately follows from \eqref{eq: Ward4PPS}. 
	
	Recall that
	$\pa J_{(0)}= h'' \ti{J}_{(0)}\circ h+(h')^2 \pa \ti{J}_{(0)}\circ h$. Also by \eqref{jmath beta}, 
	\begin{equation} \label{trans jmath beta}
	\jmath_{ \bfs{\beta} }= h' \, \tilde{ \jmath}_{ \bfs{\beta} } \circ h+ ib\frac{h''}{h'}.
	\end{equation}
	Combining the above equations, one can observe that $(ib\pa-\jmath_{ \bfs{\beta} })J_{(0)}$ is a $(2,0)$-differential. 
	Note that 
	$$
	A_{\bfs{\beta} }(\zeta) \Phi_{ \bfs{\beta} }(z) \sim A_{(0)}(\zeta)\Phi_{(0)}(z)+\Big(ib\pa-\jmath_{ \bfs{\beta} } \Big) J_{(0)}(\zeta)\Phi_{ (0) }(z).
	$$
	Then \eqref{Ward OPE beta} follows from \eqref{sing GG GC} and Proposition~\ref{Virasoro2_ER}.
	
	Next, we prove the second part of the proposition. By Proposition~\ref{Virasoro2_ER}, we have 
	\begin{align} \label{T beta A beta rel}
	T_{ \bfs{\beta} }=A_{ \bfs{\beta} }+\E\,T_{(0)}-\tfrac12 \jmath_{ \bfs{\beta} }^2+ib \pa \jmath_{ \bfs{\beta} },
	\end{align}
	and $\E\,T_{(0)}$ is an S-form of order $1/12$. Moreover, by \eqref{trans jmath beta}, the term $-\frac12 \jmath_{ \bfs{\beta} }^2+ib \pa \jmath_{ \bfs{\beta} }$ is an S-form of order $-b^2$, which completes the proof. 
\end{proof}

We define holomorhpic (resp., anti-holomorphic) function $\varphi^+_{ \bfs{\beta} }$ (resp., $\varphi^-_{ \bfs{\beta} }$) as 
\begin{align}
\begin{split}
\varphi^+_{ \bfs{\beta} }&:=+ib\log w'_z-\dfrac{i}{2}\sum \beta_k \log \Big\{ \Theta_\chi(w_k-w_z) \Theta_\chi(\bar{w}_k-w_z) \Big\},
\\
\varphi^-_{ \bfs{\beta} }&:=
-ib\log \overline{w_z'}+\dfrac{i}{2}\sum \beta_k \log \Big\{ \Theta_\chi(\bar{w}_k-\bar{w}_z) \Theta_\chi(w_k-\bar{w}_z) \Big\}.
\end{split}
\end{align}
The chiral and anti-chiral parts of $\Phi_{ \bfs{\beta} }$ are defined by
\begin{equation}
\Phi_{ \bfs{\beta} }^+:=\Phi_{ (0) }^++\varphi^+_{ \bfs{\beta} }, \qquad \Phi_{ \bfs{\beta} }^-:=\Phi_{ (0) }^-+\varphi^-_{ \bfs{\beta} }.
\end{equation}
Notice that by definition, $\Phi_{ \bfs{\beta} }= \Phi_{ \bfs{\beta} }^++\Phi_{ \bfs{\beta} }^-$. 

For a double divisor $(\bfs{\sigma},\bfs{\sigma_*})=( \sum \sigma_j \cdot z_j, \sum \sigma_{*j} \cdot z_j )$ satisfying $\int (\bfs{\sigma}+\bfs{\sigma_*})=0$, we write 
\begin{equation}
\varphi_{ \bfs{\beta} }\,[\bfs{{\sigma}},\bfs{{\sigma_*}}]:= \sum \sigma_j \, \varphi_{ \bfs{\beta} }^+-\sigma_{*j} \, \varphi_{ \bfs{\beta} }^- .
\end{equation}
We define the formal bosonic field 
$\Phi_{ \bfs{\beta} }\,[\bfs{{\sigma}},\bfs{{\sigma_*}}]$ associated with the background charge $\bfs{\beta}$ as 
\begin{align}
\begin{split}
\Phi_{\bfs{\beta}}\,[\bfs{{\sigma}},\bfs{{\sigma_*}}] :=\sum \sigma_j \Phi_{ \bfs{\beta} }^+(z_j) -\sigma_{*j} \Phi_{ \bfs{\beta} }^-(z_j) 
=\Phi_{(0)}\,[\bfs{{\sigma}},\bfs{{\sigma_*}}]+ \varphi_{ \bfs{\beta} }\,[\bfs{{\sigma}},\bfs{{\sigma_*}}].
\end{split}
\end{align}

The modified multi-vertex field $\OO_{ \bfs{\beta} }[\bfs{{\sigma}},\bfs{{\sigma_*}}]$ is defined by
\begin{equation} \label{vertex beta}
\OO_{\bfs{\beta}}{[\bfs{\sigma},\bfs{\sigma_*}]} := \frac{ C_{(b)}{[\bfs{\sigma}+\bfs{\beta}/2,\bfs{\sigma_*}+\bfs{\beta}/2]} }{ C_{(b)}{[\bfs{\beta}/2,\bfs{\beta}/2]} } V^\odot[\bfs\sigma,\bfs\sigma_*].
\end{equation}
In particular, if $\bfs{\sigma_*}=-\bfs{\sigma}$, we simply omit $\bfs{\sigma_*}$ in the notations, e.g., 
\begin{equation}
\varphi_{ \bfs{\beta} }\,[\bfs{{\sigma}}]:=\sum \sigma_j \varphi_{ \bfs{\beta} }, \qquad \Phi_{\bfs{\beta}}\,[\bfs{{\sigma}}]:=\sum \sigma_{j}\, \Phi_{ \bfs{\beta} }(z_j).
\end{equation}

The modified multi-vertex fields are called the OPE exponentials of GFF.
The reason for this terminology is clear due to the following proposition.

\begin{prop} \label{V=O}
	We have
	$\OO_{\bfs{\beta}}{[\bfs{\sigma}]}=e^{ \ast i \Phi_{ \bfs{\beta} } [\bfs{\sigma}] }$. 
\end{prop}
\begin{proof}
	By \cite[Lemma 5.2]{KM17} and \eqref{OPE exponential}, 
	\begin{align*}
	e^{ \ast i \Phi_{ \bfs{\beta} } [\bfs{\sigma}] }=e^{ i \varphi_{ \bfs{\beta} } [\bfs{\sigma}] } e^{ \ast i \Phi_{ (0) } [\bfs{\sigma}] }
	= C_{(0)}[\bfs{\sigma}] e^{ i \varphi_{ \bfs{\beta} } [\bfs{\sigma}] } e^{ \odot i \Phi_{(0)}[\bfs{\sigma}] }.
	\end{align*}
	Therefore all we need to show is 
	\begin{equation}\label{exp Cou iden}
	C_{(0)}[ \bfs{\sigma},\bfs{\sigma_{*}} ] e^{ i \varphi_{ \bfs{\beta} } [\bfs{\sigma},\bfs{\sigma_{*}}] } = \frac{ C_{(b)}{[\bfs{\sigma}+\bfs{\beta}/2,\bfs{\sigma}_{*}+\bfs{\beta}/2]} }{ C_{(b)}{[\bfs{\beta}/2,\bfs{\beta}/2]} }. 
	\end{equation} 

Observe that $e^{i\varphi_{ \bfs{\beta} }\,[\bfs{{\sigma}},\bfs{{\sigma_*}}]}$ is a scalar field with respect to each $q_k$ and is a differential with dimension $(-b\sigma_{j},-b\sigma_{*j})$ at each $z_j$. 
	Therefore by \eqref{Cou dimension}, conformal dimensions of both sides of \eqref{exp Cou iden} are identical. Thus it is enough to show \eqref{exp Cou iden} in the identity chart of $\CC_r$. 
 By construction, 
	\begin{align*}
	\begin{split}
	e^{i\varphi_{ \bfs{\beta} }\,[\bfs{{\sigma}},\bfs{{\sigma_*}}]}&= \prod_{j < k} \Big( \Theta_\chi(r,q_k-z_j) \Theta_\chi(r,\bar{q}_k-z_j) \Big)^{\frac{ \sigma_{j} \beta_k}{2} } \Big( \Theta_\chi(r,\bar{q}_k-\bar{z}_j) \Theta_\chi(r,q_k-\bar{z}_j) \Big)^{\frac{ \sigma_{*j} \beta_k}{2} }.
	\end{split}
	\end{align*}
Now proposition follows from \eqref{Coulomb ftn}.
\end{proof}

We now generalize Proposition~\ref{WOPE:V_ER} for a non-trivial background charge $\bfs\beta$.
\begin{prop}\label{WOPE:V}
	The vertex field $\OO_{\bfs{\beta}}{[\bfs{\sigma},\bfs{\sigma_*}]}$ satisfies Ward's OPE
	$$
	T_{\bfs{\beta}}(\zeta)\OO_{\bfs{\beta}}{[\bfs{\sigma},\bfs{\sigma_*}]} \sim \lambda_j\frac{ \OO_{\bfs{\beta}}{[\bfs{\sigma},\bfs{\sigma_*}]} }{(\zeta-z_j)^2}+\frac{\pa_{z_j} \OO_{\bfs{\beta}}{[\bfs{\sigma},\bfs{\sigma_*}]}}{\zeta-z_j}, \qquad\text{as } \zeta \to z_j,
	$$
	and similar OPE holds for $\overline{\OO}_{\bfs{\beta}}{[\bfs{\sigma},\bfs{\sigma_*}]} $. 
\end{prop}
\begin{proof}
We write $\OO_{\bfs{\beta}}\equiv \OO_{\bfs{\beta}}{[\bfs{\sigma},\bfs{\sigma_*}]}$ and $\OO_{(b)}\equiv \OO_{(b)}{[\bfs{\sigma},\bfs{\sigma_*}]}$ throughout the proof. 
	Since $\OO_{\bfs{\beta}}$ is a differential, it suffices to show the proposition in the identity chart of $\CC_r$. By combining Propositions~\ref{Virasoro2_ER} and ~\ref{AT beta}, we have 
	\begin{align*}
	T_{\bfs{\beta}}=T_{(b)}&- ( \jmath_{ \bfs{\beta} }-\jmath_{(b)} )J_{(0)}
	-\frac12 ( \jmath_{ \bfs{\beta} }^2-\jmath_{(b)}^2 ) +ib\pa ( \jmath_{ \bfs{\beta} }-\jmath_{(b)} ).
	\end{align*}
	Then it follows from Proposition~\ref{WOPE:V_ER} and $\OO_{\bfs{\beta}}=(\E \OO_{\bfs{\beta}}/\E \OO_{(b)}) \OO_{(b)}$ that 
	\begin{align*}
	T_{\bfs{\beta}}(\zeta)\OO_{\bfs{\beta}} &\sim \frac{\E \OO_{\bfs{\beta}} }{\E \OO_{(b)} } 	T_{(b)}(\zeta)\OO_{(b)}-( \jmath_{ \bfs{\beta} }-\jmath_{(b)} )J_{(0)}(\zeta)\OO_{ \bfs{\beta} }
	\\
	&\sim \frac{ \lambda_j \, \OO_{\bfs{\beta}} }{(\zeta-z_j)^2}+\frac{\E \OO_{\bfs{\beta}} }{\E \OO_{(b)} } \frac{\pa_{z_j} \OO_{(b)}}{\zeta-z_j}-( \jmath_{ \bfs{\beta} }-\jmath_{(b)} )J_{(0)}(\zeta)\OO_{ \bfs{\beta} }
	\\
	&\sim \frac{ \lambda_j \, \OO_{\bfs{\beta}} }{(\zeta-z_j)^2}+ \frac{\pa_{z_j} \OO_{ \bfs{\beta} } }{\zeta-z_j}-\frac{ \pa_{z_j}( \E \OO_{ \bfs{\beta} }/\E \OO_{(b)} ) }{ \E \OO_{ \bfs{\beta} }/\E \OO_{(b)} }\frac{ \OO_{ \bfs{\beta} } }{\zeta-z_j}-( \jmath_{ \bfs{\beta} }-\jmath_{(b)} )J_{(0)}(\zeta)\OO_{ \bfs{\beta} }.
	\end{align*}
	
	On the other hand, by Wick's calculus together with \eqref{formal ++ cor},\eqref{formal +- cor} and \eqref{h-zero}, we have 
	\begin{align} \label{J O sing}
	J_{(0)}(\zeta)\OO_{ \bfs{\beta} } &\sim \OO_{ \bfs{\beta} } \, \E J_{(0)}(\zeta) \sum (	i\sigma_j\Phiplus_{(0)}(z_j)-i\sigma_{*j}\Phiminus_{(0)}(z_j)	)
	\sim -i \sigma_j \, \frac{ \OO_{\bfs{\beta}} }{ \zeta-z_j }.
	\end{align}
	Moreover by \eqref{Coulomb ftn} and \eqref{jmath beta}, we have
	$$
	\frac{\pa_{z_j} \E \OO_{ \bfs{\beta} } }{\E \OO_{ \bfs{\beta} }}-\frac{\pa_{z_j} \E \OO_{(b)}}{\E \OO_{(b)}} =i\sigma_j ( \jmath_{ \bfs{\beta} }-\jmath_{(b)} ).
	$$
	This completes the proof. 
\end{proof}

\section{Eguchi-Ooguri and Ward's equations} \label{WARD Section}

In \cite{MR873031}, Eguchi and Ooguri stated Ward's equation for tensor product $\XX$ of primary fields (differentials in the OPE family) on a complex torus of genus one in terms of linear partial differential operators  with respect to the nodes of fields and the modular parameter $\tau$. 
They explained the so-called Teichm\"uller term $\pa_\tau \E\XX$ by means of the Virasoro generator $L_0$, the zeroth mode of the Virasoro field.
In this section we derive Ward's equation in a doubly connected domain not from Eguchi and Ooguri's path integral formalism but some basic identities in special function theory together with Wick's calculus.

\subsection{Ward's functional and Ward's identities} 

We first introduce global Ward's functional and present the relation between a stress tensor and Ward's functional. 

Let $W=(A_{\bfs{\beta}},\bar{A}_{\bfs{\beta}})$ be a stress tensor for $\mathcal F_{\bfs{\beta}}$. 
By Proposition~\ref{AT beta}, it is easy to see that $A=A_{\bfs{\beta}}$ is continuous and real on the boundary. 
For a given open set $U\subset D_r$ and a smooth vector field $v$ in $\overline{U}$ continuous up to the boundary, Ward's functionals $W^{\pm}(v; U)$ are defined by
$$
W^+(v;U)=\dfrac{1}{2\pi i} \int_{\pa U}vA-\frac{1}{\pi}\int\!\!\int_U (\bp v) A, \qquad \qquad	W^-(v;U)=\overline{W^+(v;U)}.
$$
We also write $W(v;U):=2\,\Re\, W^+(v;U)$. Then for any Fock space functional $\XX$ with nodes in $D_r$, $\E [W^+(v;U)\XX]$ is a well-defined correlation function in $D_{\textrm{hol}}(v)$.

We will frequently use the following proposition. 
See \cite[Propositions 5.3 and 5.10]{MR3052311} for more details. 
The last equations in the following proposition are known as Ward's identities. 
\begin{prop} \label{prop: Ward's identities}
	The following statements are equivalent:
	\begin{itemize}
		\item Ward's OPE holds for a Fock space field $X;$
		\smallskip 
		\item The residue form of Ward's identity for $X$	
		\begin{align*}
			\LL_v X(z)=\frac{1}{2\pi i}\oint_{(z)}vA X(z)-\frac{1}{2\pi i}\oint_{(z)} \overline{v}\bar{A} X(z)
		\end{align*}
		holds on $D_{\textrm{hol}}(v)\cap U$ for all smooth vector field $v;$
			\smallskip 
		\item	For all $z\in D_{\textrm{hol}}(v)$, 
		\begin{align*}
			\E[\YY\LL(v,U)X(z)]=\E[\YY W(v,U)X(z)]
		\end{align*} 
		holds for all Fock space functionals $\YY$ with nodes outside supp($v$).
	\end{itemize}
\end{prop}

The definition of Ward's functionals can be extended to meromorphic vector fields. For a meromorphic vector field $v$ with poles $\{ p_k \}$ in $D_r$,
we define 
$$
W^{+}(v)=W^{+}(v; D_r):=\lim_{\ve \to 0}W^{+}(v; D_r^\ve),
$$
where $D_r^{\ve}=D_r \sm \bigcup_k B(p_k,\ve)$. 
Thus, we have 
$$
W^{+}(v)=W^+(v;D_r)	=	\frac{1}{2\pi i}\int_{\pa D_r}vA-\frac{1}{\pi} \int \!\! \int_{D_r}(\bp v)A,
$$
with the interpretation of $\bp v$ in the sense of distributions.

\begin{prop}\label{Arep_ER}
	In the identity chart $\id_{\bar{\CC}_r }$, we have
	$$
	2	A_{ \bfs{\beta} }(\zeta) = W^+(v_\zeta)+W^- (v_{\bar{\zeta}}) + \frac{1}{\pi}\int_{[-\pi+ir,\pi+ir]}	A_{ \bfs{\beta} }(\xi) \, d\xi,
	$$ 
	where $(v_\zeta \| \id_{\bar{\mathcal C}_r})(z)=H(r,\zeta-z)$.
\end{prop}
\begin{proof} 
	Recall that a stress tensor $A_{ \bfs{\beta} }$
	is a holomorphic quadratic differential in $\mathcal C_r$, and real on the boundary. For $\zeta \in \mathcal C_r$, the reflected vector field $v_\zeta^\sharp$ is defined by
	$$
	v_\zeta^\sharp(z):=\overline{v_\zeta(\bar{z}) }=v_{\bar{\zeta}}(z).
	$$
	Since $v_\zeta$ is meromorphic on $\mathcal C_r$ with poles at $\{ \zeta+2\pi m+2ir n : m,n \in \mathbb Z \}$, the reflected vector field $v_\zeta^\sharp$ is holomorphic on $\mathcal C_r$. Then by definition,
	$$
	W^+(v_\zeta)=- \frac{1}{\pi}\int_{\mathcal C_r}(\bp v_\zeta)	A_{ \bfs{\beta} }+ \frac{1}{2\pi i}\int_{\pa \mathcal C_r}v_{\zeta}	A_{ \bfs{\beta} },
	\qquad W^+(v_{\zeta}^{\sharp})=\frac{1}{2\pi i}\int_{\pa\mathcal C_r}v_\zeta^{\sharp}	A_{ \bfs{\beta} }.
	$$
	Also, since $	A_{ \bfs{\beta} }=\bar{	A}_{ \bfs{\beta} }$ on $\pa \mathcal C_r$, we have
	$$
	W^-(v_{\bar{\zeta}})=\overline{W^{+}(v_{\zeta}^{\sharp})}=\overline{\frac{1}{2\pi i}\int_{\pa\mathcal C_r }v_{\bar{\zeta}} (\xi )	A_{ \bfs{\beta} }(\xi)\, d\xi}=
	-\frac{1}{2\pi i}\int_{\pa\mathcal C_r}v_\zeta (\bar{\xi} )	A_{ \bfs{\beta} }(\xi)\, d\xi.
	$$
	Then it follows from $\bp v_\zeta=-2\pi \delta_\zeta$, and $H(r,z+2ir)=H(r,z)-2i$ that
	\begin{align*}
		W^+(v_\zeta)+W^- (v_{\bar{\zeta}})
		&=- \frac{1}{\pi}\int_{\mathcal C_r}(\bp v_\zeta)	A_{ \bfs{\beta} }-\frac{1}{\pi}\int_{[-\pi+ir,\pi+ir]}	A_{ \bfs{\beta} }(\xi)\, d\xi 
		= 2	A_{ \bfs{\beta} }(\zeta) -\frac{1}{\pi}\int_{[-\pi+ir,\pi+ir]}	A_{ \bfs{\beta} }(\xi)\, d\xi.
	\end{align*}
	This completes the proof.
\end{proof}

\subsection{Eguchi-Ooguri equations} \label{EO subsec}

We now state and prove a version of Eguchi-Ooguri equations. 

First let us consider the case $\bfs\beta=\bfs 0$.
Recall that $\mathcal{F}_{(b)} = \mathcal{F}_{(b),\bfs0}$ and $ \mathcal{F}_{(b),\bfs\beta}$ is the OPE family generated by $\Phi_{(b),\bfs\beta}$. 

\begin{lem} \label{Bint_ER} For any $\XX \in \FF_{(b)}$, in the $\CC_r$-uniformization, 
	\begin{equation} \label{bW}
		\frac{1}{\pi}\oint_{[-\pi+ir,\pi+ir]}\E A(\zeta)\XX \,d\zeta=\pa_r\E \XX.
	\end{equation}
\end{lem}
\begin{proof}
	Let us denote by $\mathcal{W}$ the family of Fock space fields satisfying \eqref{bW}. 
	We show $\mathcal F_{(b)} \subset \mathcal{W}$ through the following steps.
	
	\medskip
	\noindent\textbf{Step 1.} A tensor product of two GFFs is in $\mathcal W$.
	
	\medskip
	\noindent\textbf{Step 2.} All tensor products of GFFs and their derivatives are in $\mathcal W$.
	
	\medskip
	\noindent\textbf{Step 3.} The family $\mathcal W$ is closed under linear combinations over $\C$ and OPE products $\ast_n$.
	
	\medskip
	
	\noindent \emph{Proof of Step 1.} 
	Let us denote $A\equiv A_{(b)}$, $\Phi \equiv \Phi_{(b)}$. Recall that in the identity chart $\id_{\mathcal C_r}$, $ A=-\frac{1}{2} J_{(0)}\odot J_{(0)}+ib\pa J_{(0)}$, and $\Phi=\Phi_{(0)}$. Since $\pa J$ is exact, by the fundamental theorem of calculus, and the fact that all correlation functions lifted to the covering space of $\CC_r$ are 2$\pi$ periodic, it suffices to show the case $b=0$. 
	
	First, we show this step for the ER boundary condition (i.e., $\chi=\infty$).
	For $\mathcal{X}=\Phi(z_1)\Phi(z_2)$, by \eqref{2pt cor_ER} and the fact 
	$\pa_r \log \Theta= H'/2+H^2/4$ (which follows from \eqref{H Theta} and \eqref{heat eq. Theta}), we have
	\begin{align*}
		\pa_r \E [\Phi(z_1)\Phi(z_2) ] 
		= -\RN{1}(z_1,z_2)+\RN{1}(\bar{z}_1,z_2)+\RN{1}(z_1,\bar{z}_2)-\RN{1}(\bar{z}_1,\bar{z}_2),
	\end{align*}
	where $\RN{1}(z,w):=H'(z-w)/2+H^2(z-w)/4$.
	On the other hand, by Wick's calculus and \eqref{2pt cor CG_ER}, we have
	\begin{equation*}
		\E[A(\zeta)\Phi(z_1)\Phi(z_2)]
		=-\RN{2}_{z_1,z_2}(\zeta)+\RN{2}_{\bar{z}_1,z_2}(\zeta)+\RN{2}_{z_1,\bar{z}_2}(\zeta)-\RN{2}_{\bar{z}_1,\bar{z}_2}(\zeta),
	\end{equation*}
	where 
	$\RN{2}_{z,w}(\zeta):=H(\zeta-z)H(\zeta-w)/4$. Note that by \eqref{add H}, $\RN{2}_{z,w}(\zeta)$ can be rewritten as 
		\begin{align} \label{123 rel}
	\RN{2}_{z,w}(\zeta)=\frac{1}{2}\RN{1}(z,w)+\frac{1}{4}\RN{3}_{z,w}(\zeta)+ \frac{1}{8}H^2(\zeta-z)+\frac{1}{8}H^2(\zeta-w)+\frac{3\zeta_r(\pi)}{2\pi},
	\end{align}
	where 
	$$
	\RN{3}_{z,w}(\zeta):=H'(\zeta-z)+H'(\zeta-w)+( H(\zeta-w)-H(\zeta-z) ) H(w-z).
	$$
	Combining all of the above, we obtain 
	\begin{align*}
		\E[A(\zeta)\Phi(z_1)\Phi(z_2)]&= \frac{1}{2} \pa_r \E [\Phi(z_1)\Phi(z_2)]
		-\frac{1}{4} \Big( \RN{3}_{z_1,z_2}(\zeta)-\RN{3}_{\bar{z}_1,z_2}(\zeta)-\RN{3}_{z_1,\bar{z}_2}(\zeta)+\RN{3}_{\bar{z}_1,\bar{z}_2}(\zeta) \Big). 
	\end{align*}
	Finally, observe that by \eqref{H Theta} and \eqref{theta per}, we have
	$\oint_{[-\pi+ir,\pi+ir]} \RN{3}_{z,w}(\zeta) \, d\zeta \equiv 0$, 
	which leads to \eqref{bW} for $\mathcal X=\Phi(z_1)\Phi(z_2)$ with ER boundary condition. 
	
	We now show this step for the weighted boundary conditions with any $\chi \in [0,\infty)$. For notational convenience, we add the superscript $``ER"$ for the field with ER boundary condition, e.g., $\Phi^{ER}, A^{ER}$ etc. 
 Then by \eqref{Green interpol} and the periodicity of $H$, we have
	\begin{align*}
	\begin{split}
	\frac{1}{\pi}\oint \Big( \E [ A(\zeta) \Phi(z_1)\Phi(z_2) ]-\E[ A^{ER}(\zeta) \Phi^{ER}(z_1)\Phi^{ER}(z_2) ] \Big) \, d\zeta =-\frac{(z_1-\bar{z}_1)(z_2-\bar{z}_2)}{2\,(r+\chi)^2}.
	\end{split}
	\end{align*}
	On the other hand, 
	\begin{align*}
	\begin{split}
	\pa_r \Big( \E \Phi(z_1)\Phi(z_2)- \E \Phi^{ER}(z_1)\Phi^{ER}(z_2) \Big) =
	\pa_r \Big( \frac{1}{r+\chi} \Big) \frac{ (z_1-\bar{z}_1)(z_2-\bar{z}_2)}{2}.
	\end{split}
	\end{align*}
	Combining all of the above identities, the claim follows. 
	\medskip 
	
	\noindent \emph{Proof of Step 2.} 
	There is nothing to prove for $\mathcal X=\Phi(z_1)\cdots\Phi(z_{2n-1})$. Therefore it is enough to show \eqref{bW} for $\mathcal X=\Phi(z_1)\cdots\Phi(z_{2n})$.
	For $j \neq k$, let us denote 
	$$
	\mathcal{X}_{jk}=\Phi(z_1)\cdots\Phi(z_{k-1})\Phi(z_{k+1})\cdots\Phi(z_{j-1})\Phi(z_{j+1})\cdots \Phi(z_{2n}).
	$$
	Then by Wick's formula and an induction argument, we have 
	\begin{align*}
		\frac{1}{\pi}\int\E [A(\zeta)\Phi(z_1)\cdots\Phi(z_{2n}) ]\, d\zeta
		&= \frac{1}{\pi}\int\sum_{j \neq k}\E [A(\zeta)\Phi(z_j)\Phi(z_k) ] \E[\mathcal{X}_{jk}] \, d\zeta 
	= \sum_{j \neq k} \pa_r\E[\Phi(z_j)\Phi(z_k)]\E[\mathcal{X}_{jk}].
	\end{align*}
	On the other hand, by \eqref{n-pt cor} and Leibniz's rule, we express the left-hand side of \eqref{bW} as
	\begin{align*}
		\pa_r \E [ \Phi(z_1) \cdots \Phi(z_{2n}) ] &= \pa_r \sum \prod_k \E [ \Phi(z_{i_k})\Phi(z_{j_k}) ]
		\\
		&= \sum \Big( \prod_k \E [ \Phi(z_{i_k})\Phi(z_{j_k}) ] \Big) \Big( \sum_k \frac{ \pa_r \E [ \Phi(z_{i_k})\Phi(z_{j_k}) ]}{\E [ \Phi(z_{i_k})\Phi(z_{j_k}) ]}\Big),
	\end{align*}
	where the outer sum is taken over all partitions of the set $\{ 1,\cdots ,2n \}$ into disjoint pairs $\{ i_l ,j_l\}$. 
	Rearranging the terms, we rewrite the above identity as
	$$
	\pa_r \E [ \Phi(z_1) \cdots \Phi(z_{2n}) ]=\sum_{j \neq k} \Big(\pa_r\E[\Phi(z_j)\Phi(z_k)]\Big) \Big(\sum \prod_l \E [ \Phi(z_{i_l})\Phi(z_{j_l}) ]\Big),
	$$
	where the sum is over all partitions of the set $\{ 1,\cdots ,2n \} \sm \{ j,k\}$ into disjoint pairs $\{ i_k ,j_k\}$. Therefore again, by \eqref{n-pt cor}, we have \eqref{bW} for $\mathcal X=\Phi(z_1)\cdots\Phi(z_{2n})$. Differentiating term by term, we complete the proof of Step 2. 
	
 \medskip 
	
	\noindent \emph{Proof of Step 3.} The first part of this step is obvious. For the second part, we consider the case that the field $X$ is holomorphic. The other cases can be treated in a similar way. Suppose that we have the following OPE
	$$
	X(\zeta)Y(z)=\sum C_n(z)(\zeta -z)^n,\qquad \zeta \to z.
	$$ 
	If both $X$ and $Y$ are in $\mathcal W$, then the term by term integration gives
	$$
	\frac{1}{\pi}\int_{[-\pi+ir,\pi+ir]} \E[A(\xi)X(\zeta)Y(z)] \, d\xi =\sum_n \Big( \frac{1}{\pi}\int_{[-\pi+ir,\pi+ir]}\E[A(\xi)C_n(z)]\, d\xi \Big) (\zeta -z)^n
	$$
	as $\zeta \to z$. On the other hand, differentiating term by term, we have
	$$
	\pa_r \E[X(\zeta)Y(z)]=\sum\pa_r\E[C_n(z)](\zeta-z)^n.
	$$ 
	Therefore, by comparing the coefficients of the two expansions above, we conclude $C_n \in \mathcal W$.
\end{proof}

\begin{rmk*} 
Let us consider the GFF $\Phi$ with 2-point correlation function given by a probabilistic interpolation (via Bernoulli distribution with probability $p \equiv p(r)$) of Green's functions with ER and Dirichlet boundary conditions, i.e., 
	\begin{align*}
	\E\Phi(\zeta)\Phi(z)&=2\Big( p(r)G^{Diri}_r(\zeta,z)+(1-p(r))G^{ER}_r(\zeta,z) \Big)
=2\log\Big| \frac{\Theta(r,\zeta-\bar{z})}{\Theta(r,\zeta-z)} \Big|-\frac{2\,p(r)}{r} \, \Im\,\zeta \, \Im\,z.
	\end{align*}
Then the conclusion of Lemma~\ref{Bint_ER} can be drawn if and only if $\pa_r(r/p(r))=1$. In other words, to derive Eguchi-Ooguri type equation of the form \eqref{bW}, it is required that 
	\begin{equation}
	p(r)=r/(r+\chi), \qquad (\chi \ge 0).
	\end{equation}
This corresponds to the rational interpolation \eqref{Green interpol}. 
\end{rmk*}

Let us emphasize that Lemma~\ref{Bint_ER} holds only for the non-chiral field $\XX$. 
To our purpose of implementing Ward's equation for the OPE family of GFFs associated with the general background charge, we now extend Eguchi-Ooguri equations to a certain family of chiral fields.
We write $(\bfs{{\sigma}},\bfs{{\sigma_*}})=(\sum \sigma_j \cdot z_j ,\sum \sigma_{*j} \cdot z_j )$ for a double divisor satisfying $(\mathrm{NC}_0)$ and denote by $\OO\equiv \OO_{(b)}{[\bfs{\sigma},\bfs{\sigma_*}]}$ the associated multi-vertex field. 
Due to the following proposition, we can revise the definition of $\FF_{(b)}$ so that it includes the fields of the form $\OO \XX$. 
\begin{prop} \label{E-O eq.} Let $\gamma=[-\pi+ir,\pi+ir]$. Then for any $\XX \in \FF_{(b)}$, 
	\begin{equation} \label{E-O}
	\frac{1}{\pi}\oint_{\gamma}\E A(\zeta) \OO \XX \, d\zeta=\pa_r\E \, \OO\XX,
	\end{equation}
	where all fields are evaluated in the identity chart of $\mathcal C_r$. 
\end{prop}
\begin{proof}
We present the proof for ER boundary conditions and leave it to the reader for the general case.
	
	Again it suffices to show the case $b=0$. In the sequel, we omit the subscript $(0)$ and write for instance $A\equiv A_{(0)}, C \equiv C_{(b)}$. Throughout the proof, we use symbols $\RN{1}(z,w), \RN{2}_{z,w}(\zeta)$ and $\RN{3}_{z,w}(\zeta)$ given in the proof of Lemma~\ref{Bint_ER}.
	
	By definition of the multi-vertex fields, \eqref{E-O} is equivalent to 
	\begin{align} \label{E-O_1}
	\begin{split}
	\frac{1}{\pi}\oint_{\gamma}\E A(\zeta) V^\odot[\bfs\sigma,\bfs\sigma_*] \XX \, d\zeta=\frac{\pa_r C{[\bfs{\sigma},\bfs{\sigma_*}]} }{C{[\bfs{\sigma},\bfs{\sigma_*}]}}\E V^\odot[\bfs\sigma,\bfs\sigma_*] \XX
	+\pa_r \E V^\odot[\bfs\sigma,\bfs\sigma_*] \XX,
	\end{split}
	\end{align}
	where $V^\odot[\bfs\sigma,\bfs\sigma_*]=e^{\odot i\Phi[\bfs{{\sigma}},\bfs{{\sigma_*}}]}$. Using similar arguments in the proof of Lemma~\ref{Bint_ER}, it suffices to show the identity \eqref{E-O_1} for tensor products of GFFs. 
	
	Let us first show that the identity \eqref{E-O_1} holds for trivial string $\XX=1$.
	Combining \eqref{Coulomb ftn} with \eqref{H Theta}, \eqref{heat eq. Theta} and \eqref{Mod inv}, we have 
	\begin{align}\label{pa r logC}
	\begin{split}
\pa_r \log C{[\bfs{\sigma},\bfs{\sigma_*}]}	&=- \frac{3\zeta_r(\pi)}{2\pi} \sum (\sigma_j^2+\sigma_{*j}^2) +\sum \sigma_{j}\sigma_{*j} \RN{1}(z_j,\bar{z}_j)
	\\
	&+\sum_{ j<k} \sigma_j\sigma_k \RN{1}(z_j,z_k)+\sigma_{*j}\sigma_k \RN{1}(\bar{z}_j,z_k)+\sigma_j\sigma_{*k}\RN{1}(z_j,\bar{z}_k)+\sigma_{*j}\sigma_{*k}\RN{1}(\bar{z}_j,\bar{z}_k).
	\end{split}
	\end{align}
	Differentiating \eqref{formal ++ cor} and \eqref{formal +- cor}, we have
	\begin{align}\label{J bos}
	\E J(\zeta) \Phi\,[\bfs{{\sigma}},\bfs{{\sigma_*}}]
	=-\dfrac{1}{2} \sum \Big( \sigma_j H(\zeta-z_j)+\sigma_{*j} H(\zeta-\bar{z}_j) \Big) .
	\end{align}
	By Wick's formula and \eqref{J bos}, 
	\begin{align} \label{pre chi0}
	\begin{split}
	 \E A(\zeta) V^\odot[\bfs\sigma,\bfs\sigma_*] &=\frac{1}{8} \sum \sigma_j^2 H^2(\zeta-z_j)+\sigma_{*j}^2 H^2(\zeta-\bar{z}_j)+\sum \sigma_{j}\sigma_{*j} \RN{2}_{z_j,\bar{z}_j}(\zeta)
	\\
	&+\sum_{ j<k} \sigma_j\sigma_k \RN{2}_{z_j,z_k}(\zeta)+\sigma_{*j}\sigma_k \RN{2}_{\bar{z}_j,z_k}(\zeta)+\sigma_j\sigma_{*k} \RN{2}_{z_j,\bar{z}_k}(\zeta)+\sigma_{*j}\sigma_{*k} \RN{2}_{\bar{z}_j,\bar{z}_k}(\zeta).
	\end{split}
	\end{align}
 Then it follows from \eqref{123 rel} and the neutrality condition \eqref{NC0_ER} that
	\begin{align*}
	&\quad \E A(\zeta) V^\odot[\bfs\sigma,\bfs\sigma_*]-\frac{1}{2}\pa_r \log C{[\bfs{\sigma},\bfs{\sigma_*}]}
	\\
	&=\frac{1}{4}\sum_{ j<k} \sigma_j\sigma_k \RN{3}_{z_j,z_k}(\zeta)+\sigma_{*j}\sigma_k \RN{3}_{\bar{z}_j,z_k}(\zeta)+\sigma_j\sigma_{*k} \RN{3}_{z_j,\bar{z}_k}(\zeta)+\sigma_{*j}\sigma_{*k} \RN{3}_{\bar{z}_j,\bar{z}_k}(\zeta).
	\end{align*}
	Now claim follows from the fact that 
	$\oint_{\gamma} \RN{3}_{z,w}(\zeta) \, d\zeta \equiv 0$. 
 
 Next, we show that the identity \eqref{E-O_1} holds for the case $\XX=\Phi$.
 By Wick's formula, we have
	\begin{align}
	\begin{split}
	\label{V phi}
	\E V^\odot[\bfs\sigma,\bfs\sigma_*]\Phi(z)=i\, \E \Phi\,[\bfs{{\sigma}},\bfs{{\sigma_*}}] \Phi(z)
=i \sum \sigma_j \log \frac{\Theta(z_j-\bar{z})}{\Theta(z_j-z)}+\sigma_{*j} \log \frac{\Theta(\bar{z}_j-\bar{z})}{\Theta(\bar{z}_j-z)} .
	\end{split}
	\end{align}
	Therefore by \eqref{heat eq. Theta}, 
	$$
	\pa_r \E V^\odot[\bfs\sigma,\bfs\sigma_*]\Phi(z)=i \sum \sigma_j \Big(\RN{1}(z_j,\bar{z})-\RN{1}(z_j,z) \Big) + \sigma_{*j} \Big(\RN{1}(\bar{z}_j,\bar{z})-\RN{1}(\bar{z}_j,z) \Big) .
	$$
	
	On the other hand, by \eqref{pre chi0} and \eqref{V phi},
	\begin{align} \label{pre chi1}
	\begin{split}
	\E A(\zeta) V^\odot[\bfs\sigma,\bfs\sigma_*]\Phi(z)
	= F(\zeta,z,\bfs{z})-iG(\zeta,z,\bfs{z}),
	\end{split}
	\end{align}
	where
	\begin{align*}
	F(\zeta,z,\bfs{z})=\E A(\zeta) V^\odot[\bfs\sigma,\bfs\sigma_*] \E V^\odot[\bfs\sigma,\bfs\sigma_*]\Phi(z),
	\qquad 
	G(\zeta,z,\bfs{z})=\E J(\zeta)\Phi(z) \E J(\zeta) \Phi\,[\bfs{{\sigma}},\bfs{{\sigma_*}}].
	\end{align*}
	Note that 
	\begin{align}
	\begin{split}
	\frac{1}{\pi}\oint_{\gamma} F(\zeta,z,\bfs{z})\,d\zeta= \frac{1}{\pi}\oint_{\gamma} \E A(\zeta) V^\odot[\bfs\sigma,\bfs\sigma_*] \, d\zeta \, \E V^\odot[\bfs\sigma,\bfs\sigma_*]\Phi(z)
	=\frac{\pa_r C{[\bfs{\sigma},\bfs{\sigma_*}]} }{C{[\bfs{\sigma},\bfs{\sigma_*}]}}\E V^\odot[\bfs\sigma,\bfs\sigma_*] \Phi(z).
	\end{split}
	\end{align}
	Thus it follows from the neutrality condition \eqref{NC0_ER} and \eqref{J bos} that 
	\begin{align*}
	G(\zeta,z,\bfs{z})&=\sum \sigma_{j}\Big( \RN{2}_{z_j,z}(\zeta)-\RN{2}_{z_j,\bar{z}}(\zeta) \Big)+ \sigma_{*j}\Big( \RN{2}_{\bar{z}_j,z}(\zeta)-\RN{2}_{\bar{z}_j,\bar{z}}(\zeta) \Big) 
	\\
	&=\frac{1}{2}\sum \sigma_{j}\Big( \RN{1}(z_j,z)-\RN{1}(z_j,\bar{z}) \Big)+\sigma_{*j}\Big( \RN{1}(\bar{z}_j,z)-\RN{1}(\bar{z}_j,\bar{z}) \Big) 
	\\
	&\quad +\frac{1}{4}\sum \sigma_{j}\Big( \RN{3}_{z_j,z}(\zeta)-\RN{3}_{z_j,\bar{z}}(\zeta) \Big)+ \sigma_{*j}\Big( \RN{3}_{\bar{z}_j,z}(\zeta)-\RN{3}_{\bar{z}_j,\bar{z}}(\zeta) \Big) .
	\end{align*}
	Therefore we obtain
	\begin{align} \label{int G}
	\begin{split}
	-\frac{i}{\pi}\oint_{\gamma} G(\zeta,z,\bfs{z})\, d\zeta&=\pa_r \E V^\odot[\bfs\sigma,\bfs\sigma_*]\Phi(z),
	\end{split}
	\end{align}
	which leads to \eqref{E-O_1} for $\XX=\Phi$.
	
	Now Proposition for the tensor products of GFFs follows from a similar argument of Claim 2 in the proof of Lemma~\ref{Bint_ER}. We leave it to the reader as an exercise. 
\end{proof}

In the physics literature, instead of a stress tensor $A$, the Virasoro field $T$ is often utilized to state Ward's equation. 
To describe Eguchi-Ooguri equation in terms of the Virasoro field, let us define \emph{effective multi-vertex field} $\OO^{\textup{eff}}\equiv\OO_{(b)}^{\textrm{eff}}{[\bfs{\sigma},\bfs{\sigma_*}]}$ as 
\begin{equation} \label{Z_r}
\OO_{(b)}^{\textrm{eff}}{[\bfs{\sigma},\bfs{\sigma_*}]}:=\OO_{(b)}{[\bfs{\sigma},\bfs{\sigma_*}]} Z_r, \qquad 	Z_r=\Theta'(r,0)^{-1/3} (r+\chi)^{-1/2}.
\end{equation}
We remark that the effective multi-vertex field plays an import role in describing the restriction observables in a doubly connected domain \cite{MR1992830,MR2343934,dubedat2004critical,MR2902453}. 

By Proposition~\ref{E-O eq.}, we obtain following corollary.
\begin{cor} \label{E-O eq.T} Let $\gamma=[-\pi+ir,\pi+ir]$. Then for any $\XX \in \FF_{(b)}$, 
	\begin{equation} \label{E-O T}
	\frac{1}{\pi}\oint_{\gamma}\E T(\zeta) \OO^{\textup{eff}} \XX \, d\zeta=\pa_r\E \OO^{\textup{eff}}\XX,
	\end{equation}
	where all fields are evaluated in the identity chart of $\mathcal C_r$. 
\end{cor}
\begin{proof}
	By Proposition~\ref{E-O eq.}, we have
	\begin{align*}
	&\quad \frac{1}{\pi}\oint_{\gamma}\E T(\zeta) \OO^{\textrm{eff}} \XX \, d\zeta=\Big( \frac{1}{\pi}\oint_{\gamma}\E T(\zeta) \OO \XX \, d\zeta \Big) Z_r
	\\
	&=\Big( \frac{1}{\pi}\oint_{\gamma}\E A(\zeta) \OO \XX \, d\zeta\Big)Z_r + \Big( \frac{1}{\pi}\oint_{\gamma}\E T(\zeta) \, d\zeta \Big) \E \OO\mathcal{X} Z_r
	= \Big( \pa_r \E \OO\XX \Big)Z_r +2 \E T \, \E \OO\mathcal{X} Z_r.
	\end{align*}
	Note that by Proposition~\ref{Virasoro2_ER}, \eqref{Mod inv} and \eqref{Z_r}, we have 
	\begin{equation}
		\E T=-\frac{1}{6}\frac{\Theta'''(r,0)}{\Theta'(r,0)}-\frac{1}{4(r+\chi)}=\frac12 \pa_r \log Z_r. 
	\end{equation}
 Therefore we conclude that 
	$$
	\frac{1}{\pi}\oint_{\gamma}\E T(\zeta) \OO^{\textrm{eff}} \XX \, d\zeta=\Big( \pa_r \E \OO\XX \Big)Z_r+ \E \OO\mathcal{X} \pa_r Z_r = \pa_r\E \OO^{\textup{eff}}\XX,
	$$
	which completes the proof. 
\end{proof}

By \eqref{insertion bcm gen}, one can see that the background charge modifications can be treated as an insertion of a multi-vertex field. (See Subsection~\ref{one-leg subsec} for further details.)
Therefore, as a consequence of Proposition~\ref{E-O eq.}, we extend Lemma~\ref{Bint_ER} to the OPE family $\FF_{\bfs{\beta}}$ with non-trivial background charges $\bfs{\beta}$. 
\begin{cor} \label{EO lem} For any $\XX \in \FF_{\bfs{\beta}}$, in the $\CC_r$-uniformization, 
	\begin{equation} \label{EO eq}
	\frac{1}{\pi}\oint_{[-\pi+ir,\pi+ir]}\E A_{\bfs{\beta}}(\zeta)\XX \, d\zeta=\pa_r\E \XX. 
	\end{equation}
\end{cor}

\subsection{Ward's equations} \label{Ward's identities_ER}

In this subsection we derive Ward's equations (Theorem~\ref{Thm_Ward}) in terms of a stress tensor, Lie derivative operator and the modular parameter. Later, we combine these equations with the level two degeneracy equations for the one-leg operator to derive BPZ equations.

We first prove Theorem~\ref{Thm_Ward} in the case trivial background charge $\bfs\beta = \bfs0$. 

\begin{proof}[Proof of Theorem~\ref{Thm_Ward} in the case $\bfs\beta = \bfs0$]
	Let us choose 
	$v(z) = v_\zeta(z)=H(r,\zeta-z)$ with $\zeta \in \mathcal C_r$. 
	By Proposition~\ref{Arep_ER}, Eguchi-Ooguri equations (Corollary~\ref{EO lem}), and Ward's identities (see Proposition~\ref{prop: Ward's identities}), we obtain
	\begin{align*}
		2\E A_{\bfs\beta}(\zeta) \mathcal{X} 
		&= \E W^{+}(v_\zeta) \mathcal{X} +\E W^{-}(v_{\bar{\zeta}}) \mathcal{X}+\pa_r \E \mathcal{X}
		= \Big(\LL_{v_\zeta}^+ +\LL_{v_{\bar{\zeta}}}^- \Big)\E \mathcal{X}+\pa_r \E \mathcal{X} ,
	\end{align*}
	which completes the proof.	
\end{proof}

For the general background charge $\bfs \beta$, it is required to generalize the residue form of Ward's identity at the nodes of $\bfs \beta$, see Lemma~\ref{Lem_Ward beta node}. 
Since this lemma can be more easily obtained by employing the insertion procedure, we defer its proof to Subsection~\ref{one-leg subsec}.
Then the proof of Theorem~\ref{Thm_Ward} for the non-trivial background charge can be obtained along the same lines for the trivial background charge with slight modifications. 

As a corollary of Theorem~\ref{Thm_Ward}, we derive the following form of Ward's equations. 
\begin{cor} \label{Wequation_ER}
	For a holomorphic differential $Y\in \FF_{\bfs\beta}$ with conformal dimension $h$, and $X=X_1\cdots X_n$, $( X_j \in \FF_{\bfs\beta}$), we have
	$$ 
	2\E A_{\bfs\beta}\ast Y(z) X =\E Y(z) \LL_{v_z}^+ X+\LL_{v_{\bar{z}}}^-\E Y(z)X + \Big(2h \frac{\zeta_r(\pi)}{\pi}+\pa_r \Big)\E Y(z)X,
	$$ 
	where all fields are evaluated in the identity chart of $\mathcal C_r$. 
\end{cor}

\begin{proof}
	By Ward's OPE, we have the following singular part of the operator product expansion: 
	$$
	\Sing_{\zeta \to z}A_{\bfs\beta}(\zeta)Y(z)X=\frac{h}{(\zeta-z)^2}Y(z)X+\frac{\pa_z}{\zeta-z}Y(z)X.
	$$ 
	On the other hand, by Theorem~\ref{Thm_Ward}, we have 
	\begin{align*}
		2\E A_{\bfs\beta}(\zeta)Y(z) X 
		= (hv'_\zeta (z)+v_\zeta(z) \pa_z+\pa_r )\E Y(z)X 
		+\E Y(z) \LL_{v_\zeta}^+ X +\LL_{v_{\bar{\zeta}}}^-\E Y(z)X.
	\end{align*}
	Subtracting the singular part from the both sides of the above and then taking the limit as $\zeta \to z$, we obtain
	\begin{align*}
		2\E A_{\bfs\beta}\ast Y(z) X&= \E Y(z) \LL_{v_z}^+ X+\LL_{v_{\bar{z}}}^-\E Y(z)X +\pa_r\E Y(z)X 
		\\
		&+ \lim_{\zeta \to z} \Big\{ hv_{\zeta}'(z)-\frac{2h}{(\zeta-z)^2}+ \Big(v_{\zeta}(z)-\frac{2}{\zeta-z}\Big) \pa_z \Big\} \E Y(z)X.
	\end{align*}
	Corollary now follows from the asymptotic behavior \eqref{h-zero} of the Loewner vector field $v_\zeta(z)=H(r,\zeta-z)$. 
\end{proof}

\section{Annulus SLE martingale-observables} \label{SLE MOs ER}
In this section we present the precise relation between CFT and SLE in a doubly connected domain. 
In the first subsection we define the one-leg operator $\Psi$ and explain its implementation as a boundary condition changing operator. 
Subsection~\ref{Subsec_NVE} is devoted to the study of general null-vector equations. 
In Subsection~\ref{Main Theorem_sub}, we show a version of BPZ-Cardy type equation and prove Theorem~\ref{main}. 

\subsection{One-leg operator and BPZ equations} \label{one-leg subsec}

We introduce the \emph{one-leg operator} $\Psi_{\bfs{\beta}}$, a specific form of the modified multi-vertex field set to satisfy level two degeneracy equation. Throughout this subsection we denote by $\bfs{\beta}= \sum_k \beta_k \cdot q_k$ a given background charge. 

By definition, for $n \in \mathbb Z$, Virasoro and current generators $L_n$, $J_n$ are given as
$$
L_n(z):=\frac{1}{2\pi i}\oint_{(z)}(\zeta -z )^{n+1} T_{ \bfs{\beta} }(\zeta) \, d\zeta, \qquad
J_n(z):=\frac{1}{2\pi i}\oint_{(z)} (\zeta -z )^n J_{ \bfs{\beta} }(\zeta) \, d\zeta,
$$
respectively.
Let us recall the following proposition.
\begin{prop}[Cf. Proposition 7.5 \cite{MR3052311}] \label{V.P.F.} 
	Let $X$ be a Fock space field. Then any two of the following assertions imply the third one:
	\begin{itemize}
		\item	$X$ belongs to $\FF(A,\bar{A})$;
		\smallskip 
		\item $X$ is a $[\lambda,\lambda_*]$-differential;
		\smallskip 
		\item $L_n X=0$ for all $n\ge 1$, $L_0 X=\lambda X, L_{-1}X=\pa X$, and similar equations hold for $\bar{X}$. 
	\end{itemize}
\end{prop}

Any field satisfying all conditions in Proposition~\ref{V.P.F.} is called a \emph{(Virasoro) primary} field in $\FF(A,\bar{A})$. 
Furthermore, a Virasoro primary field $X$ is called \emph{current primary} if	
$ J_n X=J_n\bar{X} =0$ for all $n\ge 1$ and
$J_0 X=-iq X$, $J_0 \bar{X}=i\overline{q}_* \bar{X}$ for some numbers $q$ and $q_*$, which are called charges. One of the main ingredients to connect CFT with SLE theory is the following \emph{level two degeneracy equation}. 

\begin{prop}[Cf. Proposition 11.2 \cite{MR3052311}] \label{level2}
	For a current primary field $V$ with charges $q$, $q_*$ in $\FF_{ \bfs{\beta} }$, we have	
	\begin{align*}
		L_{-2}V = \frac{1}{2q^2}L^2_{-1}V
	\end{align*}
	provided $2q(b+q)=1$. 
\end{prop}
Suppose that the real parameters $a$ and $b$ are related to the SLE parameter $\kappa$ as 
$$
a=\sqrt{2/\kappa}, \qquad b=a(\kappa/4-1).
$$
In the sequel, we sometimes use the following Kac labeling of the conformal weights or dimensions:
$$
h_{r,s}=\frac{(r\kappa -4s)^2-(\kappa -4 )^2}{16\kappa}. 
$$

For a divisor $\bfs{\tau}= \sum_{j=1}^N \tau_j \cdot z_j$ satisfying the neutrality condition
\begin{equation} \label{NC ins}
	a+\int \bfs{\tau} =0,
\end{equation}
we define the one-leg operator
\begin{equation}\label{Psi}
	\Psi (z) \big(\equiv \Psi_{\bfs{\beta}}(z, \bfs{z})\equiv \Psi_{\bfs{\beta}}[ a \cdot z+ \bfs{\tau} ] \, \big) := \OO_{\bfs\beta} [ a\cdot z +\tfrac12 \bfs{\tau} \, , \, \tfrac12 \bfs{\tau} ].
\end{equation}
Notice that $\Psi$ is a scalar field with respect to the each node $q_j$ of the background charge $\bfs\beta$. On the other hand, by \eqref{Cou dimension}, the conformal dimensions of $\Psi$ with respect to $z$ and $z_j$ are given by 	
\begin{equation} \label{oneleg dim}
	h_{z}=\tfrac{6-\kappa}{2\kappa}=h_{1,2},
	\qquad 
	h_{j}=h_{*j}=\tfrac18 \tau_j^2-\tfrac12 \tau_j b.
\end{equation}
It follows from \eqref{vertex beta} and \eqref{Psi} that 
\begin{equation} \label{Partition}
\E \Psi(z)=\frac{ C_{(b)} [a \cdot z+\tfrac12 (\bfs{\tau}+ \bfs{\beta})\,,\, \tfrac12 (\bfs{\tau}+ \bfs{\beta}) ] }{ C_{(b)}{[\tfrac12 \bfs{\beta}\,,\, \tfrac12 \bfs{\beta}]} } .
\end{equation} 
From now on, we call 
$ Z_{\bfs{\beta}} :=\E \Psi_{ \bfs \beta }$ the \emph{partition function} following the terminology in CFT literature, see e.g., \cite[Section 11.3]{MR1424041}.
Set 
$\Lambda_{\bfs{\beta}} \equiv \Lambda_{\bfs{\beta} }(z,\bfs{z} ):= \kappa \, \pa_{z} \log Z_{\bfs{\beta}}.
$
By \eqref{Partition} and \eqref{Coulomb ftn}, the field $\Lambda_{\bfs{\beta}}$ is evaluated in the identity chart of $\CC_r$ as
\begin{equation} \label{Lambda bulk}
	\Lambda_{\bfs{\beta}}= \sqrt{\dfrac{\kappa}{2}} 
	\Big[ \sum \dfrac{\beta_j}{2} \Big( H_\chi(r, z-z_j)+H_\chi(r,z-\bar{z}_j) \Big)
+ \sum \dfrac{\tau_j}{2} \Big( H_\chi(r, z-q_j)+H_\chi(r,z-\bar{q}_j) \Big) \Big].
\end{equation}
Now we derive the level two degeneracy equation for $\Psi$. 

\begin{lem} \label{psi current}
	The one-leg operator $\Psi$ is a current primary field with charges $q = a, q_*=0$ at $z$, i.e.,	$$
	J_0 \Psi=-ia\Psi, \;\;\; J_0\overline{\Psi}=0,\;\;\; J_n\Psi=J_n \overline{\Psi}=0 \qquad (n\ge1). 
	$$
\end{lem}
\begin{proof} 
	To show the above equations, it is enough to show that as $\zeta \to z$,
	$$
	J_{(0)}(\zeta)\Psi(z) \sim -ia\frac{\Psi(z)}{\zeta-z}, \qquad  J_{(0)}(\zeta)\overline{\Psi}(z)\sim 0
	$$
	hold in the identity chart. By Wick's formula, we have
	\begin{align*}
		J_{(0)}(\zeta) \Psi(z) &= J_{(0)}(\zeta) \odot \Psi(z) + ia\E [J_{(0)}(\zeta)\Phiplus_{(0)}(z) ]\Psi(z)
	\\
	&	- \sum \frac{\tau_j}{2}\E [J_{(0)}(\zeta) \ti{\Phi}_{(0)}(z_j) ]\Psi(z)-\sum \frac{\beta_j}{2}\E [J_{(0)}(\zeta) \ti{\Phi}_{(0)}(q_j) ]\Psi(z),
	\end{align*} 
	where $\ti{\Phi}_{(0)}$ is the dual boson in \eqref{dual boson}. It is also easy to see that a similar expression holds for $\overline{\Psi}$.
	Now lemma follows from \eqref{formal ++ cor} and \eqref{h-zero}.
\end{proof}

By Proposition \ref{level2}, we immediately obtain the following level two degeneracy equation of $\Psi$. 

\begin{prop} \label{level2_ER}
	If $2a(a+b)=1$, then
	$$
	T_{\bfs\beta}\ast_{z} \Psi=\frac{1}{2a^2}\pa^2_{z}\Psi(z).
	$$
\end{prop}

We are now ready to show BPZ equation, Theorem~\ref{BPZ_ER}.
\begin{proof}[Proof of Theorem~\ref{BPZ_ER}]	
	By Proposition \ref{level2_ER}, we have
	$$
	\frac{1}{a^2} \pa^2_{z} \E \Psi(z) \XX = 2 \E ( T_{\bfs\beta}\ast_{z}\Psi ) \XX. 
	$$
	Then it follows from Corollary~\ref{Wequation_ER} that 
	\begin{align*}
		2 \E ( T_{\bfs\beta}\ast_{z}\Psi ) \XX 
		= 2 \E ( A_{\bfs\beta}\ast_{z}\Psi ) \XX + 2\E T_{\bfs\beta} (z) \, \E \Psi(z) \XX,
	\end{align*}
	which completes the proof. 
\end{proof}

We now discuss how the one-leg operator $\Psi$ acts on Fock space fields in $D_r$ as a boundary condition changing operator. Given marked boundary points $p \in \R$ and $ \{ q_j \}, \{\xi_j\}\in \pa \CC_r$, set $\bfs\beta=\sum \beta_j \cdot q_j$ and $\bfs{\tau}= \sum \tau_j \cdot \xi_j$. By abuse of notation, for $z_j \in \CC_r$, we sometimes use the same symbol $\bfs{\tau}$ for $\sum\tau_j \cdot z_j$. 
From now on, we write $w_z=w(z)$, $w_p:=w(p)$, $w_j:=w(\xi_j)$ and $\ti{w}_j:=w(q_j)$, where $w: D_r \to \CC_r$ is a conformal transformation. 
The insertion of field $\Psi/\E \Psi$ produces an operator $\XX \mapsto \wh{\XX} $ on Fock space fields. It is given by 
$$ 
\pa\XX \mapsto \pa\wh{\XX},\qquad	\bar{\pa}\XX \mapsto \bar{\pa}\wh{\XX}, \qquad \alpha\XX+\beta\YY \mapsto \alpha\wh{\XX}+\beta\wh{\YY},\qquad	\XX \odot \YY \mapsto \wh{\XX}\odot \wh{\YY}
$$ 
and by the formula:
\begin{align*}
\wh{\Phi}_{ \bfs{\beta} }(z)=\Phi_{ \bfs{\beta} }(z)+2a \arg \Theta_\chi(w_p-w_z)&+\sum \tau_j \arg \Big\{ \Theta_\chi(w_{j}-w_z)\Theta_\chi(\bar{w}_{j}-w_z) \Big \}.
\end{align*}

When $\kappa=4$, this insertion procedure produces specific height gap of GFF, which gives rise to the coupling with SLE($4,\Lambda$). For instance, in \cite{MR2651436}, Hagendorf, Bernard and Bauer studied coupling of GFF with Dirichlet boundary condition and SLE($4,\Lambda$) with one force point $q=\xi_1$ using the one-leg operator $\Psi = \mathcal{O}[a\cdot p -a\cdot q,\bfs{0}]$. 
In the case of multiple force points, the (multivalued) harmonic function $\wh{\varphi}:= \E \wh{\Phi}$ in $\bar{\CC}_r \sm \{p,\xi_1,\cdots, \xi_N\}$ has piecewise Dirichlet boundary conditions having additional jump of $2\pi a$ at $p$ and $2\pi \tau_j$ at $\xi_j$'s. Due to the neutrality condition, we require that all jumps should add up to $0$. This type of boundary condition of GFF was studied by Izyurov and Kyt\"{o}l\"{a} \cite{MR3010393}. 

As in Subsection~\ref{Main result sub}, we let 
$$
\wh{\E}[\mathcal X]=
\displaystyle \E [ e^{\odot ia \Phi_{(0)}^+(p)-\sum \frac{\tau_j}{2} \ti{\Phi}_{(0)}(\xi_j) } \mathcal X ].
$$
Then it is easy to check that $\wh{\E}[\XX]=\E[\wh{\XX}]$ for all $\XX\in \FF_{\bfs\beta}$.

The operator $\XX \mapsto \wh{\XX}$ acting on $\FF_{(b)}$ plays a role in imposing background charge $a \cdot p+\bfs\tau$. 
To be precise, by letting $\bfs \beta= a \cdot p + \bfs \tau,$ we have
\begin{equation} \label{insertion bcm gen}
 \wh{\XX} \in \FF_{ \bfs \beta }, \qquad \text{if } \XX \in \FF_{(b)}.   
\end{equation}
For a deeper discussion of the relation between the insertion procedure and the background charge modifications, we refer the reader to \cite{KM}. 

\begin{egs*}  
It is instructive to compare the following examples with those in Subsection~\ref{Subsec b.c.m.}. 
\begin{itemize}
 \item The current field $\wh{J}$ is a PS-form of order $ib$. We have 
	$$
	\wh{J}(z)=J(z)+\frac{ia w'_z}{2} H_\chi(w_p-w_z) +\frac{iw'_z}{4}\sum \tau_j \Big ( H_\chi(w_{j}-w_z)+H_\chi(\bar{w}_j-w_z) \Big ).
	$$
	\item The Virasoro field $\wh{T}$ is an S-form of order $c/12$. 
	We have 
	\begin{align*}
		\wh{T}(z)&= -\frac{1}{2} \wh{J} \ast \wh{J}+ ib \pa \wh{J}
		= T-(\wh{\jmath}-\jmath)J_{(0)}-\frac{1}{2} ( \wh{\jmath}^2-\jmath^2 )+ib \pa (\wh{\jmath}-\jmath )
		\\
		&= A_{(0)}-\wh{\jmath}J_{(0)}+ib\pa J_{(0)}+\frac{c}{12}S_w+(w'_z)^2 \E T-\frac{1}{2} ( \wh{\jmath}^2-\jmath^2 )+ib \pa(\wh{\jmath}-\jmath),
	\end{align*}
	where $\jmath:=\E J$ and $\wh{\jmath}:=\E \wh{J}$. By Proposition~\ref{Virasoro2_ER}, its evaluation in the identity chart of $\CC_r$ is given as 
	\begin{align*}
	\E \wh{T}(z)&=\frac18 \Big[a H_\chi(p-z) +\frac{1}{2}\sum \tau_j \Big ( H_\chi(\xi_j-z)+H_\chi(\bar{\xi}_j-z) \Big ) \Big]^2
	\\
	&+\frac{b}{2} \Big[ a H'_\chi(p-z) +\frac{1}{2}\sum \tau_j \Big ( H'_\chi(\xi_j-z)+H'_\chi(\bar{\xi}_j-z) \Big ) \Big]+\dfrac{\zeta_r(\pi)}{2\pi}-\frac{1}{4(r+\chi)} .
	\end{align*}
	\item The multi-vertex field $\wh{\OO}{[\bfs{\sigma},\bfs{\sigma_*}]}$ whose nodes do not intersect with $\mathcal{S}_\Psi$ is a differential with conformal dimension $(h_j,h_{*j})$. We have 
	\begin{align*}
	\E \wh{\OO}{[\bfs{\sigma},\bfs{\sigma_*}]}&=	C_{(b)}{[\bfs{\sigma},\bfs{\sigma_*}]} \E V^\odot [a\cdot p+\bfs{\tau}/2,\bfs{\tau}/2 ] V^\odot[\bfs{{\sigma}},\bfs{{\sigma_*}}] 
	\\
	&=C_{(b)}{[\bfs{\sigma},\bfs{\sigma_*}]} e^{ -\E \Phi_{(0)} [a\cdot p+\bfs{\tau}/2,\bfs{\tau}/2 ] \E \Phi_{(0)}[\bfs{{\sigma}},\bfs{{\sigma_*}}] }
	\end{align*} 
	and its evaluation in the identity chart of $\CC_r$ is given as (up to a multiplicative constant) 
 \begin{align*}
 \E \wh{\OO} &=\Theta'_\chi(0)^{\sum \frac{1}{2}( \sigma_j^2+\sigma_{*j}^2 ) } \prod \Theta_\chi(p-z_j)^{a\sigma_j} \Theta_\chi(p-\bar{z}_j)^{a \sigma_{*j} } \Theta_\chi(z_j-\bar{z}_j)^{\sigma_{j} \sigma_{*j}}
 \\
 &\times \prod_{ j,k} ( \Theta_\chi(\xi_j-z_k) \Theta_\chi(\bar{\xi}_j-z_k) )^{\tau_j \sigma_k/2 } ( \Theta_\chi(\xi_j-\bar{z}_k) \Theta_\chi(\bar{\xi}_j-\bar{z}_k) )^{ \tau_j \sigma_{*k}/2 } 
 \\
 &\times \prod_{j < k} \Theta_\chi(z_j-z_k)^{\sigma_j\sigma_k}\Theta_\chi(\bar{z}_{j}-z_k)^{\sigma_{*j}\sigma_k} \Theta_\chi(z_j-\bar{z}_k)^{\sigma_j\sigma_{*k}}\Theta_\chi(\bar{z}_j-\bar{z}_k)^{\sigma_{*j}\sigma_{*k}}.
 \end{align*}
\end{itemize}

\end{egs*}

Suppose $\bfs \beta$ is a background charge with $\textup{supp }\bfs \beta=\{ q_k \}$ ($q_k\in D$). As discussed in Subsection~\ref{Ward's identities_ER}, we need the following residue form of Ward's identity at $q_k$ to complete the proof of Theorem~\ref{Thm_Ward}. For reader's convenience, we borrow its proof from \cite{KM}. 

\begin{lem} \label{Lem_Ward beta node}
For any $\XX_{ \bfs \beta } \in \FF_{ \bfs \beta }$ with nodes outside $\textup{supp }\bfs \beta=\{ q_k \}$ and a vector field $v$ holomorphic in a neighborhood of each of $q_k$, we have
\begin{equation}
\frac{1}{2\pi i} \oint_{(q_k)} v \, \E A_{\bfs \beta } \XX_{\bfs \beta } =\E \LL_v^+(q_k) \XX_{\bfs \beta }. 
\end{equation}
\end{lem}
\begin{proof}
Since $\XX_{\bfs \beta }$ is a scalar with respect to $q_k$, we have 
$$
\E \LL_v^+(q_k) \XX_{\bfs \beta }= v(q_k) \pa_{q_k}\E \XX_{\bfs \beta }.
$$
Moreover, it follows from $\E \XX_{\bfs \beta }= \E [ e^{ \odot i \Phi[ \bfs{\beta}/2, \bfs{\beta}/2 ]  } \XX_{  \bfs 0 } ]$ that
\begin{equation}
\pa_{q_k} \E \XX_{\bfs \beta }= \frac{i \beta_k}{2} \E \XX_{ \bfs 0 } J(q_k) \odot e^{ \odot i \Phi[ \bfs{\beta}/2, \bfs{\beta}/2 ]  }. 
\end{equation}
On the other hand, by the residue calculus using \eqref{jmath beta}, we have
$$
\frac{1}{2\pi i} \oint_{(q_k)} v \, \E A_{\bfs \beta } \XX_{\bfs \beta }= \frac{i \beta_k }{2} \,v(q_k) \E J(q_k) \XX_{\bfs \beta } . 
$$
Thus all we need to show is 
\begin{equation}\label{Ward beta node eq}
\E J(q_k) \XX_{\bfs \beta }=  \E \XX_{ \bfs 0 } J(q_k) \odot e^{ \odot i \Phi[ \bfs{\beta}/2, \bfs{\beta}/2 ]  }. 
\end{equation}

By Wick's calculus and \eqref{insertion bcm}, we have that for $q_k' \not= q_k$, 
$$
J(q_k') \odot e^{ \odot i \Phi[ \bfs{\beta}/2, \bfs{\beta}/2 ]  } =  J(q_k') e^{ \odot i \Phi[ \bfs{\beta}/2, \bfs{\beta}/2 ]  } + ( \jmath(q_k') - \jmath_{ \bfs \beta }(q_k') ) e^{ \odot i \Phi[ \bfs{\beta}/2, \bfs{\beta}/2 ]  }. 
$$
Therefore by \eqref{insertion bcm gen}, we have 
\begin{align*}
\E \XX_{ \bfs 0 } J(q_k') \odot e^{ \odot i \Phi[ \bfs{\beta}/2, \bfs{\beta}/2 ]  } &= \E \XX_{ \bfs 0 }  J(q_k') e^{ \odot i \Phi[ \bfs{\beta}/2, \bfs{\beta}/2 ]  }+    ( \jmath(q_k') - \jmath_{ \bfs \beta }(q_k') ) \E \XX_{ \bfs 0 }  e^{ \odot i \Phi[ \bfs{\beta}/2, \bfs{\beta}/2 ]  } 
\\
&= \E \XX_{ \bfs \beta }  J_{ \bfs \beta }(q_k')- \jmath_{ \bfs \beta }(q_k')  \E \XX_{ \bfs \beta } = \E J(q_k') \XX_{ \bfs \beta }. 
\end{align*}
Now \eqref{Ward beta node eq} follows by taking the limit $q_k' \to q_k.$
\end{proof}

\subsection{Null-vector equations} \label{Subsec_NVE}

Throughout this subsection, until further notice, we consider the case that all marked points are on the outer boundary component. 
This allows us to simply consider the one-leg operator of the form $\Psi(z) = \OO_{\bfs\beta} [a\cdot p +\bfs{\tau}\,,\, \bfs{0}]$.
By \eqref{Partition}, the associated partition function 
$$ Z(p) \big(\equiv Z_{ \bfs{\beta} }(p,\bfs \xi) \equiv  Z_{ \bfs{\beta} }[a\cdot p+\bfs{\tau}] \big):=\E \Psi (p)
$$ 
is evaluated in the cylinder $\CC_r$ as
\begin{equation}
	Z=
		\displaystyle C_{\bfs{\tau} } \prod_j \Theta_\chi(r,p-\xi_j) ^{a \tau_j} \prod_j \Theta_\chi(r,p-q_j) ^{a \beta_j} \prod_{j < k} \Theta_\chi(r,\xi_j-\xi_k)^{\tau_j \tau_k} \prod_{j, k} \Theta_\chi(r,\xi_j-q_k)^{\tau_j \beta_k}, 
\end{equation}
where $C_{ \bfs{\tau} }=\Theta'(r,0)^{\frac{a^2}{2}+\sum \frac{\tau_j^2}{2}}$. It also follows from the neutrality condition that 
\begin{align} \label{Partition gen}
	\begin{split}
		Z&=
		\displaystyle C_{\bfs{\tau} } 	 \exp\Big( -\frac{ (ap+\int z \,d\bfs{\tau} +\int z \,d\bfs{\beta})^2-(\int z \,d\bfs{\beta})^2 }{4(r+\chi)} \Big) 
		\\
		&\times \prod_j \Theta(r,p-\xi_j) ^{a \tau_j} \prod_j \Theta(r,p-q_j) ^{a \beta_j} 
		 \prod_{j < k} \Theta(r,\xi_j-\xi_k)^{\tau_j \tau_k} \prod_{j, k} \Theta(r,\xi_j-q_k)^{\tau_j \beta_k}.
	\end{split}
\end{align}
We write $\lambda_p$, $\lambda_j$ for the (holomorphic) conformal dimension of $\Psi$ with respect to $p$ and $\xi_j$. Then by \eqref{Cou dimension}, we have
\begin{equation} \label{Partition gen dim}
	\lambda_p=\frac{6-\kappa}{2\kappa},\qquad \lambda_j=\frac12 \beta_j^2-\beta_j b.
\end{equation} 

It is worth to pointing out that by the neutrality condition, the sum of all the exponents (of theta functions) in the expression \eqref{Partition gen} vanishes, i.e.,
\begin{align*}
\frac{a^2}{2}+\sum_j \frac{\tau_j^2}{2}+\sum_j a\tau_j+\sum_k a \beta_k + \sum_{j<k} \tau_j \tau_k+\sum_{j,k} \tau_j \beta_k
=\frac12 \Big( a+\sum_j \tau_j+\sum_k \beta_k \Big)^2-\Big( \sum_k \beta_k \Big)^2=0.	
\end{align*}
Therefore it follows from asymptotic behaviors of theta function (see e.g., \cite[pp.65-69]{MR808396})
\begin{equation} \label{degen theta}
	\lim\limits_{r\to \infty} e^{r/4} \Theta(r,x) =2\sin\Big(\frac{x}{2}\Big), \qquad \lim\limits_{r\to \infty}e^{r/4}\Theta'(r,0)=1
\end{equation}
that the partition function $Z$ has a non-trivial limit as $r \to \infty$. Indeed, the partition functions degenerate to (up to a multiplicative constant)
\begin{align}
	\begin{split}
		\label{degen gen}
		Z_\infty:=\lim\limits_{r\to \infty} Z&= \prod_{j} \sin ^{a \tau_j}\Big( \frac{p-\xi_j}{2} \Big) \prod_{j} \sin ^{a \beta_j}\Big( \frac{p-q_j}{2} \Big) 
		\\
		&\times \prod_{j < k} \sin ^{\tau_j \tau_k} \Big( \frac{\xi_j-\xi_k}{2} \Big) \prod_{j,k} \sin ^{\tau_j \beta_k} \Big( \frac{\xi_j-q_k}{2} \Big).
	\end{split}
\end{align}

As a consequence of BPZ equation (Theorem~\ref{BPZ_ER}), we obtain a general form of the null-vector equation for the partition function 
$Z$. Here and subsequently, $\pa$ and $'$ stand for the differentiation with respect to $p$. 
\begin{cor} \label{pre screen null}
	The partition function $Z$ satisfies 
	\begin{align}
		\begin{split}
			\pa_r Z =\frac{\kappa}{2} \pa^2 Z&+ \sum_j \Big( \lambda_j H'(r,p-\xi_j)-H(r,p-\xi_j)\pa_{\xi_j} \Big) Z
			-\sum_{j} H(r,p-q_j)\pa_{q_j}Z +C(r,\bfs{q}) Z,
		\end{split}
	\end{align}
	where
	\begin{align} \label{null vec const}
		\begin{split}
			C(r,\bfs{q})
			&=-\dfrac{1}{16}\Big[ \sum \beta_k \Big\{ H_\chi(r,q_k-p)+H_\chi(r,\bar{q}_k-p) \Big\} \Big]^2
			\\
			&-\dfrac{b}{2} \, \sum \beta_k \Big\{ H'_\chi(r,q_k-p)+H'_\chi(r,\bar{q}_k-p) \Big\}-\frac{6}{\kappa} \frac{\zeta_r(\pi)}{\pi}+\frac{1}{2(r+\chi)}. 
		\end{split}
	\end{align}
\end{cor}
\begin{proof}
	By applying trivial string $\XX \equiv 1$ to Theorem~\ref{BPZ_ER}, we have 
	$$
	\pa_r Z=\frac{\kappa}{2} \pa^2 Z-\LL_{v_{p}}^+Z +C(r,\q) Z, \qquad 	C(r,\q)=-2h_{1,2} \frac{\zeta_r(\pi)}{\pi}-2\E T_{\bfs\beta}.
	$$ 
 Now corollary follows from Proposition~\ref{AT beta} and the fact that 
	$$\LL_{v_{p}}^+Z = \sum \Big( v_p(\xi_j)\pa_{\xi_j}+\lambda_j v'_p(\xi_j) \Big) Z+\sum v_p(q_j) \pa_{q_j} Z.$$
\end{proof}

\begin{eg*}
	Let us consider the simplest case that 
	$\bfs\beta=\bfs 0$ and $\bfs\tau=-a\cdot q$. Write
	$\Psi_*:= \mathcal{O}[a\cdot p -a\cdot q,\bfs{0}]$. 
	By \eqref{Partition gen}, we have 
	\begin{equation}\label{k=4}
	 Z_*:=\E \Psi_*= \Theta'_\chi(r,0)^{\frac{2}{\kappa}}\Theta_\chi(r,p-q) ^{-\frac{2}{\kappa}}. 
	\end{equation}
	For $\chi=0$ and $\kappa=4$, the SLE associated with the partition function \eqref{k=4} was studied in \cite{MR3010393,MR2651436}. 
	
	By translation invariance, we have $\pa_q Z_*= -Z'_*$. Also by \eqref{Partition gen dim}, the conformal dimension $\lambda_q$ of $\Psi_*$ with respect to $q$ is given as 
	$$\lambda_q=\frac12 a^2+ab= \frac{1}{2}-\frac{1}{\kappa}.$$
	Therefore it follows from Corollary~\ref{pre screen null} that 
	$$
	\pa_r Z_*= \frac{\kappa}{2} Z_*''+H Z_*'+\Big(\frac{1}{2}-\frac{1}{\kappa} \Big) H' Z_*+C(r) Z_*, \qquad C(r):=-\frac{6}{\kappa} \frac{\zeta_r(\pi)}{\pi}+\frac{1}{2(r+\chi)}.
	$$
	Notice that for $\kappa=4$, this corresponds to the null-vector equation \eqref{Lawler Zhan null}.
\end{eg*}

\subsection{BPZ-Cardy equations} \label{Main Theorem_sub} 

In this subsection we prove BPZ-Cardy type equation, a key ingredient to connect CFT with SLE theory, and complete the proof of Theorem~\ref{main}. 

Fix a marked point $p=\xi$ on the outer boundary. Recall the definitions
$$ Z_{\bfs{\beta} } (z,\bfs{z}) :=\E \Psi_{\bfs{\beta}}(z, \bfs{z}), \qquad \Lambda_{\bfs{\beta} }(z,\bfs{z} ):= \kappa \,  \pa_{z} \log Z_{\bfs{\beta}} (z,\bfs{z}).$$ 
For a Fock space functional $\mathcal X$ whose nodes do not intersect marked boundary points, set
$$
R_{z}:=\wh{\E}_{z}[\XX]:=\frac{\E[\Psi_{\bfs{\beta}}(z, \bfs{z})\XX]}{ Z_{\bfs{\beta} } (z,\bfs{z}) }, \qquad 	R_{\xi}:=\wh{\E}_{\xi}[\XX]=
		\displaystyle \E\Big[e^{\odot ia \Phiplus_{(0)}(\xi)-\sum \frac{\tau_j}{2} \ti{\Phi}_{(0)}(\xi_j)}\XX\Big]. 
$$
Let us denote 
$ \check{\LL}_{v_z} := \LL_{v_z}(\bar{D}_r \sm ( \{ z \} \cup \mathcal S_\Psi ) )$, i.e., $\check{\LL}_{v_z} \Psi Y(z)X:= \Psi Y(z)\check{\LL}_{v_z}X$ if the nodes of $\XX$ do not intersect $z$ and $\mathcal S_\Psi$.

Now we prove the following BPZ-Cardy equation.
\begin{prop}\label{BPZ-C_ER}
	For any $\XX \in \mathcal F_{\bfs{\beta}}$, in the $\mathcal C_r$-uniformization, we have		
	$$
	\frac{1}{a^2}\pa^2_{\xi} R_{\xi}+\Lambda_{ \bfs{\beta} } (\xi,\bfs{\xi})\pa_{\xi} R_{\xi}= \check{\LL}_{v_{\xi}} R_{\xi}+ \Big( \sum v_{\xi}(\xi_j) \pa_{\xi_j}+\bar{v}_{\xi}(\xi_j)\bar{\pa}_{\xi_j} \Big) R_{\xi}+\pa_rR_{\xi}.
	$$
Here $\pa_{\xi}=\pa+\bar{\pa}$ is the operator of differentiation with respect to the real variable $\xi$ whereas $\pa_{\xi_j}$ denotes the one with respect to the complex variable $\xi_j$.
\end{prop}
\begin{proof}	Let us write $\bfs{\zeta}=S_{\XX}$. Since $\Psi(z)$ is a holomorphic differential with respect to $z$, we rewrite Theorem~\ref{BPZ_ER} as
	\begin{align*}
		\frac{1}{a^2}\pa^2_{z} \Big( R_{z} Z_{\bfs{\beta} } (z,\bfs{z}) \Big) &= \Big( \LL_{v_{z}}^+ +\LL_{v_{\bar{z}}}^- \Big)(\bfs{\zeta},\bfs{z},\q) R_{z} Z_{\bfs{\beta} } (z,\bfs{z}) 
		+\Big( \pa_r+ 2h_{1,2} \frac{\zeta_r(\pi)}{\pi}+2\E T_{\bfs\beta}(z) \Big) R_{z} Z_{\bfs{\beta} } (z,\bfs{z}). 
	\end{align*} 
	By applying the trivial string $\mathcal X \equiv 1$ to the above, we have
	\begin{align*}
			\frac{1}{a^2}\pa^2_{z} Z_{\bfs{\beta} } (z,\bfs{z})&=\Big( \LL_{v_{z}}^+ +
		\LL_{ v_{ \bar{z} } }^- \Big)(\bfs{\zeta},\bfs{z},\q) Z_{\bfs{\beta} } (z,\bfs{z})
		+\Big( \pa_r+ 2h_{1,2} \frac{\zeta_r(\pi)}{\pi}+2\E T_{\bfs\beta}(z) \Big) Z_{\bfs{\beta} } (z,\bfs{z}). 
	\end{align*}
	Subtracting the above two equations, we obtain
	$$
	\frac{1}{a^2}\pa^2_{z} R_{z}+\Lambda_{ \bfs{\beta} }(z,\bfs{z})\pa_{z} R_{z}=\Big( \LL_{v_{z}}^+ +\LL_{v_{\bar{z}}}^- \Big)(\bfs{\zeta},\bfs{z},\q) R_{z}+\pa_rR_{z}.
	$$
	We now take the limit $z \to \xi$ and $z_j \to \xi_j$. Since $\pa_{\xi}=\pa+\bar{\pa}$, we obtain 
	$$
	\frac{1}{a^2}\pa^2_{\xi} R_{\xi}+\Lambda_{\bfs{\beta} } (\xi,\bfs{\xi})\pa_{\xi} R_{\xi}=\LL_{v_{\xi}}(\bfs{\zeta},\bfs{q})R_{\xi}+\LL_{v_{\xi}}(\bfs{\xi})R_{\xi}+\pa_rR_{\xi}.
	$$
	Now proposition follows from the fact $R_\xi$ is a scalar field with respect to each $\xi_j$.
\end{proof}	

Now we are ready to prove Theorem~\ref{main}. 

\begin{proof}[Proof of Theorem~\ref{main}] 
	Let us denote $R_{\xi}=\wh{\E}_{\xi}[\XX]$. Then 
	$$
	M_t=m \Big(r-t,\xi_t,\ti{g}_t(\bfs{\xi}) \Big), \qquad m(r,\xi,\bfs{\xi},t )=\Big( R(r,\xi,\bfs{\xi}) \| \ti{g}_t^{-1}\Big).
	$$ 
	Hence, by It\^{o}'s formula, we have
	\begin{align*}
		dM_t = -\pa_r m\, dt+ \pa_\xi m \, d\xi_t+ \pa^2_\xi m \frac{d\langle\xi\rangle_t}{2}
		+ ( L_t R\| \ti{g_t}^{-1})\,dt
		+ \sum_{j=1}^N \Big[ \pa_{\xi_j}m \, d\xi_j(t) + \bar{\pa}_{\xi_j}m \, d \overline{\xi_j(t)}\Big],
	\end{align*}
	where 
	$$
	L_t:= \frac{d}{ds} \Big|_{s=0} \Big( R (r-t, \xi_t, \ti{g}_t(\bfs{\xi})) \| \ti{g}_{t+s}^{-1} \Big).
	$$ 
	Let $f_{s,t}= \ti{g}_{t+s}^{-1} \circ \ti{g}_t^{-1} $. Then since the time-dependent flow $f_{s,t}$ satisfies the ODE
	$$
	\frac{d}{ds} f_{s,t} = -v_{\xi_{t+s}}f_{s,t}, \qquad v_{\xi}(z)=H(r,\xi-z),
	$$
	we obtain 
	$$
	f_{s,t}=\text{id}-s \, v_{\xi_t}+o(s).
	$$ 
	From the fact that the fields in $\mathcal{F}_{\bfs{\beta}}$ depends smoothly on local charts, it follows
	$$
	L_t= - \Big( \check{ \mathcal L}_{v_{\xi_t}} R_\xi \| \ti{g}_t^{-1} \Big).
	$$
	By the Loewner's equation $d\xi_j(t)=-v_{\xi_t}(\xi_j(t))$, it follows from Proposition~\ref{BPZ-C_ER} that the drift term of $dM_t$ vanishes. This completes the proof of theorem.
\end{proof}

\begin{eg*}
	The simplest examples of SLE martingale-observables are the 1-point function $M(z)$ and the 2-point function $N(z_1,z_2)$ of the bosonic field $\wh{\Phi}_{(b)}$ when $\bfs{\beta}=\bfs 0$. 
 For instance, when $\chi=\infty$, for all $\kappa$, 
	$$
	dM_t(z)=-\sqrt{2}\,\Im\, H (w_t(z))\,dB_t+ \Big( a+\sum \tau_j \Big) \Im \, \Big( \frac{H^2}{2}+H' \Big) (w_t(z) )\, dt.
	$$
	Due to the neutrality condition \eqref{NC ins}, $M_t$ is driftless. 
	For the $2$-point function, one can easily show that
	\begin{align*}
	dN_t(z_1,z_2)&=	-\sqrt{2} \, \Big( M_t(z_1) \, \Im\,H(w_t(z_2))+ M_t(z_2) \, \Im\,H( w_t(z_1)) \Big) \, dB_t
		\\
		& +\Big( a+\sum \tau_j \Big) \Big \{ M_t(z_1)\Big( \frac{H^2}{2}+H' \Big)(w_t(z_2))+ M_t(z_2)\Big( \frac{H^2}{2}+H' \Big)(w_t(z_1)) \Big \}\, dt.
	\end{align*}
	Thus the drifts disappear again by the neutrality condition. We remark that this is equivalent to Hadarmard's variational formula. 
 In \cite{MR3010393}, Izyurov and Kyt\"{o}l\"{a} studied these types of martingale-observables for $\kappa=4$ with Dirichlet boundary condition. 
\end{eg*}

\begin{eg*}
Let us consider the case that there is only one force point $q$ on $\R_r$. By the neutrality condition, we impose the symmetric charge $(-a/2,-a/2)$ at $q$. 

By \eqref{Lambda bulk}, the drift function $\Lambda$ is given by 
\begin{equation} \label{Lambda_pq cro}
\Lambda(p,q)=H_I(r,p-\Re\, q)+\frac{p-\Re \, q}{r+\chi}. 
\end{equation}
Thus the driving process $\xi_t$ is given by 
\begin{equation}
d\xi_t =\sqrt{\kappa}\,dB_t+\Big( H_I(r-t, \xi_t-\Re \,\tilde{g}_t(q) ) +\frac{\xi_t- \Re\,\tilde{g}_t(q) }{r+\chi-t} \Big)\,dt,\qquad \xi_{0}=p. 
\end{equation} 
Note that in the degenerate limit $r \to \infty$, the associated SLE($\kappa,\Lambda$) converges to the law of radial SLE$(\kappa)$. We remark that when $\kappa=4$, the SLE($4,\Lambda$) corresponds to the scaling limit of level lines of \emph{compactified} GFF (cf. \cite{MR3010393}), see Figure~\ref{Fig_cGFF}.
\begin{figure}[h!]
	\begin{center}
		\includegraphics[width=0.8\textwidth]{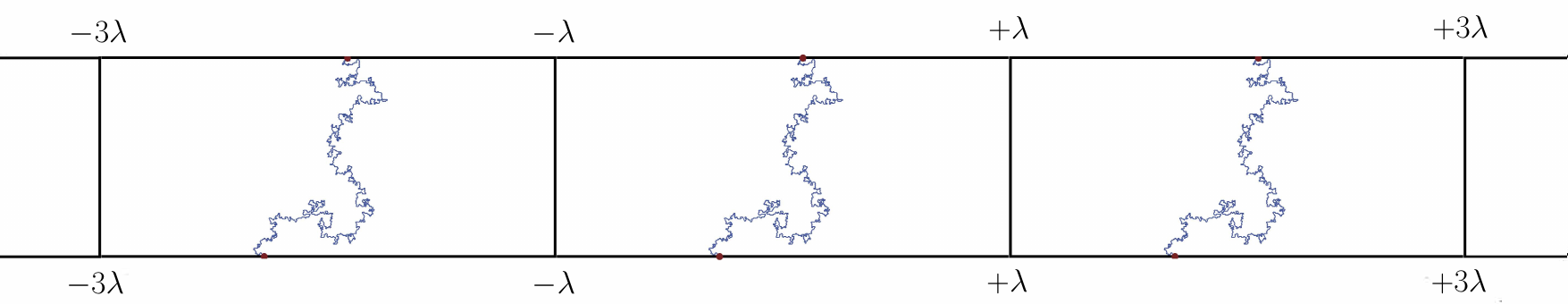}
	\end{center}
\caption{Level lines of compactified GFF}\label{Fig_cGFF}
\end{figure}

Using the bosonic observables, one can show the following version of the left-passage probability of SLE trace. For this, recall that $\ti{\Theta}:=\Theta_{0}$. 

\begin{prop}\label{Prop_LPP SLE4}
For $\chi=0$ let $\eta$ be the trace of SLE($\kappa,\Lambda$), where $\Lambda$ is given by \eqref{Lambda_pq cro}. Then $\eta$ a.s. ends at the force point $q$. 
Furthermore, when $\kappa=4$, the probability $\P(z)$ that $z \in \S_r$ is left to $\eta$ is given by 
\begin{equation} \label{LPP}
\P(z)=\frac{1}{\pi} \Big( \arg \tilde{\Theta}(r,z-p)- \arg \tilde{\Theta}(r,z-q) \Big).
\end{equation}
\end{prop}
\begin{proof}
By Loewner's equation, the angle-difference process $X_t:=\xi_t-\Re \, \tilde{g}_t(q) $ satisfies
\begin{equation}
	dX_t=\sqrt{\kappa}\,dB_t +\frac{X_t\,dt}{r+\chi-t}, \qquad X_0=p-\Re\,q. 
\end{equation} 
Therefore the process $X_t$ is realized as a Brownian bridge starting from $X_0$, which reaches $0$ at time $t=r+\chi$, see e.g., \cite{MR2001996}. Thus when $\chi=0$, we have $\lim_{t \uparrow r} X_t=0$, which means $\eta$ a.s. ends at the force point $q$.

For $\kappa=4$, by Theorem~\ref{main}, the bosonic observable
$$M_t(z):=\frac{1}{\pi} \Big( \arg \tilde{\Theta}(r-t,w_t(z))- \arg \tilde{\Theta}(r-t,w_t(z)-w_t(q)) \Big)$$ is a (bounded) martingale. On the other hand, note that $M=M_0$ is a harmonic function satisfying the boundary condition 
$$
M(z)= 
\begin{cases}
	0 &\text{if}\quad z \in (0,p)\cup (ir,q),
	\\
	1 &\text{if}\quad z \in (p,2\pi) \cup (q,2\pi+ir).
\end{cases}
$$
Let us write $E$ the event that $z$ is left to $\eta$. Then we have
$$
\lim_{t \uparrow r} M_t(z)=\begin{cases}
	0 &\text{on }E,
	\\
	1 &\text{otherwise}.
\end{cases}
$$
Then proposition follows from optional stopping theorem. 
\end{proof}

\end{eg*}

\section{Further generalizations}\label{Section_further generalizations}

In this section we discuss several methods to extend the class of SLE martingale-observables based on some intrinsic ideas from CFT. 

\subsection{Non-atomic background charges} \label{Subsec_NVE msr}

In \cite{MR3729653}, Powell and Wu studied GFF with general Dirichlet boundary conditions and its coupling with SLE in a simply connected domain, see also a recent work \cite{lupu2021level}. 
Intending to account for bosonic fields with such boundary conditions from the viewpoint of CFT, we implement general measure-valued background charges beyond the linear combination of the atomic masses. 

Let $\check{\bfs\beta}$ be a (signed) measure on $\R_r$ with total mass $-2b$ and set 
$\bfs{\beta}:=2b \cdot q+\check{\bfs\beta}$ where $q \in \R$.
We write $\bfs\beta_n$ for its discrete counterpart given by 
$$
\bfs\beta_n:=2b \cdot q + \sum_{k=1}^n \frac{\beta_k}{n} \cdot q_k, \qquad (\sum \beta_k=-2b ).
$$
Set $\bfs \tau= -a \cdot q$ and define the one-leg operator $\Psi_{ \bfs\beta }$ as the continuum limit of its discrete counterpart, i.e., 
$$
\Psi_{ \bfs\beta }[a\cdot p-a \cdot q]:=\lim_{ n \to \infty } \Psi_{ \bfs\beta_n }[a\cdot p-a \cdot q].
$$ 
Let us also write 
$$
Z_{ \bfs \beta }:= \E \Psi_{ \bfs\beta }[a\cdot p-a \cdot q], \qquad Z_{ \bfs \beta_n }:= \E \Psi_{ \bfs\beta_n }[a\cdot p-a \cdot q].
$$
Denoting $q_k=s_k+ir$, by \eqref{Z beta n} below, we have 
\begin{align}
	\begin{split} \label{Z msr beta}
		Z_{ \bfs \beta }
		&=\Theta'(0)^{\frac12} \Theta(p-q)^{1-\frac{6}{\kappa} } \exp\Big( a \int \log \frac{\Theta_I(p-s)}{\Theta_I(q-s)} \,d \check{\bfs\beta}(s) \Big) 
		\\
		&\times \exp\Big[ \Big( 1-\frac{6}{k} \Big) \frac{(p-q)^2}{4(r+\chi)}+ a \frac{p-q}{4(r+\chi)} \int (p+q-2s)\, \,d \check{\bfs\beta}(s) \Big] 
	\end{split}
\end{align}
in the identity chart of $\CC_r$. In the sequel, without loss of generality, let $p+q=0$. 

We now derive the null-vector equation for $Z_{\bfs \beta}$. 

\begin{prop}\label{Prop_NVEnonatom}
The partition function $Z_{\bfs \beta}$ satisfies
\begin{equation}
	\pa_r Z_{ \bfs \beta }=\frac{\kappa}{2} \pa_p^2 Z_{ \bfs \beta }+H(p-q) \pa_p Z_{ \bfs \beta } +h_{1,2} H'(p-q) Z_{ \bfs \beta } +F_{ \bfs\beta } Z_{ \bfs \beta },
\end{equation}
where 
\begin{equation}
	F_{ \bfs\beta }
		= -2\Big( \E T_{ \bfs\beta }(p)+h_{1,2} \frac{\zeta_r(\pi) }{\pi} \Big) 
	- \frac{a}{2} \int \Big( H(p-q)- H_I(p-s)\Big) \Big( H_I(p-s)-H_I(q-s)+\frac{p-q}{r+\chi} \Big) \, d\check{\bfs\beta}(s) .
\end{equation}
Here 
	\begin{align*}
	\begin{split}
		\E T_{ \bfs{\beta} }(p)&= \dfrac{\zeta_r(\pi)}{2\pi}-\frac{1}{4(r+\chi)}
		+b^2\Big[ H'(p-q)+\frac12 \Big(H(p-q)+\frac{p-q}{r+\chi}\Big)^2 \Big]
		\\
		&+\dfrac{1}{8}\Big[ \int \Big(H_I(p-s)+\frac{p-s}{r+\chi}\Big) \, d\check{\bfs\beta}(s) \Big]^2+\dfrac{b}{2} \, \int H'_I(p-s) \, d\check{\bfs\beta}(s). 
	\end{split}
\end{align*}
\end{prop}
\begin{proof}
	By \eqref{Partition gen}, we have (up to a multiplicative constant)
	\begin{equation} \label{Z beta n}
	Z_{ \bfs \beta_n }	=\Theta(p-q)^{1-\frac{6}{\kappa} } \prod_{k=1}^n \Big( \frac{ \Theta_I(p-s_k)}{ \Theta_I(q-s_k)} \Big)^{ \frac{a\beta_k}{n} }	 \exp\Big[ \Big( 1-\frac{6}{k} \Big) \frac{(p-q)^2}{4(r+\chi)}+ a \sum_{ k=1 }^n \frac{\beta_k}{n} \frac{(p-q)(p+q-2s_k)}{4(r+\chi)} \Big]. 
	\end{equation}
By Theorem~\ref{BPZ_ER}, the partition function $Z_{ \bfs \beta_n }$ satisfies the following form of the null-vector equation
	\begin{align*}
		\pa_r Z_{ \bfs \beta_n }&=\frac{\kappa}{2} \pa_p^2 Z_{ \bfs \beta_n }-H(p-q) \pa_q Z_{ \bfs \beta_n } +h_{1,2} H'(p-q) Z_{ \bfs \beta_n } 
		\\
		&-2\Big( \E T_{ \bfs\beta_n }(p)+h_{1,2} \frac{\zeta_r(\pi) }{\pi} \Big) Z_{ \bfs \beta_n }- \sum H_I(p-s_k) \pa_{s_k}Z_{ \bfs \beta_n }+ i( \pa_{ q_k }-\bp_{ q_k } ) Z_{ \bfs \beta_n }. 
	\end{align*}
	It follows from the translation invariance
	$( \pa_p +\pa_q+\sum \pa_{s_k} ) Z_{ \bfs \beta_n }=0$
	that this equation is rewritten as 
	\begin{align*}
		\pa_r Z_{ \bfs \beta_n }&=\frac{\kappa}{2} \pa_p^2 Z_{ \bfs \beta_n }+H(p-q) \pa_p Z_{ \bfs \beta_n } +h_{1,2} H'(p-q) Z_{ \bfs \beta_n } 
		\\
		&-2\Big( \E T_{ \bfs\beta_n }(p)+h_{1,2} \frac{\zeta_r(\pi) }{\pi} \Big) Z_{ \bfs \beta_n }
		+ \sum \Big( H(p-q)- H_I(p-s_k)\Big) \pa_{s_k} Z_{ \bfs \beta_n }+i( \pa_{ q_k }-\bp_{ q_k } ) Z_{ \bfs \beta_n }.
	\end{align*}
	
	Note that by \eqref{Z beta n} we have 
	$$
	\frac{\pa_{ q_k } Z_{ \bfs \beta_n }} {Z_{ \bfs \beta_n }} =-\frac{a}{2} \frac{\beta_k }{n} \Big( H_I(p-s_k)-H_I(q-s_k)+\frac{p-q}{r+\chi} \Big).
	$$
	This gives rise to
	$$
	\pa_r Z_{ \bfs \beta_n }=\frac{\kappa}{2} \pa_p^2 Z_{ \bfs \beta_n }+H(p-q) \pa_p Z_{ \bfs \beta_n } +h_{1,2} H'(p-q) Z_{ \bfs \beta_n } +F_{ \bfs\beta_n } Z_{ \bfs \beta_n },
	$$
	where 
	\begin{equation}
	 	F_{ \bfs\beta_n }
		= -2\Big( \E T_{ \bfs\beta_n }(p)+h_{1,2} \frac{\zeta_r(\pi) }{\pi} \Big) 
		- \frac{a}{2} \sum \frac{\beta_k}{n} \Big( H(p-q)- H_I(p-s_k)\Big) \Big( H_I(p-s_k)-H_I(q-s_k)+\frac{p-q}{r+\chi} \Big).
	\end{equation}
	Furthermore, by Proposition~\ref{AT beta}, we have
	\begin{align*}
		\begin{split}
			\E T_{ \bfs{\beta}_n }(p)&= \dfrac{\zeta_r(\pi)}{2\pi}-\frac{1}{4(r+\chi)}
			+b^2\Big[ H'(p-q)+\frac12 \Big(H(p-q)+\frac{p-q}{r+\chi}\Big)^2 \Big]
			\\
			&+\dfrac{1}{8}\Big[ \sum \frac{\beta_k}{n} \Big(H_I(p-s_k)+\frac{p-s_k}{r+\chi}\Big) \Big]^2+\dfrac{b}{2} \, \sum \frac{\beta_k}{n} H'_I(p-s_k). 
		\end{split}
	\end{align*}
Now proposition follows by letting $n \to \infty$. 
\end{proof}

\begin{eg*}
	Let $\check{\bfs\beta}$ be the uniform (signed) measure on $[-\pi,\pi]$ with total mass $-2b$. In this case, by \eqref{Z msr beta} and the periodicity \eqref{theta I per} of $\Theta_I(r,\cdot)$, we have 
	\begin{equation}
		Z_{ \bfs \beta }=\Theta'(0)^{\frac12} \Theta(p-q)^{1-\frac{6}{\kappa} } \exp\Big[ \Big( 1-\frac{6}{k} \Big) \frac{(p-q)^2}{4(r+\chi)} \Big]. 
	\end{equation} 
Note also that we have 
	\begin{align*}
		\begin{split}
			\E T_{ \bfs{\beta} }(p)&= \dfrac{\zeta_r(\pi)}{2\pi}-\frac{1}{4(r+\chi)}
			+b^2 \Big( H'+\frac{H^2}{2} \Big)(p-q) 
		+b^2\Big( \frac{p-q}{r+\chi} H(p-q) +\frac58 \frac{(p-q)^2}{(r+\chi)^2} \Big).
		\end{split}
	\end{align*}
	On the other hand, it follows from \eqref{add H} that 
	\begin{align*}
		&\quad \frac{1}{2\pi}\int_{-\pi}^\pi \Big( H(p-q)- H_I(p-s)\Big) \Big( H_I(p-s)-H_I(q-s)+\frac{p-q}{r+\chi} \Big)\,ds
		\\
		&=\frac{p-q}{r+\chi}\,H(p-q)-\frac{1}{2\pi}\int_{-\pi}^\pi H_I(p-s) \Big( H_I(p-s)-H_I(q-s) \Big)\,ds
		\\
		&=\frac{p-q}{r+\chi}\,H(p-q)+\Big(H'+\frac{ H^2}{2}\Big)(p-q)+\frac{6}{\pi}\zeta_r(\pi).
	\end{align*}
	Therefore we obtain (up to an additive constant)
	\begin{align*}
		F_{\bfs \beta}=\Big( 2-\frac{\kappa}{2} \Big)\Big( 1-\frac{6}{\kappa} \Big) \Big[ \Big( \frac{H'}{2}+\frac{H^2}{4} \Big)(p-q) +\frac{p-q}{2(r+\chi)}\,H(p-q)+ \frac{(p-q)^2}{4(r+\chi)^2} \Big]. 
	\end{align*}
\end{eg*}

\ms
\noindent 
\begin{eg*}[GFF with general boundary data]
Here we focus on the case when $b=\chi=0$. Namely, we consider the Dirichlet boundary condition and $\bfs{\beta}$ is a measure defined on the upper boundary component of $\CC_r$ satisfying $\int \bfs{\beta}=0$. Then the associated bosonic field $\E \wh{\Phi}_{ \bfs\beta }$ satisfies the boundary condition
	\begin{equation} \label{Bosonic_bc gen}
	\E \wh{\Phi}_{ \bfs\beta }(z)= 
	\begin{cases}
		\hspace{4.5em}-\lambda &\text{if}\quad z \in (0,p)\cup (q,2\pi),
		\\
		\hspace{4.5em}+\lambda &\text{if}\quad z \in (p,q),
		\\
		\displaystyle \Big(1-\frac{2}{a} \int_{0}^{\Re\,z} \,d\bfs{\beta}(s) \Big)\lambda &\text{if}\quad z \in \R_r,
	\end{cases}
\end{equation}
where $\lambda=a \pi =\pi/\sqrt{2}.$ 
Notice that the neutrality condition $\int \bfs{\beta}=0$ is equivalent to the fact that the harmonic function $\E \wh{\Phi}_{ \bfs\beta }$ is well defined in the cylinder. 
To see \eqref{Bosonic_bc gen}, note that if $\bfs\beta$ is a linear combination of atomic measure of the form
$\bfs{\beta}=\sum_{k=1}^n \frac{\beta_k}{n} \cdot \delta_{q_k}$, the correlation function of the bosonic observable has piecewise Dirichlet boundary condition having additional jump of $2\pi a$ at $p$, $-2\pi a$ at $q$ and $2\pi \beta_k/n$ at $q_k$'s. 

For a general background charge $\bfs{\beta}$, we have (up to a multiplicative constant)
\begin{align} 
	\begin{split}
		Z_{ \bfs \beta }(p,q)&=\Theta(p-q)^{-\frac12} \exp\Big( -\frac{(p-q+\sqrt{2} \int s \, d\bfs{\beta}(s) )^2}{8r} \Big) \exp \Big( \frac{1}{\sqrt{2}} \int \log \Big( \frac{\Theta_I(p-s)}{\Theta_I(q-s)} \Big) \, d\bfs{\beta} (s) \Big).
	\end{split}
\end{align}
By Proposition~\ref{Prop_NVEnonatom}, the null-vector equation for $Z_{ \bfs \beta }$ is given by 
	\begin{equation}
	\pa_r Z_{ \bfs \beta }=2 \,\pa_p^2 Z_{ \bfs \beta }+H(p-q) \pa_p Z_{ \bfs \beta } +\frac14 H'(p-q) Z_{ \bfs \beta } +F_{ \bfs\beta } Z_{ \bfs \beta },
\end{equation}
where 
\begin{align*}
	\begin{split}
		F_{ \bfs\beta }
		&= -\frac{3}{2}\frac{\zeta_r(\pi)}{\pi}+\frac{1}{2r}-\frac14 \Big[ \int \Big(H_I(p-s)+\frac{p-s}{r}\Big) \, d\bfs\beta(s) \Big]^2
		\\
		&+ \frac{1}{2\sqrt{2}} \int H_I(p-s) \Big( H_I(p-s)-H_I(q-s)+\frac{p-q}{r} \Big) \, d\bfs\beta(s) .
	\end{split}
\end{align*}
Moreover, by \eqref{add H}, one can also write $	F_{ \bfs\beta }$ as 
\begin{align*}
	\begin{split}
		F_{ \bfs\beta }
		&= -\frac{3}{2}\frac{\zeta_r(\pi)}{\pi}+\frac{1}{2r}-\frac14 \Big[ \int \Big(H_I(p-s)+\frac{p-s}{r}\Big) \, d\bfs\beta(s) \Big]^2
		\\
		&- \frac{1}{2\sqrt{2}} \int \Big( H_I'-\frac{H_I^2}{2} \Big)(p-s) +\Big( H_I'+\frac{H_I^2}{2} \Big)(q-s) \, d\bfs\beta(s)
	+\frac{1}{2\sqrt{2}} \frac{p-q}{r} \int H_I(p-s)\,d\bfs{\beta}(s).
	\end{split}
\end{align*}
\end{eg*}
\begin{rmk*}
	Note that if $\bfs{\beta}=\beta \cdot q_1-\beta \cdot q_2$ for some $q_1, q_2 \in \R_r$ satisfying $q_2=q_1+2\pi$, the boundary condition \eqref{Bosonic_bc gen} reduces to \eqref{non-zero Diri b.c.} below.
	Also it follows from the periodicity of $H(r,\cdot)$ that
	$$
	Z(p,q)=\Theta'(0)^{ \frac12 } \Theta(r,x)^{-\frac12}\exp\Big( -\frac{ x^2-2\mu x }{ 8r } \Big), \qquad \mu=2\sqrt{2}\beta \pi
	$$
	and 
	$$
	F_{ \bfs{\beta} }(p,q)=-\frac{3}{2}\frac{\zeta_r(\pi)}{\pi}+\frac{1}{2r}-\Big( \frac{\beta \pi}{r} \Big)^2.
	$$
\end{rmk*}

\subsection{Periodization}

In this section we present the implementation of weighted summation of the chiral conformal fields. 
We refer to \cite{MR896667} for a similar idea (under the name of \emph{chiral bosonization}) on compact Riemann surfaces. 
This allows us to construct some important examples of SLE martingale-observables in the next section. 

For a suitable weight function $\omega: \mathbb{Z}\to \R_{+}$, let us define
$$
\Psi_{\bfs\beta}^\omega(p,\bfs{\xi}):=\sum_{ n \in \mathbb{Z} } \omega(n) \Psi_{\bfs\beta}(p+2n\pi,\bfs{\xi})
$$ 
and write 
$$
Z_{\bfs\beta}^\omega(p,\bfs{\xi}):=\E \Psi_{\bfs\beta}^\omega(p,\bfs{\xi}), \qquad \Lambda^\omega(r,p,\bfs{\xi}):=\kappa \, \pa_p \log Z_{\bfs\beta}^\omega(p,\bfs{\xi}). 
$$
We remark that in \cite{MR3334276}, Zhan used the weight function 
$$\omega(n):=e^{-\frac{2\pi s_0}{\kappa} n}, \qquad (s_0 \in \R)
$$ 
to construct partition functions having rotational periodicity.

\begin{eg*}
Let us consider the case that all $\xi_j$'s and $q_k$'s are on $\R_r$. Then up to a multiplicative constant, 
\begin{align*} 
	\begin{split}
		Z_{ \bfs{\beta} }(p,\bfs{\xi})&=	 \exp\Big( -\frac{ (ap+\sum_j \tau_j \,\Re\, \xi_j +\sum_k \beta_k \, \Re\, q_k )^2 }{4(r+\chi)} \Big) 
		\\
		&\times \prod_j \Theta_I(r,p-\Re \, \xi_j) ^{a \tau_j} \prod_j \Theta_I(r,p-\Re\, q_j) ^{a \beta_j} 
	 \prod_{j < k} \Theta(r,\xi_j-\xi_k)^{\tau_j \tau_k} \prod_{j, k} \Theta(r,\xi_j-q_k)^{\tau_j \beta_k}.
	\end{split}
\end{align*}
Note that for $\chi < \infty$, the partition function $Z_{ \bfs{\beta} }(p,\bfs{\xi})$ is not $2\pi$-periodic with respect to the space variable $p$. 
Set $\omega(n)\equiv 1$. By \eqref{Theta_I Gauss}, we have
\begin{align*} 
	\begin{split}
		Z_{ \bfs{\beta} }^\omega(p,\bfs{\xi})&=	 \Theta_I\Big(\frac{\kappa}{2}(r+\chi), p+\sum_j \frac{\tau_j}{a} \,\Re\, \xi_j +\sum_k \frac{\beta_k}{a} \, \Re\, q_k +\pi \Big)
		\\
		&\times \prod_j \Theta_I(r,p-\Re \, \xi_j) ^{a \tau_j} \prod_k \Theta_I(r,p-\Re\, q_k) ^{a \beta_k} 
	 \prod_{j < k} \Theta(r,\xi_j-\xi_k)^{\tau_j \tau_k} \prod_{j, k} \Theta(r,\xi_j-q_k)^{\tau_j \beta_k}
	\end{split}
\end{align*}
and 
\begin{align} 
	\begin{split}
		\Lambda_{ \bfs{\beta} }^\omega(p,\bfs{\xi})&=	 \frac{\kappa}{2}H_I\Big(\frac{\kappa}{2}(r+\chi), p+\sum_j \frac{\tau_j}{a} \,\Re\, \xi_j +\sum_k \frac{\beta_k}{a} \, \Re\, q_k +\pi \Big)
		\\
		&+\sqrt{\frac{\kappa}{2}}\Big[ \sum_j \tau_j H_I(r,p-\Re \, \xi_j)+ \sum_k \beta_k H_I(r,p-\Re\, q_k) \Big]. 
	\end{split}
\end{align}
Then it follows from the periodicity of $H_I$ that $	\Lambda_{ \bfs{\beta} }^\omega$ is well-defined in the cylinder $\CC_r$.
\end{eg*}

We now prove Theorem~\ref{main_period}. 

\begin{proof}[Proof of Theorem~\ref{main_period}] 
Let $\XX$ be a string of fields in the OPE family $\FF_{\bfs\beta}$ and set 
$$
R_\xi=\wh{\E} \XX:= \frac{\E \Psi^\omega_{ \bfs{\beta} }(\xi , \bfs{\xi})\XX }{ \E \Psi^\omega_{ \bfs{\beta} }(\xi , \bfs{\xi}) }.
$$ 
By the argument presented in the proof of Theorem~\ref{main}, it suffices to show the following form of BPZ-Cardy equation
\begin{equation} \label{BPZ_C_period}
	\frac{1	}{a^2}\pa_{\xi}^2R_\xi+ 	\Lambda_{ \bfs{\beta} }^\omega \pa_{\xi}R_\xi
	= \LL_{v_{\xi}} R_\xi +\pa_r R_\xi.
\end{equation}

By Theorem~\ref{BPZ_ER}, we have
	\begin{align*}
	\frac{1}{a^2}\pa^2_{\xi}\E \Psi_{\bfs{\beta}}(\xi , \bfs{\xi}) \XX &= \Big( \LL_{v_{\xi}}+\pa_r+2h_{1,2} \frac{\zeta_r(\pi)}{\pi}+2\E T_{\bfs\beta}(\xi) \Big) \E \Psi_{\bfs{\beta}}(\xi , \bfs{\xi}) \XX.
\end{align*}
Therefore we obtain 
	\begin{align*}
		\begin{split}
	\frac{1}{a^2}\pa^2_{\xi}\E \Psi_{\bfs{\beta}}(\xi , \bfs{\xi}) \XX &= \Big( \LL_{v_{\xi}}+\pa_r+2h_{1,2} \frac{\zeta_r(\pi)}{\pi} \Big) \E \Psi_{\bfs{\beta}}^\omega(\xi , \bfs{\xi}) \XX
	+2\sum_{n \in \mathbb Z} \E T_{\bfs\beta}(\xi+2n\pi) \E \Psi_{\bfs{\beta}}(\xi+2n\pi , \bfs{\xi}) \XX. 
			\end{split}
\end{align*}
By Proposition~\ref{AT beta} and the neutrality condition $\int \bfs \beta=0$,
	\begin{align}
	\begin{split}
		\E T_{ \bfs{\beta} }(\xi)
		&=\dfrac{1}{32}\Big[ \sum \beta_k \Big\{ H(r,q_k-\xi)+H(r,\bar{q}_k-\xi)+2 \frac{\Re\,q_k}{r+\chi} \Big\} \Big]^2
		\\
		&+\dfrac{b}{4} \, \sum \beta_k \Big\{ H'(r,q_k-\xi)+H'(r,\bar{q}_k-\xi) \Big\}
	+ \dfrac{\zeta_r(\pi)}{2\pi}-\frac{1}{4(r+\chi)}. 
	\end{split}
\end{align}
Thus we have for any $n \in \mathbb{Z}$,
\begin{equation}
	\E T_{ \bfs{\beta} }(\xi+2n\pi)=	\E T_{ \bfs{\beta} }(\xi),
\end{equation}
which leads to 
	\begin{align}
	\frac{1}{a^2}\pa^2_{\xi}R_\xi Z^\omega_{\bfs{\beta}}(\xi , \bfs{\xi}) &= \Big( \LL_{v_{\xi}}+\pa_r+2h_{1,2} \frac{\zeta_r(\pi)}{\pi}+2\E T_{\bfs\beta}(\xi) \Big) R_\xi Z^\omega_{\bfs{\beta}}(\xi , \bfs{\xi}).
\end{align}
By applying the trivial string $\mathcal X \equiv 1$ to the above, we have the null-vector equation 
\begin{align}
	\frac{1}{a^2}\pa^2_{\xi} Z^\omega_{\bfs{\beta}}(\xi , \bfs{\xi})&=\Big( \LL_{v_{\xi}}+\pa_r+2h_{1,2} \frac{\zeta_r(\pi)}{\pi}+2\E T_{\bfs\beta}(\xi) \Big) Z^\omega_{\bfs{\beta}}(\xi , \bfs{\xi}).
\end{align}
Subtracting the above two equations, we obtain \eqref{BPZ_C_period}, which completes the proof.

\end{proof}

\subsection{Screening} \label{Screening_sub}

From now on, we focus on an SLE$(\kappa,\Lambda)$ starting from $p$ with one force point $q$. 
In this subsection we discuss the screening method to construct family of solutions to the null-vector equation
\begin{equation} 	\label{Lawler-Zhan PDE}
	\displaystyle \pa_r Z = \frac{\kappa}{2} Z'' + H Z' + \Big( \frac{3}{\kappa}-\frac{1}{2} \Big)H' Z +C(r) Z
\end{equation}
for the chordal type annulus partition function and prove Theorem~\ref{main_SLE Z}. The screening is a well-known method in CFT to find solutions of the Knizhnik-Zamolodchikov type equations, see e.g., \cite{MR2016311,MR1629472}.

In general, it has been turned out that the screening method is useful particularly in the theory of multiple SLEs. For instance, in \cite{MR2253875}, Dub\'{e}dat found Euler integral representations for solutions to a system of PDEs characterizing commuting SLEs in $\H$ by means of screening. We refer the reader to \cite{kytola2016pure, MR3922531, MR4019203,MR3296159} and references therein for recent studies on solutions to such PDE system and the geometric properties of associated multiple SLEs. 
 
For $q_1,q_2 \in \C$, we set 
\begin{equation} \label{tau beta}
\bfs\beta= \beta \cdot q_1- \beta \cdot q_2.
\end{equation}
Here we assume that $q_1, q_2$ are two different fibers of a marked point in annulus satisfying $q_2=q_1 + 2\pi$. 
For $p,q \in \R$ and $\zeta \in \C$, write 
\begin{equation}
\Psi_\beta^\sharp(p,q,\zeta):= \OO_{\bfs\beta}[ a \cdot p+a \cdot q-2a \cdot \zeta,\,\bfs{0}], \qquad Z_\beta^{\sharp}:=\E \Psi_\beta^{\sharp}. 
\end{equation} 
Then the conformal dimension of $\Psi^{\sharp}$ at $p,q$ and $\zeta$ are given by 
\begin{equation}\label{pre screen dim}
\lambda_p=\lambda_q=\frac{6-\kappa}{2\kappa}, \qquad \lambda_\zeta=1.
\end{equation}
By \eqref{Partition gen}, we have 
\begin{align*}
\begin{split}
Z^\sharp_{\beta }(p,q,\zeta)&=\Theta'(0)^{\frac{6}{\kappa}} \Theta_\chi(p-q)^{\frac{2}{\kappa}} \Theta_\chi(p-\zeta)^{-\frac{4}{\kappa} } \Theta_\chi(q-\zeta)^{-\frac{4}{\kappa} } 
 \Big| \frac{ \Theta_\chi(p-z_0) }{ \Theta_\chi(p-z_1) } \frac{ \Theta_\chi(q-z_0) }{ \Theta_\chi(q-z_1) } \frac{ \Theta_\chi(\zeta-z_1)^2 }{ \Theta_\chi(\zeta-z_0)^2 } \Big|^{a\beta}.
\end{split}
\end{align*}
In particular since $q_2-q_1 = 2\pi$ and \eqref{theta per}, this expression reduces to 
(up to a multiplicative constant)
\begin{align}\label{pre Z}
\begin{split}
Z^\sharp_{\beta }(p,q,\zeta)&=\Theta'(0)^{\frac{6}{\kappa}} \exp\Big( \frac{\beta^2 \pi^2}{r+\chi} \Big)\Theta(p-q)^{\frac{2}{\kappa}}
\Theta(p-\zeta)^{-\frac{4}{\kappa} } \Theta(\zeta-q)^{-\frac{4}{\kappa} } \exp\Big( - \frac{ (\Sigma-2\beta \pi )^2 }{ 4(r+\chi) } \Big),
\end{split}
\end{align}
where $\Sigma=a ( p+q-2\zeta )$. 
For generic $\kappa>0$, we need to specify a principal branch (as a function of $\zeta$) in the expression \eqref{pre Z}. In what follows, we use the branch cut $[-\infty,q] \cup [p,\infty]$.

For a closed contour $\gamma$, set
\begin{equation} \label{Psi beta}
	\Psi_\beta \equiv \Psi_\beta(p,q):= C(\kappa) \oint_{\gamma} \Psi_\beta^\sharp(p,q,\zeta)\, d\zeta,
\end{equation}
where the normalization constant $C(\kappa)$ is given by
\begin{equation} \label{scr normal const}
	C(\kappa)=\sin^{-2} \Big(\frac{4\pi}{\kappa}\Big) \, \frac{1}{4\,\Gamma(1-4/\kappa)}.
\end{equation} 
Here $\Gamma$ is the Gamma function. It is easy to see that $C(\kappa)$ has a simple pole at $\kappa$ if and only if $4/\kappa$ is a positive integer. By \eqref{pre Z}, the partition function 
\begin{equation} \label{scr Z}
Z_\beta \equiv Z_\beta(r,p-q):=\E \Psi_\beta = C(\kappa) \oint_{\gamma} Z_\beta^\sharp(p,q,\zeta)\,d\zeta
\end{equation}
is evaluated as (up to a multiplicative constant)
\begin{align}\label{Z_beta} 
\begin{split}
Z_\beta&= \Theta(p-q)^{\frac{2}{\kappa}} \oint_\gamma \Theta(p-\zeta)^{-\frac{4}{\kappa} } \Theta(q-\zeta)^{-\frac{4}{\kappa} } \exp\Big( - \frac{ (\Sigma-2\beta \pi )^2 }{ 4(r+\chi) } \Big) \, d\zeta.
\end{split}
\end{align}
\begin{rmk*}
	In general, partition functions $Z_\beta(r,\cdot)$ do not have periodic property with respect to $2\pi$ except the case $\chi=\infty$ (i.e., ER b.c.). On the other hand, for $\chi=0$ (i.e., Dirichlet b.c.) $Z_\beta$ has following rotational periodicity with respect to $2ir$. 
	\begin{equation} \label{2ir rot period}
	Z_\beta(r,x+2ir)= \exp\Big( a (2\beta-a) \pi i \Big) Z_\beta(r,x).
	\end{equation}
\end{rmk*}

We denote by $\Pi_1 \equiv \Pi_1(\C \sm \{p,q \}, \zeta_0)$ the fundamental group of $\C \sm \{ p,q \}$ with base point $\zeta_0$. 
If $4/\kappa$ is not a positive integer, in order to make the partition function $Z_\beta$ single-valued, the homotopy class $[\gamma]$ should reside in the commutator subgroup of the $\Pi_1$. Moreover, this specific choice \eqref{scr normal const} of $C(\kappa)$ makes the expression \eqref{scr Z} non-trivial for all values of $\kappa >0$ after removing possible singularities. We will explain this in Proposition~\ref{screen null} below. 

The simplest example of such $\gamma$ is the Pochhammer contour $\mathcal{P}(p,q)$ entwining $p$ and $q$.
See Figure~\ref{Pochh} for the description of $\mathcal{P}(p,q)$, where we indicate the branch cut as a dotted line. 
It is an elementary but remarkable property that the winding number of $\mathcal{P}(p,q)$ with respect to each point $p,q$ is zero. This property allows us to define $Z^{\sharp}(p,q,\cdot)$ as a single-valued analytic function on the contour. 

Indeed, the Pochhammer contour is the typical example of path contained in the class of \emph{first twisted homology} or \emph{loaded cycle}, see \cite{yoshida2013hypergeometric,MR2016311}. In other words, for any (single-valued) function $G(\zeta)$ defined on $\mathcal{P}(p,q)$, we have
\begin{equation} \label{boundary Poch}
	\oint_{\mathcal{P}(p,q)} \pa_\zeta G(\zeta) \, d\zeta =0 .
\end{equation}

We now present the following version of null-vector equations, which immediately leads to the first assertion of Theorem~\ref{main_SLE Z}.
\begin{prop} \label{screen null}
	For any $\kappa>0$, $Z_\beta(r,\cdot)$ is a (non-trivial) real-valued solution to 
	\begin{equation} 	\label{Z PDE_cho}
		\displaystyle \pa_r Z = \frac{\kappa}{2} Z'' + H Z' + \Big( \frac{3}{\kappa}-\frac{1}{2} \Big)H' Z +C(r) Z,
	\end{equation}
	where $C(r)$ is given as 
		\begin{align} \label{null vec const1}
	\begin{split}
	C(r)&=-\frac{6}{\kappa} \frac{\zeta_r(\pi)}{\pi}+\frac{1}{2(r+\chi)}-\Big( \frac{\beta \pi}{r+\chi} \Big)^2. 
	\end{split}
	\end{align}

\end{prop}
\begin{proof}
	Recall that $Z_\beta^{\sharp}(p,q,\zeta)=\E \Psi_\beta^{\sharp}(p,q,\zeta)$. Combining Corollary~\ref{pre screen null} with \eqref{pre screen dim}, 
	\begin{align*}
		\pa_r Z_\beta^\sharp&= \frac{\kappa}{2} \pa^2 Z_\beta^\sharp+\Big[\Big(\frac{3}{\kappa}-\frac{1}{2}\Big) H'+H'(p-\zeta) \Big]Z^\sharp_\beta
		-\Big( H\pa_q Z^\sharp_{\beta }+H(p-\zeta) \pa_{\zeta}Z^\sharp_\beta \Big)+C(r)Z^\sharp_\beta.
	\end{align*}
	Here we simplify the term $C(r,\bfs{q})$ in \eqref{null vec const} by the ($2\pi$-) periodicity of $H(r,\cdot)$ and the fact $q_2-q_1=2\pi$. 
	
	Observe from the evaluation \eqref{pre Z} that
	$(\pa+\pa_{q}+\pa_{\zeta}) Z^\sharp(p,q,\zeta)=0$. 
	Therefore the above equation is rewritten as 
	\begin{align}\label{pre screen null vec}
		\begin{split}
			\pa_r Z^\sharp_\beta &= \frac{\kappa}{2} \pa^2 Z^\sharp_\beta + H \pa Z^\sharp_\beta + \Big( \frac{3}{\kappa}-\frac{1}{2} \Big)H' Z^\sharp_\beta+C(r) Z^\sharp_\beta+\pa_{\zeta}F(\zeta),
		\end{split}
	\end{align}
	where 
	$$
	F(\zeta):= H(p-q) \,Z^\sharp_\beta- H(p-\zeta)Z^\sharp_\beta. 
	$$
	By integrating \eqref{pre screen null vec} with respect to screening variable $\zeta$ along $\mathcal{P}(p,q)$, the desired null-vector equation \eqref{Z PDE_cho} follows from \eqref{boundary Poch}.
	
	It remains to verify the real-valuedness and non-triviality of $Z_\beta(r,\cdot)$. For $\kappa >4$, the expression \eqref{scr Z} can be decomposed as
	\begin{align} \label{Pochh int decom}
		Z_\beta&= C(\kappa)\sum_{j=1}^{4} \int_{\gamma_j} Z_\beta^\sharp(p,q,\zeta)\,d\zeta,
	\end{align}
	where the path $\gamma_j$'s are given as in Figure~\ref{Decomp Pochh} below. 
	\begin{figure}[h]
		\begin{center}
			\includegraphics[width=5.0in]{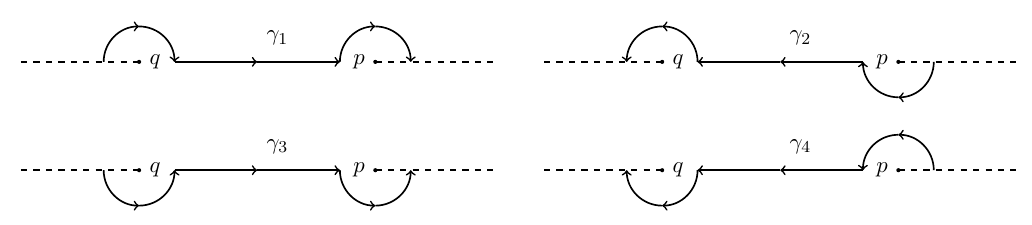}
		\end{center}
		\caption{A contour deformation}\label{Decomp Pochh}
	\end{figure}
	
	We denote by $f(\cdot)$ the principal branch of $Z^\sharp_\beta(p,q,\cdot)$ defined on $\C \setminus ( [-\infty,q] \cup [p,\infty] )$. Then one can express $Z_\beta^\sharp(p,q,\cdot)$ on $\gamma_j$ as
	$$
	Z^\sharp_\beta(p,q,\zeta)=
	\begin{cases}
	f(\zeta) & \textrm{if}\quad \zeta \in \gamma_1,
	\smallskip 
	\\
	f(\zeta)\, e^{-2\pi i \frac{4}{\kappa}\pi} & \textrm{if}\quad \zeta \in \gamma_2,
	\smallskip 
	\\
	f(\zeta) & \textrm{if}\quad \zeta \in \gamma_3,
\smallskip 
	\\
	f(\zeta)\, e^{2\pi i \frac{4}{\kappa}\pi} & \textrm{if}\quad \zeta \in \gamma_4.
	\end{cases}
	$$
	Therefore the Pochhammer contour integration \eqref{Pochh int decom} coincides with the following Euler type integral 
	\begin{equation} \label{Euler integral}
		Z_\beta(r,p-q)=\frac{1}{\Gamma(1-4/\kappa)} \int_{q}^{p} Z_\beta^\sharp(p,q,t) \, dt. 
	\end{equation}
	The integrability condition is satisfied both at $p$ and $q$ if and only if $\kappa >4$. Note that the integrand $Z_\beta^\sharp(p,q,\cdot)$ is a positive function since $\Theta_\chi(r,\cdot)$ is a non-negative function in the interval $[0,2\pi]$ vanishing only at $0,2\pi$. Therefore the positivity of $Z_\beta(r,\cdot)$ for $\kappa>4$ also follows. 
	
	We remark that as a function of $\kappa$, $Z_\beta$ can be interpreted as an analytic continuation of the right-hand side of \eqref{Euler integral}, which implies the real-valuedness of $Z_\beta$ for $\kappa \le 4$. 
	Now it remains to check the non-triviality of $Z_\beta$.
	By \eqref{degen gen}, we have (up to a multiplicative constant)
	\begin{align*}
		\begin{split}
			Z_\infty(p-q)&:=\lim\limits_{r\to \infty} Z_\beta(r,p-q)
		=C(\kappa) \sin^{\frac{2}{\kappa}}\Big( \frac{p-q}{2} \Big) \oint_{\mathcal{P}(p,q)} \sin^{-\frac{4}{\kappa}} \Big( \frac{p-\zeta}{2} \Big) \sin^{-\frac{4}{\kappa}}\Big( \frac{\zeta-q}{2} \Big)\, d\zeta.
		\end{split}
	\end{align*}
	Therefore for $\kappa >4$, 
	\begin{align*}
		Z_\infty(p-q)&=\frac{1}{\Gamma(1-4/\kappa)} \sin^{\frac{2}{\kappa}}\Big( \frac{p-q}{2} \Big) \int_{q}^{p} \sin^{-\frac{4}{\kappa}}\Big( \frac{p-t}{2} \Big) \sin^{-\frac{4}{\kappa}}\Big( \frac{t-q}{2} \Big) \, dt.
	\end{align*}
	Recall that the regularized hypergeometric function
	$$
	{}_2\textbf{F}_1(a,b;c;z):=\frac{ {}_2F_1(a,b;c;z) }{\Gamma(c)}
	$$
	is an entire function of $a,b,c$. 
	Then by \cite[Eq.(15.6.1)]{olver2010nist}, one can express $Z_\infty$ in terms of $	{}_2\textbf{F}_1$ as follows: 
	\begin{equation} \label{Z deg}
		Z_\infty(x)=\cos^{\frac{2}{\kappa}-1} \Big( \frac{x}{4} \Big) \sin^{1-\frac{6}{\kappa}} \Big( \frac{x}{4} \Big) {}_2\textbf{F}_1\Big(\frac{1}{2}, 1-\frac{4}{\kappa}; \frac{3}{2}-\frac{4}{\kappa}, -\tan^2\Big(\frac{x}{4}\Big) \Big).
	\end{equation}
	Therefore by the identity theorem, \eqref{Z deg} holds not only for $\kappa >4$ but also for $\kappa\le4$. 
	Thus we conclude that $Z_\beta(r,\cdot)$ is a non-trivial function for all $\kappa >0$. 
\end{proof}

\begin{rmk*}
	We emphasize that not only the partition functions $Z_\beta(r,x)$, but also their derivatives with respect to the ``variables'' $\beta$ or $\chi$
	\begin{equation}\label{Z deri}
	\pa_\beta^m \pa_\chi^n Z_\beta(r,x) \qquad (m,n\in \mathbb{Z}_+)
	\end{equation}
	satisfy the null-vector equations \eqref{Lawler-Zhan PDE}. 
	The partition functions of the form \eqref{Z deri} can be constructed from \eqref{Z_beta} by taking proper non-degenerate limits. 
\end{rmk*}

\begin{eg*} As a consequence of Proposition~\ref{screen null}, it is easy to observe that $Z_\infty$ in \eqref{Z deg} solves following ordinary differential equation:
\begin{equation} \label{null vec deg}
	0 = \frac{\kappa}{2} Z''_\infty + \cot\Big( \frac{x}{2}\Big) Z'_\infty + \Big( \frac{3}{\kappa}-\frac{1}{2} \Big)\cot\Big( \frac{x}{2}\Big)' Z_\infty -\frac{1}{2\kappa} Z_\infty.
\end{equation}
The ordinary differential equation \eqref{null vec deg} was introduced by Zhan as a commutation relation for the radial SLE($\kappa,\Lambda$) process, see \cite[Section 4.5]{MR3334276}. 
In particular for the values of $\kappa$ such that $4/\kappa$ is a positive integer, the expression \eqref{Z deg} is simplified and expressed in terms of trigonometric functions, see e.g., \cite[Section 15.4]{olver2010nist}. See for instance Table~\ref{table partition deg} below, where we write $\sin_2(x):=\sin(x/2)$ and $\cot_2(x)=\cot(x/2)$ to lighten notations.
\end{eg*}

\begin{table}[h!] \renewcommand{\arraystretch}{2}
	\caption{ Partition functions $Z_\infty$ } \label{table partition deg}
	\begin{center}
		\begin{tabular}{c|c}
			\toprule 
			$\kappa$ & $Z_\infty(\cdot)$ \\
			\midrule
			$4$ &$\sin_2^{-1/2}$ \\ 
			$2$ &$\sin_2^{-1} \cot_2$ \\
			$4/3$ & $\sin_2^{-3/2} \Big(3\cot_2^2-2\cot'_2+\dfrac{1}{3} \Big) $ \\
			$1$ & $\sin_2^{-2} \Big(4\cot_2^3-6\cot_2\cot'_2+\cot''_2+\cot_2 \Big) $ \\
			\bottomrule
		\end{tabular}
	\end{center}
\end{table}

\ms
\begin{eg*}[Some particular solutions] Observe that if $4/\kappa$ is a positive integer, $Z_\beta^\sharp(p,q,\cdot)$ is a well-defined meromorphic function in $\C$ having poles only at $p$ and $q$. Therefore one can choose sufficiently small circle around $p$ as an integration contour in \eqref{scr Z}. In this case, it is easy to calculate $Z\equiv Z_\beta$ by residue calculus. We present some particular solutions (up to a multiplicative constant) in Table~\ref{table partition} for ER boundary condition with $\beta=0$.
\end{eg*}

\begin{table}[h!] \renewcommand{\arraystretch}{2}
	\caption{Partition functions $Z$} \label{table partition}
	\begin{center}
		\begin{tabular}{c|c}
			\toprule 
			$\kappa$ & $Z(r,\cdot)$ \\
			\midrule
			$4$ & $\Theta^{-1/2}$ \\ 
			$2$ & $\Theta^{-1} H$ \\
			$4/3$ & $\Theta^{-3/2} \Big(3H^2-2H'+4\dfrac{\zeta_r(\pi)}{\pi} \Big) $\\
			$1$ & $\Theta^{-2} \Big(4H^3-6HH'+H''+12 \dfrac{\zeta_r(\pi)}{\pi} H \Big) $ \\
			\bottomrule
		\end{tabular}
	\end{center}
\end{table}

Now we present the proof of Theorem~\ref{main_SLE Z}. Recall that
$$\Psi_\beta \equiv \Psi_\beta(\xi,q):= C(\kappa) \oint_{\mathcal{P}(\xi,q)} \Psi_\beta^\sharp(\xi,q,\zeta)\, d\zeta. $$
By definition, $\E\Psi=Z$ and $\Lambda=\kappa (\log Z)'$. 

\begin{proof}[Proof of Theorem~\ref{main_SLE Z}]
	By Proposition~\ref{screen null}, it remains to prove the second assertion. For any string $\XX$ of fields in the OPE family $\FF_{\bfs\beta}$, let 
	$$R_\xi=\wh{\E} \XX:= \frac{\E\Psi_\beta(\xi,q) \XX}{ \E\Psi_\beta(\xi,q) }.$$ 
	Note that all we need to show is the following version of BPZ-Cardy equation
	\begin{equation} \label{BPZ_C_scr}
		\frac{1	}{a^2}\pa_{\xi}^2R_\xi+ \Lambda \pa_{\xi}R_\xi
		= \LL_{v_{\xi}} R_\xi +\pa_r R_\xi.
	\end{equation}
	By Theorem~\ref{BPZ_ER}, we have
	\begin{align}
		\begin{split}
			\label{scr BPZ}
			\frac{1	}{a^2	}\pa_{\xi}^2 \E[ \Psi_\beta^\sharp(\xi,q,\zeta) \XX] &
			=\LL_{v_{\xi} }(q, S_\XX )\E [\Psi_\beta^\sharp(\xi,q,\zeta) \XX]+\LL_{v_{\xi} }(\zeta )\E [\Psi_\beta^\sharp(\xi,q,\zeta) \XX]
			\\
			&+ (\pa_r-C(r) ) \E[\Psi_\beta^\sharp(\xi,q,\zeta) \XX],
		\end{split}
	\end{align}
	where $C(r)$ is given as \eqref{null vec const1}. Note that since $\Psi_\beta^\sharp$ is a $(1,0)$-differential with respect to $\zeta$, we have
	\begin{equation*}
		\LL_{v_{\xi} }(\zeta )\E \Big[\Psi_\beta^\sharp(\xi,q,\zeta) \XX \Big]=\pa_\zeta \Big[		v_{\xi}(\zeta) \E [\Psi_\beta^\sharp(\xi,q,\zeta) \XX]\Big].
	\end{equation*}
	Therefore by \eqref{boundary Poch}, 
	$$ \oint_{\mathcal{P}(p,q)} \LL_{v_{p} }(\zeta )\E [\Psi_\beta^\sharp(p,q,\zeta) \XX]\, d\zeta=0.$$
	Integrating \eqref{scr BPZ} with respect to $\zeta$ along $\mathcal{P}(p,q)$, we obtain
	\begin{equation}\label{gen screen}
		\frac{1	}{a^2} \pa_{\xi}^2\Big(Z_\beta R_\xi \Big)=\LL_{v_{\xi}} \Big(Z_\beta R_\xi\Big)+ (\pa_r-C(r) ) \Big(Z_\beta R_\xi \Big).
	\end{equation}
	Combining \eqref{gen screen} with Proposition~\ref{screen null} we conclude \eqref{BPZ_C_scr}, which completes the proof.
\end{proof}

\begin{rmk*}
	For $\kappa=4$, let $\gamma$ be a contour around $q$ not encircling $p$ and write
	$$
	\Psi=\OO[a\cdot p-a\cdot q],\qquad \Psi_\gamma=\int_{\gamma} \OO[a\cdot p+a\cdot q-2a\cdot \zeta] \, d\zeta. 
	$$
	Then one can easily see that $\E \Psi=\E \Psi_\gamma=\Theta_\chi^{-\frac12}$ up to a multiplicative constant.  
	Thus the insertion of such one-leg operators produces martingale-observables for same SLE$(4,\Lambda)$, where $\Lambda=-H_\chi$. 
	However, we emphasize that as operators $\XX \to \wh{\XX}$, they are not the same in general. 

	For instance, let us assume that $\gamma$ encircles the node $z$. Then for ER boundary condition, we have 
	$$
	\frac{\E J(z) \Psi(p,q) }{\E \Psi(p,q)}=X(z):=w_z'\Big(H(w_p-w_z)-H(w_q-w_z)\Big),
	$$
	whereas 
	$$
	\frac{\E J(z) \Psi_\gamma(p,q) }{\E \Psi_\gamma(p,q)}=X(z)+Y(z), \qquad Y(z):=\frac{w_z' \,\Theta'(0) \Theta(w_p-w_q)}{ \Theta(w_p-w_z) \Theta(w_q-w_z) }.
	$$
	Combining Theorems~\ref{main} and ~\ref{main_SLE Z}, the non-random field $Y(z)$ is also a martingale-observable for SLE$(4,\Lambda)$.
\end{rmk*}

\section{Examples of martingale-observables} \label{Section_Examples}

This section is devoted to applying the theory in previous sections to construct martingale-observables associated with the scaling limits of lattice models. As a consequence, some examples are indicated as well.

\subsection{Bosonic observables} Since Schramm and Sheffield discovered the relation between level lines of discrete GFF and SLE$(4)$ type curves \cite{MR2486487}, it has played an important role in the study of random conformal geometries, see e.g., \cite{MR2525778}. 
A key ingredient for the coupling of GFF and SLE in the continumm level is the associated bosonic observables. In this spirit, Izyurov and Kyt\"{o}l\"{a} proposed a general framework to develop Schramm-Sheffield's coupling relations in various conformal setups including doubly connected domains with Dirichlet or Dirichlet-Neumann mixed boundary conditions with several force points \cite{MR3010393}. See also \cite{katori2020gaussian,MR3825935} and references therein for further examples of GFF/SLE couplings in different conformal setups.

For $\kappa=4$, let us consider the one-leg operator $\Psi_{ \beta }$ of the form
\begin{equation}\label{Psi k=4}
\Psi_{ \beta }(p,q):=\OO_{\bfs\beta}[ a\cdot p-a \cdot q ],
\end{equation} 
where $\bfs\beta$ is given as \eqref{tau beta}.
For the chordal case that $p,q\in \R$, we have 
\begin{equation}\label{k=4 Z}
Z_{\beta}(r,x)=\Theta(r,x)^{-\frac12}\exp\Big( -\frac{ (x-\mu)^2 }{ 8(r+\chi) } \Big), \qquad \mu=2\beta \pi/a.
\end{equation}
Here we write $p=x$ and $q=0$.
\begin{rmk*}
For the Dirichlet boundary condition (i.e., $\chi=0$), such GFFs correspond to those studied in \cite{MR3010393,MR2651436}. The partition functions \eqref{k=4 Z} are also presented in \cite[Section 8.1]{MR3334276} as a specific solutions for the null-vector equation with $\kappa=4$. 
Here one may find a simple geometric interpretation of $\beta \in \R$ since the bosonic field $\wh{\Phi}_{ \bfs\beta }$ has a piecewise Dirichlet boundary condition: 
\begin{equation} \label{non-zero Diri b.c.}
\E \wh{\Phi}_{ \bfs\beta }(z)= 
\begin{cases}
\hspace{1.5em}-\lambda &\text{if}\quad z \in (0,p)\cup (q,2\pi),
\\
\hspace{1.5em}+\lambda &\text{if}\quad z \in (p,q),
\\
(1-\mu/\pi)\lambda &\text{if}\quad z \in \R_r,
\end{cases}
\end{equation}
where $\lambda=a\pi=\pi/\sqrt{2}$.
\end{rmk*}
For SLE($4,\Lambda$) processes associated with partition functions of the form \eqref{k=4 Z}, we obtain their inner circle hitting probabilities using basic properties of Brownian motions. 
\begin{prop}
Let $\eta$ be the trace of $\SLE(4,\Lambda)$ whose partition function is given as \eqref{k=4 Z}. Then the probability $\P_{\chi}^\mu(x)$ that $\eta$ hits the inner boundary component is given by 
\begin{align} \label{Prob:hitting}
\begin{split}
\P_{\chi}^\mu(x)&= \frac{1}{4\pi} \sqrt{ \frac{r+\chi}{\chi} } e^{\frac{(x-\mu)^2}{8(r+\chi) } }
\int_{0}^{2\pi} e^{ -\frac{(s-\mu)^2}{\sqrt{8\chi}}}\Big[ \Theta_I\Big( \frac{r}{2}, \frac{x-s}{2}+\pi \Big)-\Theta_I\Big( \frac{r}{2}, \frac{x+s}{2}+\pi \Big) \Big] \, ds.
\end{split}
\end{align} 
\end{prop}
\begin{proof}
By definition and \eqref{k=4 Z}, the driving process $\xi_t$ of $\SLE(4,\Lambda)$ is given by
$$
d\xi_t= 2\,dB_t-\Big( H(r-t,\xi_t-q_t )+\frac{\xi_t-q_t-\mu}{r+\chi-t} \Big)\, dt, \qquad q_t:=\tilde{g}_t(q).
$$
We denote by $X_t:=\xi_t-q_t$ the angle difference process. Then $\P_{\chi}^\mu(x)$ is the probability that $X_t$ stays in the interval $[0,2\pi]$ up to time $r$. 

On the other hand, due to Loewner's equation, $X_t$ satisfies the  following SDE:
$$
dX_t:=2\,dB_t-\frac{X_t-\mu}{r+\chi-t} \, dt, \qquad X_0=x.
$$
Therefore the process $X_t$ is identified to a Brownian bridge starting from $x$, which reaches $\mu$ at time $t=r+\chi$, see e.g., \cite{MR2001996}. Thus we obtain 
\begin{align*}
\P_{\chi}^\mu(x)&=\Pr \Big\{ 0 < \min_{0 \le s \le r} B(s) < \max_{0 \le s \le r} B(s) < \pi \Big| B(0)=\frac{x}{2}, B(r+\chi)=\frac{\mu}{2} \Big\}
\\
&=\sqrt{2\pi(r+\chi)}\exp\Big( \frac{(\mu/2-x/2)^2}{2(r+\chi)} \Big) \int_{0}^{\pi} \frac{ 1 }{\sqrt{2\pi\chi}} \exp\Big( -\frac{(z-\mu/2)^2}{2\chi} \Big)\, d\sigma(z),
\end{align*}
where $d\sigma$ is given as 
\begin{align*}
d\sigma(z)&=\Pr \Big\{ 0 < \min_{0 \le s \le r} B(s) < \max_{0 \le s \le r} B(s) < \pi, B(r)\in dz \Big| B(0)=\frac{x}{2} \Big\}
\\
&=\frac{1}{\sqrt{2\pi r}} \sum_{n \in \mathbb{Z}} \Big[ \exp\Big(-\frac{(x-2z+4n\pi)^2}{2r}\Big)-\exp\Big(-\frac{(x+2z+4n\pi)^2}{2r}\Big) \Big]\,dz.
\end{align*}
Now proposition follows from \eqref{Theta_I Gauss}.
\end{proof}

\begin{rmk*}
Note that for ER boundary condition ($\chi=\infty$), the angle difference process $X_t$ is simply a Brownian motion with speed $2$. In this case \eqref{Prob:hitting} is given as 
$$
\P_{\infty}^\mu(x)=\frac{1}{4\pi} \int_{0}^{2\pi} \Big[\Theta_I\Big( \frac{r}{2}, \frac{x-s}{2}+\pi \Big)-\Theta_I\Big( \frac{r}{2}, \frac{x+s}{2}+\pi \Big) \Big]\, ds,
$$
which corresponds to \cite[Theorem 7.45]{morters2010brownian}. 
On the other hand, for Dirichlet boundary condition ($\chi=0$), it follows from the Gaussian approximation of the Dirac delta measure 
$$
\frac{1}{\sqrt{8 \chi \pi}} e^{ -\frac{(s-\mu)^2}{\sqrt{8\chi}}} \to \delta_\mu(s),
$$
that 
$$
\P_{0}^\mu(x)=\mathbf{1}_{(0,2\pi)}(\mu) \cdot \sqrt{ \frac{r}{2\pi} }\, e^{ \frac{(x-\mu)^2}{8r} } \Big[ \Theta_I\Big( \frac{r}{2}, \frac{x-\mu}{2}+\pi \Big)-\Theta_I\Big( \frac{r}{2}, \frac{x+\mu}{2}+\pi \Big) \Big].
$$
Observe here that the requirement $\mu \in (0,2\pi)$ for the positive probability is equivalent to the fact that the boundary value of $\wh{\Phi}_{ \bfs{\beta} }$ on the inner circle is between two heights on the outer circle, see \eqref{non-zero Diri b.c.}.

	\begin{figure}[h!]
	\begin{center}
	\begin{subfigure}[h]{0.48\textwidth}
		\centering 
			\includegraphics[width=0.8\textwidth]{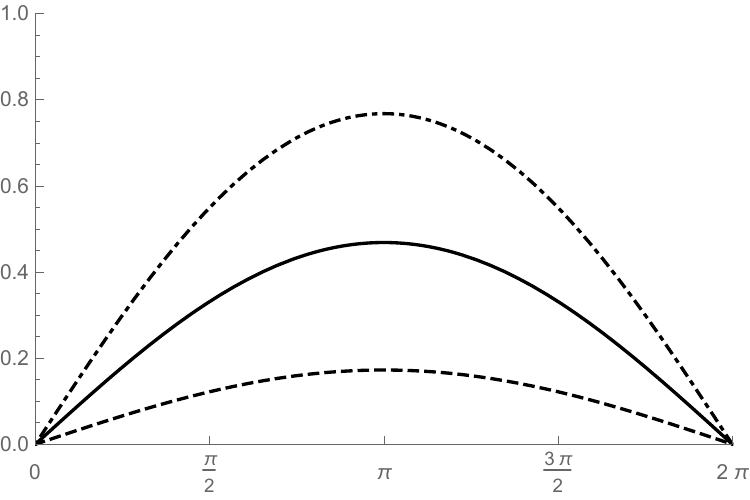}
		\caption{$\P_{\infty}^\mu(x)$}
	\end{subfigure}
 	\begin{subfigure}[h]{0.48\textwidth}
 		\centering 
 	\includegraphics[width=0.8\textwidth]{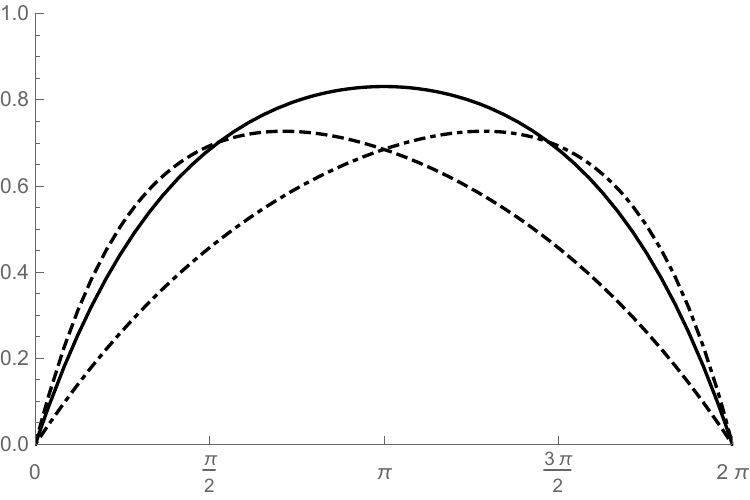}
 	\caption{$\P_{0}^\mu(x)$}
 \end{subfigure}
	\end{center}
	\caption{ (A): The plot displays $\P_{\infty}^\mu(x)$ with modulus $r=1$ (dot-dashed line), $r=2$ (full line) and $r=4$ (dashed line). (B): The plot displays $\P_{0}^\mu(x)$ with $r=2$ for $\mu=\pi/2$ (dot-dashed line), $\mu=\pi$ (full line) and $\mu=3\pi/2$ (dashed line).}
\end{figure}
\end{rmk*}

\begin{eg*}
For $\chi=0$, set 
$$
\OO^{1,n}:=\OO[a \cdot q- a \cdot q_n], \qquad \OO^{2,n}:=\OO[a \cdot q- a \cdot q_n-2\bfs{\beta}],
$$
where $q_n:=q-4n\pi$. Then one can check that after renormalization, the vertex fields $\wh{\OO}^{1,n}$, $\wh{\OO}^{2,n}$ have conformal dimension $0$ and 
\begin{align}\label{HBB MO}
\begin{split}
M^{1,n}:=\E \wh{\OO}^{1,n}&= \exp\Big( \frac{(x-\mu)^2}{8r} \Big) \exp\Big( -\frac{ (x-4\pi n-\mu)^2 }{ 8r } \Big), 
\\
M^{2,n}:=\E \wh{\OO}^{2,n}&= \exp\Big( \frac{(x-\mu)^2}{8r} \Big) \exp\Big( -\frac{ (x-4\pi n+\mu)^2 }{ 8r } \Big). 
\end{split}
\end{align}
Here notice that by \eqref{Theta_I Gauss}, we have 
$$
\P_{0}^\mu(x)=\sum_{n \in \mathbb{Z}} ( M^{1,n}-M^{2,n} ).
$$
Such martingale-observables \eqref{HBB MO} are utilized in \cite{MR2651436} to calculate $\P_{0}^\mu(x)$. 
\end{eg*}

\begin{rmk*}
For the weight function $\omega(n):= (-1)^{-\frac{n}{2}}$, one can easily observe from \eqref{Theta_I Gauss} that (up to a multiplicative constant)
$$
Z_\beta^\omega(r,x)=\Theta(r,x)^{-\frac12} \Theta_I(2r+2\chi,x-\mu+\pi).
$$
In particular, for $\chi=0$ and $\beta=a/2$, the SLE($4,\Lambda$) associated with this partition function is coupled with GFF having specific height gap on the outer boundary component which obeys Neumann condition on the inner one, see \cite{MR3010393,MR2651436}. 
\end{rmk*}

\ms
\begin{eg*}[Crossing case] Let us consider the case that $p \in \R$, $q+ir \in \R_r$. Then the partition function associated to the one-leg operator \eqref{Psi k=4} is given as 
$$
Z_\beta(r,x) = \Theta_I(r,x)^{-\frac12} \exp\Big( -\frac{ (x-\mu)^2 }{ 8(r+\chi) } \Big), \qquad \mu=2\beta \pi/a.
$$
In particular, if $\chi=\beta=0$, each partition function $Z^n_\beta(x):=Z_\beta(x+2n\pi)$ corresponds to the crossing type SLE$(4,\Lambda)$ connecting two marked points with prescribed winding number, see \cite[Section 6.3]{MR3334276}. Moreover the $2\pi$-periodic partition function 
\begin{equation}
Z_\beta^\omega(r,x):=\Theta_I(r,x)^{-\frac12 } \Theta_I(2r,x+\pi), \qquad \omega(n)\equiv 1
\end{equation}
agrees with Zhan's Feynman-Kac expression of the crossing type annulus partition function, see \cite[Section 8.1]{MR3334276}.
\end{eg*}

\subsection{LERW observables} In this subsection we present a way of constructing martingale-observables for variants of continuum loop-erased random walk (LERW). In the physics literature, a version of CFT constructed from sympletic fermions has been proposed for LERW, see e.g., \cite{majumdar1992exact,bauer2008lerw,caracciolo2004fermionic} and references therein.

In a doubly connected domain, one may consider variants of LERW by imposing boundary conditions on the inner boundary components. 
For instance, let us consider a simple random walk on the lattice approximation of the annulus started from an interior vertex near the initial boundary point conditioned to hit a boundary vertex near the target point. 
We assume that the random walk continues when it hits the inner circle, and stops when it visits the outer one. 
Then after the loop-erasing procedure, the resulting curve is the LERW connecting two marked points on the outer boundary component, obeying the Neumann condition on the inner one. 

For generic LERW, it is well known that the associated partition function $Z$ is expressed as 
\begin{equation}\label{Z LERW Green}
Z(p,q)=\frac{\pa^2}{\pa n_\zeta \pa n_z}\Big|_{\zeta=p, z=q} G(\zeta,z),
\end{equation}
see e.g., \cite[Section 4.5]{zhan2004random}, \cite[Section 9]{dubedat2007commutation} and \cite{makarov2010off}. Here the Green's function $G$ obeys the boundary condition of LERW. 

We now present a way of constructing LERW partition functions from the methods introduced in the previous sections. Then by Theorems~\ref{main_period} and ~\ref{main_SLE Z}, the way of constructing the associated martingale-observables also follows. 

Recall that for $4/\kappa \in \mathbb{Z}_+$, one can choose the integration contour $\gamma$ as a circle encircles $p$ or/and $q$ in the expression of $Z^\omega_\beta$. 
With such a choice of $\gamma$, by residue calculus, we obtain following two-parameter family of solutions $Z_\beta(r,x) \equiv Z_{\beta,\chi}(r,x)$ to null-vector equation \eqref{Lawler Zhan null} for $\kappa=2$:
\begin{equation} \label{Z k=2 nw}
	Z_\beta(r,x)=\Theta(r,x)^{-1} \exp\Big( -\frac{(x-2\beta \pi)^2}{4(r+\chi)} \Big) \Big( H(r,x)+\dfrac{x-2\beta \pi}{r+\chi} \Big) .
\end{equation}
Using the weight functions $\omega(n)= 1, (-1)^n$ and \eqref{Theta_I Gauss}, we obtain following family of solutions having $2\pi$-(anti) periodicity:
\begin{equation} \label{Z k=2 gen}
Z_\beta^\omega(r,x)=
\begin{cases}
\dfrac{ \Theta(r+\chi, x+\pi-2\beta \pi ) }{ \Theta(r,x)} \Big( H(r,x)-H(r+\chi,x+\pi-2\beta \pi) \Big), &\text{if}\quad \omega(n)= 1,
\vspace{0.5em}
\\
\dfrac{ \Theta_I(r+\chi, x+\pi-2\beta \pi ) }{ \Theta(r,x)} \Big( H(r,x)-H_I(r+\chi,x+\pi-2\beta \pi) \Big) &\text{if}\quad \omega(n)= (-1)^n.
\end{cases}
\end{equation} 
We emphasize here that the $2\pi$-periodic partition functions $Z_\beta^\omega$ are positive for all $\beta$. Let us also point out that such partition functions for $\beta \in \{ 0, b\}$ with $\chi=0$ correspond to those found by Zhan in \cite[Section 8.2]{MR3334276}.

\subsubsection{LERW partition functions}
Let us now focus on the case with $\chi=0$. By \eqref{2ir rot period}, such $2\pi$-periodic partition functions further satisfy rotational periodicities with respect to $2ir$;
\begin{equation} 
Z^\omega_\beta(r,x+2ir)= \exp\Big( (2\beta-1) \pi i \Big) Z^\omega_\beta(r,x).
\end{equation}

\medskip 

\noindent $\bullet$ \emph{LERW with Dirichlet/Riemann-Hilbert mixed boundary conditions.}
Let $\omega(n)=1$. Then by \eqref{Z k=2 gen}, we have
\begin{equation} \label{Z LERW RH}
Z_\beta^\omega(r,x)=\dfrac{ \Theta(r,x+\pi-2\beta \pi ) }{ \Theta(r,x)} \Big( H(r,x)-H(r,x+\pi-2\beta \pi) \Big).
\end{equation}
The partition function \eqref{Z LERW RH} describes the continuous LERW with \emph{Riemann-Hilbert} boundary condition, i.e., the requirement that the derivative along the oblique direction (indexed by $\beta$) vanishes. This follows from \eqref{Z LERW Green} and the representation of the Green's function satisfying Riemann-Hilbert boundary condition on the inner boundary component, see Appendix~\ref{Appendix_RH}.

It is well known that Riemann-Hilbert b.c. interpolates Neumann and ER conditions, see e.g., \cite{MR3706737}.
For $\beta=0$, we have an alternative expression 
\begin{equation} \label{Z LERW N}
Z(r,x)=H'(2r,x)-H_I'(2r,x).
\end{equation}
On the other hand, when $\beta \to b$, after normalization, we have 
\begin{equation} \label{Z LERW ER}
Z(r,x)=H'(r,x). 
\end{equation}
These partition functions \eqref{Z LERW N} and \eqref{Z LERW ER} correspond to the partition functions of LERW with Neumann and ER b.c., respectively. 

\medskip 

\noindent $\bullet$ \emph{LERW with Dirichlet/Dirichlet boundary condition.} 
Let $\beta=b$ and $\omega(n)=n$, i.e., 
\begin{equation} \label{LERW_D Psi}
\Psi_\beta^\omega(x):=\sum_{n \in \mathbb{Z}} n \Psi_{ b }(x+2n\pi).
\end{equation}
Then it is straightforward to check that 
\begin{equation}\label{Z LERW D}
Z_\beta^\omega(r,x)=\ti{H}'(r,x)=H'(r,x)+\frac{1}{r}.
\end{equation}
This corresponds to the LERW partition function with Dirichlet boundary condition, which appears in \cite{zhan2008scaling}.

Let us pause here to briefly explain the derivation of \eqref{Z LERW D}.
Note that by \eqref{Z k=2 nw}, 
$$
\E \Psi_{ b }(r,x)= \Theta(r,x)^{-1} \exp\Big( -\frac{(x+ \pi)^2}{4r} \Big)
\Big( H(r,x)+\dfrac{x+ \pi}{r} \Big).
$$
Therefore we have
\begin{align*}
Z_\beta^\omega(r,x)&=\frac{H(r,x)}{\Theta(r,x)} \sum_{n \in \mathbb{Z}} (-1)^n \, n \, \exp\Big(-\frac{(x+\pi+2n\pi)^2}{ 4r } \Big)
\\
&+\frac{1}{\Theta(r,x)} \sum_{n \in \mathbb{Z}} (-1)^n \, n \, \Big(\frac{x+\pi+2n\pi}{r}\Big)\exp\Big(-\frac{(x+\pi+2n\pi)^2}{ 4r } \Big).
\end{align*}
It follows from \eqref{Theta_I Gauss} that 
$$
\sum_{n \in \mathbb{Z}} (-1)^n \, n \, \exp\Big(-\frac{(x+\pi+2n\pi)^2}{ 4r } \Big)=\frac{r}{\pi}\sqrt{\frac{r}{\pi}} \Big( \Theta'(r,x)+\frac{x+\pi}{2r}\Theta(r,x) \Big).
$$
Combining this identity with \eqref{H Theta} and \eqref{heat eq. Theta}, we obtain (up to a multiplicative constant)
\begin{align*}
Z_\beta^\omega(r,x)&=H(r,x)\Big( \frac{\Theta'(r,x)}{\Theta(r,x)}+\frac{x+\pi}{2r} \Big)
-\Big( 2\frac{\Theta''(r,x)}{\Theta(r,x)}+\frac1r +\frac{x+\pi}{r}\frac{\Theta'(r,x)}{\Theta(r,x)} \Big)
\\
&=\Big(\frac12 H^2(r,x)+\frac{x+\pi}{2r}H(r,x)\Big)-\Big(H'(r,x)+\frac12 H^2(r,x)+\frac1r+\frac{x+\pi}{2r}H(r,x)\Big),
\end{align*}
which leads to \eqref{Z LERW D}. 

\begin{rmk*}
 In parallel, one can construct such partition functions for the crossing case that $p \in \R$ and $q \in \R_r$. In particular, for $\beta=b, \chi=0$ with the choice of weight function $\omega(n)=n$, we obtain 
	$Z_\beta^\omega(r,x)=H_I'(r,x)+1/r$, 
	which corresponds to Zhan's Feynmann-Kac solution, see \cite[Section 8.2]{MR3334276}. 
\end{rmk*}

\noindent $\bullet$ \emph{LERW aiming at a side arc.} Due to the special feature $a=1/a$, if we place the charge $a$ at the ``target point" $q$ as usual, the node $q$ can also be utilized as a screening variable.
This fact allows us to describe the one-leg operator for LERW aiming at a marked side arc. 

For a subset $I$ of boundary components ($I \cap p = \emptyset$), set 
\begin{equation}
\Psi(p,I):=\frac{1}{|I|}\int_{I} \Psi_b^\omega (p,q) \, dq, \qquad \omega(n)=n. 
\end{equation}
Then the associated partition function $Z(p,I)=\E \Psi(p,I)$ is given by
\begin{equation} \label{LERW}
Z(p,I):=\frac{1}{|I|}\int_{I} H'(r,p-q) \, dq+\frac1r. 
\end{equation}
This corresponds to the partition function of continuous LERW with a target $I$, see \cite{zhan2008scaling}. 

In particular, when $I=\mathbb{R}_r$, it describes LERW starting from $p$, which stops whenever it hits the inner boundary component. By the periodicity, the associated partition function is given by $Z=\frac1r$.
This can be realized as a partition function for standard SLE$(2)$, see \cite{zhan2008scaling,zhan2004stochastic}.

\subsubsection{Martingale-observables} Let $\Psi$ be the one-leg operator associated with variants of LERWs described above. We present some import martingale-observables. 

\medskip 

\noindent $\bullet$ \emph{Generalized Poisson kernel.} For simplicity we only consider the case that target set is a boundary point $q$. Set 
$\OO:=\OO[a\cdot z-a \cdot q]$. 
Let us define 
\begin{equation}
M(z):=\frac{1}{2i} \int^z_{\bar{z}} \wh{\E} \OO(\zeta)\, d\zeta, \qquad \wh{\E} \OO(z)=\frac{ w'(z) Z(r,w(z)-w(p)) }{ w'(q) Z(r,w(q)-w(p)) }. 
\end{equation}
The scalar field $M$ is called a generalized Poisson kernel. 
For the Dirichlet boundary condition, Zhan considered the following bounded martingale (cf. \eqref{Z LERW D})
\begin{equation}
M_t(z)=\frac{1}{w_t'(q)} \frac{\Im \, \tilde{H}(r-t,w_t(z)) }{\tilde{H}'(r-t,w_t(q) )}
\end{equation}
to show the convergence of LERW, see \cite{zhan2008scaling,zhan2004random}. 

\medskip 

\noindent $\bullet$ \emph{Ending point distribution of standard SLE(2).}
Set $I=[q_1+ir,q_2+ir] \subset \mathbb{R}_r$. Let us write 
$\Psi$ for the one-leg operator associated with the standard SLE(2). 
Set 
$$
\OO=\frac{1}{2\pi}\int_I \OO[a\cdot z-a \cdot q] \, dq.
$$ 
Then the martingale-observable 
$$
M(z)=\wh{\E} \OO = \frac{ \E \Psi \OO }{ \E \Psi }=\frac{1}{2\pi}\int_{q_1}^{q_2} \Big(r H_I'(r,p-q)+1\Big) \, dq
$$
yields the probability that standard SLE(2) ends at $I$, see \cite{zhan2006some}.

\subsection{Critical Ising observables} 

In this subsection we focus on the case $\kappa=3$. 
As in Subsection~\ref{Subsec_NVE msr}, let us denote by $\check{\bfs\beta}$ the (signed) uniform measure on $[ir-\pi,ir+\pi]$ with total mass $-2b$ and let $\bfs \beta= 2b \cdot q+\check{\bfs\beta}$.
Then we have 
$$
Z_{\bfs\beta}(r,x)=\E \Psi_{ \bfs{\beta} }[a\cdot p-a\cdot q]= \Theta(r,x)^{-1} \exp\Big( -\frac{x^2}{4r} \Big), \qquad x=p-q.
$$
The null-vector equation for $Z_{\bfs\beta}$ is given as 
\begin{equation} \label{NVE pre k3}
\pa_r Z_{ \bfs \beta }=\frac{3}{2} Z_{ \bfs \beta }''+H Z_{ \bfs \beta }' +\frac12 H' Z_{ \bfs \beta } +F_{ \bfs\beta } Z_{ \bfs \beta }, 
\end{equation}
where 
\begin{equation}
F_{\bfs \beta}(r,x)=-\Big( \frac{H'}{2}+\frac{H^2}{4} \Big)(r,x) -\frac{x}{2r}\,H(r,x)- \frac{x^2}{4r^2} . 
\end{equation}

Using the weight $\omega(n)=(-1)^n$, set 
\begin{equation}
\Psi(x):=\sum_{ n \in \mathbb{Z} } (-1)^n \Psi_{ \bfs{\beta} }[a\cdot (p+2n\pi)-a\cdot q].
\end{equation}
Then by \eqref{Theta_I Gauss}, we have
\begin{equation} \label{Z Ising}
Z(r,x):=\E \Psi(x)=\sum_{ n \in \mathbb{Z} } (-1)^n Z_{\bfs\beta}(r,x+2n\pi)=\sqrt{\frac{r}{\pi}} \, \frac{\Theta_I(r,x+\pi)}{\Theta(r,x)} .
\end{equation}
This corresponds to the partition function for the critical Ising interface studied by Izyurov \cite{MR3602850}. Therefore the insertion of the field $\Psi$ provides martingale-observables for the continuum limit of the interface curve. 

Since the interface curve satisfies reversibility, it follows from Izyurov's convergence result \cite{MR3602850} and Zhan's theory on commuting annulus SLE that the partition function \eqref{Z Ising} satisfies the null-vector equation \eqref{Lawler Zhan null} with $\kappa=3$. 
Indeed this also follows from the martingale property of parafermionic observable in the continuum limit. 
See \cite{izyurov2020multiple} for verification of BPZ type equations of the partition function in this manner. 

Here we present a field theoretical method to show that $Z$ satisfies \eqref{Lawler Zhan null}.

\begin{prop}
The partition function $Z$ in \eqref{Z Ising} satisfies the null-vector equation \eqref{Lawler Zhan null} with $\kappa=3$.
\end{prop}
\begin{proof}
By \eqref{NVE pre k3}, all we need to show is 
\begin{equation} \label{Z Ising NVE e}
\frac{1}{Z(r,x)}\sum_{n \in \mathbb{Z}} (-1)^n F_{\bfs\beta}(x+2n\pi)Z_{\bfs\beta}(r,x+2n\pi) 
\end{equation}
is a constant depending only on the modular parameter $r$.
By differentiating \eqref{Theta_I Gauss}, we have 
$$
\sum_{n\in \mathbb{Z}} \frac{x+2n\pi}{ 2r } \exp\Big( -\frac{(x+2n\pi)^2}{4r} \Big)= -\sqrt{\frac{r}{\pi}} \Theta_I'(r,x+\pi) 
$$
and 
$$
\sum_{n\in \mathbb{Z}} \frac{(x+2n\pi)^2}{ 4r^2 } \exp\Big( -\frac{(x+2n\pi)^2}{4r} \Big)= \sqrt{\frac{r}{\pi}} \Big(\Theta_I''(r,x+\pi)+\frac{1}{2r}\Theta_I(r,x+\pi)\Big). 
$$
Using these identities, we obtain that \eqref{Z Ising NVE e}  simplifies to 
$$
\frac12 H(r,x)H_I(r,x+\pi)-\Big( \frac{H_I'}{2}+\frac{H_I^2}{4} \Big)(r,x+\pi)-\Big( \frac{H_I'}{2}+\frac{H_I^2}{4} \Big)(r,x+\pi)-\frac{1}{2r}.
$$
Now the proposition follows from \eqref{add H}. 
\end{proof}

\begin{eg*}
Let $\OO_{ \bfs{\beta} }:=\OO_{ \bfs{\beta} }[ (2b-a)\cdot z-(2b-a)\cdot q ]$. 
Then the vertex observable $M(z):=\wh{\E} \OO_{ \bfs{\beta} }$ provides the local martingale
\begin{equation}
M_t(z)= \frac{ w_t'(z)^{\frac12} Z(r-t,w_t(z) ) }{ w_t'(q)^{\frac12} Z(r-t,w_t(q)) }.
\end{equation}
This corresponds to the parafermionic observable utilized in \cite{MR3602850}. 

\end{eg*}

Using a similar idea in previous sections, one can construct further partition functions $Z$ of the form
$$
\frac{\Theta(r,x+\pi-2\beta \pi)}{\Theta(r,x)}, \qquad \frac{\Theta_I(r,x+\pi-2\beta \pi)}{\Theta(r,x)},
$$
which satisfy the null-vector equation \eqref{Lawler Zhan null} and $2ir$-rotational periodicity.
The former has $2\pi$ periodicity, and the latter has $2\pi$-(anti) periodicity. Moreover, the method to construct the chordal type of partition functions in this section can be applied to the crossing case as well.

\subsection{Annulus SLE with one marked point}

One can use the implementation of the measure-valued charges to construct martingale-observables for annulus SLE with one marked point. A typical example of a such process is called the standard annulus SLE($\kappa$) \cite{zhan2004random} with driving process $\xi_t=\sqrt{\kappa}\,B_t$. 

For $k \in \{ 1,\ldots, n \}$, let $q_k:=q-\pi+2\pi \,\frac{k}{n}$, where $q \in \R_r$. We then consider 
\begin{equation}
\Psi_n(p)=\OO \Big[ a\cdot p- \sum_{ k=1 }^n \frac{a}{2n}\cdot q_k, - \sum_{ k=1 }^n \frac{a}{2n}\cdot q_k \Big].
\end{equation}
By \eqref{Lambda bulk}, it follows that the associated drift function $\Lambda_n$ is given by 
\begin{equation}
\Lambda_n(p)=-\frac{1}{2n} \sum_{ k=1 }^n H_\chi(r,p-q_k)+H_\chi(p-\bar{q}_k)
=-\frac{1}{n} \sum_{ k=1 }^n H_I(r,p-\Re \, q_k)+\frac{\Re \, q_k-p}{r+\chi}.
\end{equation}
We denote
\begin{equation}
\Psi(p):=\lim_{n \to \infty} \Psi_n(p)= \OO \Big[ a\cdot p- \frac{a}{2} \cdot \bfs{q}, - \frac{a}{2} \cdot \bfs{q} \Big], \qquad \Lambda:=\lim_{n \to \infty} \Lambda_n=\frac{\Re \,q-p}{r+\chi},
\end{equation}
where $\bfs{q}$ is the uniform measure on $[q-\pi,q+\pi]$ for $q \in \R_r$. Here we use the $2\pi$-periodicity of $H_I(r,\cdot)$.
By Theorem~\ref{main}, the insertion of the one-leg operator $\Psi(p)$ produces martingale-observables for annulus SLE driven by
\begin{equation}
d\xi_t=\sqrt{\kappa}\,dB_t+\frac{\Re\,q-\xi_t}{r+\chi-t}\,dt,\qquad \xi_0=p.
\end{equation} 
Notice that this corresponds to the law of the Brownian bridge. 
In particular, for the ER b.c. case when $\chi=\infty$ (thus $\Lambda \equiv 0$), it corresponds to the standard SLE$(\kappa)$. On the other hand, for the Dirichlet b.c. case when $\chi=0$, the SLE trace ends at the point $q$. 

\medskip 

\begin{eg*}[Level line of GFF]
In the identity chart of $\CC_r$, the associated bosonic observable $M(z):=\E \wh{\Phi}(z)$ is evaluated as 
\begin{equation}
M(z)=2a \arg \Theta_{\chi}(r,z-p)-a \int \arg \Big\{ \Theta_{\chi}(r,z-\bfs q )\Theta_{\chi}(r,z- \bar{ \bfs q} ) \Big\} \,d\bfs q.
\end{equation}
For the Dirichlet b.c., the boundary value of $M$ has discontinuity $2a \pi$ at $p+2n\pi$ and has linear growth with speed $a$ on $\R_r$. To be precise, $ M(z+2\pi)=M(z)+2\lambda$ and 
\begin{equation} 
M(z)= 
\begin{cases}
\hspace{1.5em}-\lambda &\text{if}\quad z \in (0,p),
\\
\hspace{1.5em}+\lambda &\text{if}\quad z \in (p,2\pi),
\\
\lambda \dfrac{\Re\,(z-q)}{\pi} &\text{if}\quad z \in \R_r,
\end{cases}
\end{equation}
where $\lambda=a\pi=\pi/\sqrt{2}$.
\end{eg*}

\appendix

\section{Representation of Green's function} \label{Appendix_RH}

In this appendix we present an analytic representation of Green's function in a doubly connected domain with various boundary conditions.

For a parameter $\beta \in (-1/2,1/2)$, we say a function $F$ defined in a domain $\bar{D}$ satisfies RH$_\beta$ boundary condition (b.c.) on $l \subset \pa D$ if 
\begin{equation} \label{RH condition}
\Big[ \cos(\beta \pi )\pa_n-\sin(\beta \pi) \pa_\tau \Big] F(\cdot)=0 \hspace{1em} \text{on} \hspace{0.5em} l \subset \pa D.
\end{equation}
Here $\pa_n$ is the (inwards) normal derivative and $\pa_\tau$ is the tangential one. In particular, the case $\beta=0$ corresponds to the Neumann b.c. 

Let us define 
\begin{equation}\label{f_beta d}
f_\beta(r,z):=\frac{\Theta'(r,0)}{\Theta(r,\pi-2\beta \pi)} \frac{\Theta(r,z+\pi-2\beta \pi)}{\Theta(r,z)}. 
\end{equation}
By the quasi-periodicities \eqref{theta per} of the theta function, we have 
\begin{equation} \label{f_beta per}
f_\beta(r,z+2\pi)=f_\beta(r,z), \qquad f_\beta(r,z+2ir)=e^{i(2\beta-1)\pi}f_\beta(r,z).
\end{equation}
Since $\Theta(r,\cdot)$ is an odd function, 
\begin{equation} \label{f_beta asym d}
f_\beta(r,z)=-f_{-\beta}(r,-z).
\end{equation}
On the other hand, $f_\beta$ satisfies 
\begin{equation} \label{f_beta res}
f_\beta(r,z)=\frac{1}{z}+\frac{H(r,\pi-2\beta \pi)}{2}+O(z),\qquad \text{as }z\to 0.
\end{equation}
Thus one can observe that $f_0$ has an alternative expression 
\begin{equation} \label{f_0 d}
f_0(r,z)= \frac12 (H(2r,z)-H_I(2r,z)).
\end{equation}
It is also easy to show that $f_\beta(r,\cdot)$ is a conformal map from the cylinder $\mathcal{C}_r$ to the upper half-plane minus a slit whose argument is $(1/2-\beta)\pi$. 

For given $w \in \CC_r$, a meromorphic function 
\begin{equation}
g_\beta(z):=f_\beta(r,z-\bar{w})-f_\beta(r,z-w)
\end{equation}
satisfies following properties:
\begin{itemize}
	\item $g_\beta(z+2\pi)=g_\beta(z)$, $g_\beta(r,z+2ir)=e^{i(2\beta-1)\pi}g_\beta(r,z)$;
	\item $g_\beta(r,\cdot)$ has simple poles at $w$ (resp., $\bar{w}$) with residue $-1$ (resp., $1$);
	\item $g_\beta(r,\cdot)$ maps $\R$ to $i \R$; 
	\item $g_\beta(r,\cdot)$ maps $\R_r$ to a line passing through 0 with angle $\beta$ with $\R$.
\end{itemize}
All of these properties are immediate consequences of \eqref{f_beta per}, \eqref{f_beta asym d} and \eqref{f_beta res}. For instance, the last property follows from that for $x\in \R$,
\begin{align*}
g_\beta(x+ir)&=\overline{f_\beta(r,x-ir-w)}-f_\beta(r,x+ir-w)
\\
&=e^{i(2\beta-1)\pi}\overline{f_\beta(r,x+ir-w)}-f_\beta(r,x+ir-w)
\\
&=-e^{i\beta \pi} \Big(e^{i\beta \pi} \overline{f_\beta(r,x+ir-w)}+ e^{-i\beta \pi} f_\beta(r,x+ir-w) \Big) \in e^{i\beta \pi} \R.
\end{align*}

We denote by $F_\beta(r,\cdot)$ a primitive of $f_\beta(r,\cdot)$, i.e.,
$\pa_z F_\beta(r,z)=f_\beta(r,z)$. 
To our knowledge, for general $\beta$, there is no known expression for $F_\beta(r,z)$ in terms of well-known special functions. On the other hand, by \eqref{f_0 d}, we have 
\begin{equation} \label{F0}
F_0(r,z)=\log \Big(\frac{\Theta(2r,z)}{\Theta_I(2r,z)}\Big)
\end{equation}
up to an additive constant. 

We now present Green's function $G_\beta$ in $\CC_r$ with zero Dirichlet b.c. on $\R$ which satisfies RH$_\beta$ condition on $\R_r$. The associated (non-symmetric) stochastic process is called \emph{obliquely reflected Brownian motion} (ORBM), see \cite{MR3706737} and references therein. We remark that ORBM gives a geometric interpolation between reflected Brownian motion ($\beta=0$) and ERBM ($\beta \to \pm 1/2$). 

We claim that Green's function $G_\beta$ in $\CC_r$ with zero Dirichlet b.c. on $\R$ which satisfies RH$_\beta$ condition on $\R_r$ is expressed as 
\begin{equation} \label{G RH}
G_\beta(z_1,z_2)=\Re\,[ F_\beta(r,z_1-\bar{z}_2)-F_\beta(r,z_1-z_2) ].
\end{equation}
In particular, by \eqref{F0}, Green's function $G_0$ with Dirichlet-Neumann mixed boundary condition is given by 
\begin{equation} \label{Green}
G_0(z_1,z_2)=\log\Big| \frac{\Theta/\Theta_I(2r,z_1-\bar{z}_2)}{\Theta/\Theta_I(2r,z_1-z_2)} \Big|.
\end{equation}
To show \eqref{G RH}, it suffices to check the followings: 
\begin{itemize}
	\item $G_\beta(z_1,z_2)+\log|z_1-z_2|$ is harmonic function in both variables;
	\item $G_\beta(z_1,z_2)=0$ if $z_1$ or $z_2$ is on $\R$;
	\item $G_\beta(\cdot,z_2)$ satisfies RH$_\beta$ condition on $\R_r$.
\end{itemize}
All of these requirements easily follow from the properties presented above. We remark that $G_\beta$ satisfies the asymmetric relation $G_\beta(z_1,z_2)=G_{-\beta}(z_2,z_1)$. 

\subsection*{Acknowledgements} We wish to express our gratitude to Dapeng Zhan for valuable comments and stimulating discussions.

\bibliographystyle{abbrv}
\bibliography{BKT}
\end{document}